\documentclass[a4paper, 11pt]{article}

\usepackage[utf8]{inputenc}
\usepackage[OT1]{fontenc}
\usepackage{lmodern}   
\usepackage[english]{babel}
\usepackage{cite, hyphenat, hyperref}
\usepackage{amscd, amsfonts, amsmath}
\usepackage{amssymb, amsthm}
\usepackage{bbm, mathrsfs}
\usepackage{enumerate}
\usepackage{makeidx}
\usepackage{calc}
\usepackage[shortlabels]{enumitem}
\usepackage{nccfoots}
\usepackage{appendix}
\usepackage[a4paper, top=2.5cm, bottom=2.5cm, left=3cm, right=3cm]{geometry}
\allowdisplaybreaks
\numberwithin{equation}{section}
\theoremstyle{plain}

\newtheorem{Lemma}{Lemma}[section]
\newtheorem{Proposition}[Lemma]{Proposition}
\newtheorem{Theorem}[Lemma]{Theorem}
\newtheorem{Corollary}[Lemma]{Corollary}
\theoremstyle{definition}
\newtheorem{Definition}[Lemma]{Definition}
\newtheorem{Example}[Lemma]{Example}

\newtheorem{Remark}[Lemma]{Remark}

\def\oe{\"{o}}

\def\B{{\rm B}}
\def\R{{\mathbb{R}}}
\def\N{{\mathbb{N}}}

\DeclareMathOperator*{\esssup}{ess\,\sup}
\begin{document}
%-------------------------------------------------------------------
% Title
%-------------------------------------------------------------------
\title{\textbf{Stability, uniqueness and existence of solutions to McKean-Vlasov SDEs in arbitrary moments}}
\author{Alexander Kalinin\footnote{Department of Mathematics, LMU Munich, Germany. {\tt alex.kalinin@mail.de}} \and Thilo Meyer-Brandis\setcounter{footnote}{6}\footnote{Department of Mathematics, LMU Munich, Germany. {\tt meyerbra@math.lmu.de}} \and Frank Proske\setcounter{footnote}{3}\footnote{Department of Mathematics, University of Oslo, Norway. {\tt proske@math.uio.no}}}
\maketitle

\begin{abstract}
We deduce stability and pathwise uniqueness for a McKean-Vlasov equation with random coefficients and a multidimensional Brownian motion as driver. Our analysis focuses on a non-Lipschitz drift coefficient and includes moment estimates for random It{\^o} processes that are of independent interest. For deterministic coefficients we provide unique strong solutions, even if the drift fails to be of affine growth. The theory that we develop rests on It{\^o}'s formula and leads to $p$-th moment and pathwise $\alpha$-exponential stability for $p\geq 2$ and $\alpha > 0$ with explicit Lyapunov exponents, regardless of whether a Lyapunov function exists.
\end{abstract}


%-------------------------------------------------------------------
% Introduction
%-------------------------------------------------------------------
\section{Introduction}\label{se:1}

The study of McKean-Vlasov stochastic differential equations (McKean-Vlasov SDEs), also called mean-field SDEs, was initiated by~Kac~\cite{Kac56}, McKean~\cite{McK66} and Vlasov~\cite{Vla68}. Since then, these integral equations gained much interest in a variety of fields such as physics, economics, finance and mathematics. A crucial reason for this is the fact they can be used to model the propagation of chaos of interacting particles in a
plasma, as shown in~\cite{Vla68}. In the context of applications of these types, questions regarding stability, uniqueness and existence of solutions arise. To give precise answers under verifiable conditions, we continue our analysis of the companion paper~\cite{KalMeyPro21} and present new methods.

Let $d,m\in\N$ and $(\Omega,\mathscr{F},(\mathscr{F}_{t})_{t\in\R_{+}},P)$ be a filtered probability space that satisfies the usual conditions and on which there is a standard $d$-dimensional $(\mathscr{F}_{t})_{t\in\R_{+}}$-Brownian motion $W$. Then a McKean-Vlasov equation can be written in the form
\begin{equation}\label{eq:initial McKean-Vlasov}
X_{t} = X_{0} + \int_{0}^{t}b(s,X_{s},P_{X_{s}})\,ds + \int_{0}^{t}\sigma(s,X_{s},P_{X_{s}})\,dW_{s}\quad\text{for $t\in\R_{+}$ a.s.}
\end{equation}
Thereby, $b$ and $\sigma$ are two measurable maps on $\R_{+}\times\R^{m}\times\mathscr{P}(\R^{m})$ with values in $\R^{m}$ and $\R^{m\times d}$, respectively, $P_{X}$ denotes the law map of the solution process $X$ and, as usually, $\mathscr{P}(\R^{m})$ stands for the convex space of all Borel probability measures on $\R^{m}$.

The theory of mean-field SDEs has in fact undergone groundbreaking developments since the works of~\cite{Kac56},~\cite{McK66} and~\cite{Vla68} and proven to be an
indispensable mathematical tool. For instance, another important application in physics pertains to the
analysis of incompressible Navier-Stokes equations that were considered in the classical work~\cite{Ler34} by Leray and which have a deep link with mean-field SDEs, according to~Constantin and Iyer~\cite{ConIye08}. Recently, this connection was further explored by R\oe ckner and Zhao~\cite{RoeZha21}.

Moreover, mean-field games, studied by Lasry and Lions~\cite{LasLio07}, serve as applications in economics and may be used to explain the interaction and behaviour of agents in a vast network. For other
works related to mean-field games we refer to~\cite{CarDel13},~\cite{CarDel14},~\cite{CarDelLac16},~\cite{CarDel18I} and~\cite{CarDel18II}. Regarding applications in finance, see~\cite{CarFouMouSun18},~\cite{CarFouSun15},~\cite{FouIch13},~\cite{GarPapYan13} and \cite{KleKluRei15} in connection with systemic risk modelling.

Over the years, mean-field SDEs were studied from a mathematical point of view under various assumptions on the type of noise and the regularity of the coefficients. While the authors in~\cite{JouMelWoy08} consider L\'{e}vy noise, results on unique strong solutions, based on additive Gaussian noise and a discontinuous drift, were established in~\cite{BauMeyPro18},~\cite{BauMey19-1} and~\cite{BauMey19-2}, by using Malliavin calculus. Results on weak solutions can be found in~\cite{Jou97},~\cite{Chi94} and~\cite{MisVer20}. In~\cite{LiMin16} even path-dependent coefficient are treated. It is worth pointing out that mean-field equations of backward type were considered in~\cite{BucDjeLiPen09} and~\cite{BucLiPen09}, and infinite-dimensional partial differential equation (PDEs) related to mean-field SDEs were derived in~\cite{BucLiPenRai17}.

In the sequel, let $t_{0}\in\R_{+}$ and $\mathscr{P}$ be a separable metrisable topological space in $\mathscr{P}(\R^{m})$. For $p\in [2,\infty)$ the Polish space $\mathscr{P}_{p}(\R^{m})$ of all measures in $\mathscr{P}(\R^{m})$ with a finite $p$-th moment, endowed with the $p$-th Wasserstein metric, serves as our main application. Further, assume that along with $\mathscr{P}$ the maps
\begin{equation*}
\B:[t_{0},\infty)\times\Omega\times\R^{m}\times\mathscr{P}\rightarrow\R^{m}\quad\text{and}\quad \Sigma:[t_{0},\infty)\times\Omega\times\R^{m}\times\mathscr{P}\rightarrow\R^{m\times d}
\end{equation*}
are admissible in a suitable measurable sense, as explained in Section~\ref{se:2}. We shall focus on the following \emph{generalised McKean-Vlasov SDE} with such \emph{random drift} and \emph{diffusion coefficients}:
\begin{equation}\label{eq:McKean-Vlasov}
dX_{t} = \B_{t}(t,X_{t},P_{X_{t}})\,dt + \Sigma_{t}(X_{t},P_{X_{t}})\,dW_{t}\quad\text{for $t\in [t_{0},\infty)$.}
\end{equation}
In particular, if $\Sigma$ does not depend on the measure variable $\mu\in\mathscr{P}$, then we recover the setting of the previous work~\cite{KalMeyPro21}, in which a multidimensional Yamada-Watanabe approach was developed.

The objective of this paper is to deduce stability, uniqueness and existence of solutions to~\eqref{eq:McKean-Vlasov}, by presenting methods that handle the dependence of the diffusion coefficient with respect to the measure variable and, at the same time, allow for an irregular drift. Essentially, our contributions to the existing literature can be listed as follows:
\begin{enumerate}[(1)]
\item \emph{Pathwise uniqueness} for~\eqref{eq:McKean-Vlasov} is concisely shown in Corollary~\ref{co:pathwise uniqueness} if $\B$ and $\Sigma$ satisfy an Osgood condition that is only partially restrictive for $\B$. In the specific case that both coefficients are independent of $\mu\in\mathscr{P}$, turning~\eqref{eq:McKean-Vlasov} into an SDE, this condition is required on compact sets only.
\item As a consequence of the \emph{explicit $L^{p}$-comparison estimate} for~\eqref{eq:McKean-Vlasov} in Proposition~\ref{pr:specific stability moment estimate}, we obtain \emph{(asymptotic) $p$-th moment stability} in Corollary~\ref{co:moment stability} under partial and complete mixed continuity H\oe lder conditions on $\B$ and $\Sigma$, respectively, and transparent integrability hypotheses on the random partial H\oe lder coefficients with respect to $(x,\mu)\in\R^{m}\times\mathscr{P}$.
\item \emph{Exponential $p$-th moment stability} follows from Corollary~\ref{co:exponential moment stability} if partial and complete Lipschitz conditions hold for $\B$ and $\Sigma$, respectively, and the stability factor in~\eqref{eq:specific stability coefficient 2}, which can be viewed as functional of the partial Lipschitz coefficients, does not exceed a sum of power functions. The crucial fact is that this result provides \emph{explicit $p$-th moment Lyapunov exponents}.
\item \emph{Pathwise exponential stability} is shown in Corollary~\ref{co:pathwise stability} under the just mentioned Lipschitz conditions on $\B$ and $\Sigma$ with deterministic partial Lipschitz coefficients of certain growth and the same bound involving a sum of power functions for the stability factor $\gamma_{pq,\mathscr{P}}$ in~\eqref{eq:specific stability coefficient 3}, where $q\in [2,\infty)$. Then the \emph{pathwise Lyapunov exponent agrees with the $pq$-th moment Lyapunov exponent divided by $pq$}.
\item As the companion paper~\cite{KalMeyPro21}, this work demonstrates that for a detailed stability analysis of McKean-Vlasov equations verifiable assumptions can be given and the existence of Lyapunov functions does not have to be taken for granted.
\item \emph{Unique strong solutions} are derived in Theorem~\ref{th:strong existence} under the hypotheses that $\B$ and $\Sigma$ are deterministic, an Osgood growth or an affine growth estimate that is only partially restrictive for $\B$ holds and partial and complete Lipschitz conditions are satisfied by $\B$ and $\Sigma$, respectively. Thus, $\B$ is not forced to be of affine growth or Lipschitz continuous relative to $(x,\mu)\in\R^{m}\times\mathscr{P}$, as justified in Example~\ref{ex:deterministic integral maps} involving integral maps.
\end{enumerate}

We note that the contributions~(1),~(4) and~(6) are comparable to those in~\cite{KalMeyPro21} if we assume for the moment that $\Sigma$ is independent of $\mu\in\mathscr{P}$. In this case, the \emph{pathwise uniqueness} assertions of Corollary~3.9 in~\cite{KalMeyPro21} are applicable when $\Sigma$ merely satisfies an Osgood condition on compact sets with random regularity coefficients relative to a random basis of $\R^{m}$. Thereby, in contrast to Corollary~\ref{co:pathwise uniqueness} in this article, $\B$ is ought to satisfy \emph{a partial Osgood condition that depends on the choice of this basis}.

The \emph{pathwise exponential stability} statements of Corollary~3.19 in~\cite{KalMeyPro21} require $\B$ and $\Sigma$ to satisfy a partial Lipschitz condition and an $1/2$-H\oe lder condition in terms of a random basis, respectively. Then the same bound involving a sum of power functions as in this paper is imposed on the stability factor $\gamma_{\mathscr{P}}$ given by (3.11) in~\cite{KalMeyPro21}. This entails that one half of the first moment Lyapunov exponent, studied in Corollary~3.16 there, serves as pathwise Lyapunov exponent. We emphasize that while $\gamma_{\mathscr{P}}$ is merely influenced by the regularity of $\B$, the relevant stability factor $\gamma_{pq,\mathscr{P}}$ in~\eqref{eq:specific stability coefficient 3} for Corollary~\ref{co:pathwise stability} in this work, where $q\in [2,\infty)$, is based on the partial Lipschitz coefficients stemming from partial and complete Lipschitz conditions for $\B$ and $\Sigma$, respectively.

Moreover, for the \emph{strong existence and uniqueness result} in~\cite{KalMeyPro21}, Theorem~3.27, to hold, it is assumed that $\B$ and $\Sigma$ are deterministic, $\B$ satisfies a partial affine growth and a partial Lipschitz condition relative to a random basis and $\Sigma$ vanishes at the origin at all times and satisfies the same Osgood condition on compact sets, as mentioned before. In Theorem~\ref{th:strong existence} of this paper, however, the partial affine growth and partial Lipschitz conditions on $\B$ are less restrictive, and $\Sigma$ is supposed to satisfy an affine growth and a Lipschitz condition instead.
\smallskip

This work is structured as follows. In Section~\ref{se:2} the setting and the notation of our paper are introduced. In this context, we recall and extend several concepts from~\cite{KalMeyPro21} that are related to the generalised McKean-Vlasov equation~\eqref{eq:McKean-Vlasov}.

The main results are formulated in Section~\ref{se:3}. In detail, Section~\ref{se:3.1} gives a quantitative second moment estimate for the difference of two solutions from which pathwise uniqueness follows. In Section~\ref{se:3.2} we compare solutions in arbitrary moments, which in turn leads to standard, asymptotic and exponential stability in $p$-th moment. Then, Section~\ref{se:3.3} deals with pathwise stability and $L^{p}$-growth estimates. By combining these results, a strong existence and uniqueness result is stated in Section~\ref{se:3.4}. Thereby, all results are illustrated by a variety of examples involving integral maps.

Finally, Section~\ref{se:4} derives moment and pathwise asymptotic estimates for random It{\^o} processes, from which our main results will be inferred in the proofs that can be found in Section~\ref{se:5}.

\section{Preliminaries}\label{se:2}

In what follows, for any real-valued monotone function $f$ on an interval $I$ in $\R$ we define $f(a) := \lim_{v\downarrow a} f(v)$ for $a:=\inf I$, if $a\notin I$, and $f(b):=\lim_{v\uparrow b} f(v)$ for $b:=\sup I$, if $b\notin I$. Moreover, we use $|\cdot|$ as absolute value function, Euclidean norm or Hilbert-Schmidt norm and denote the transpose of any matrix $A\in\R^{m\times d}$ by $A'$.

\subsection{Processes with locally integrable elements}\label{se:2.1}

For $p\in [1,\infty)$ we introduce a pseudonorm $[\cdot]_{p}$ on the convex cone $\mathscr{L}_{+}^{p}(\Omega,\mathscr{F},P)$ of all random variables $X$ with $E[(X^{+})^{p}] < \infty$ and a  sublinear functional $[\cdot]_{\infty}$ on the convex cone $\mathscr{L}_{+}^{\infty}(\Omega,\mathscr{F},P)$ of all random variables $X$ with $X^{+} \leq c$ a.s.~for some $c\in\R_{+}$ by
\begin{equation}\label{eq:essential functionals}
\big[X\big]_{p} := E\big[(X^{+}\big)^{p}\big]^{\frac{1}{p}}\quad\text{and}\quad \big[X\big]_{\infty} := \esssup X.
\end{equation}
Then the inequalities of H\oe lder and Young show that any $\beta,\gamma\in [0,p]$ with $\beta\leq\gamma$, each $X\in\mathscr{L}_{+}^{p/\gamma}(\Omega,\mathscr{F},P)$, every $\R_{+}$-valued random variable $Y$ and all $x\in\R_{+}$ satisfy
\begin{equation}
\begin{split}\label{eq:essential inequality}
p x E\big[XY^{p-\gamma}]E\big[Y^{p}\big]^{\frac{\beta}{p}} &\leq p x E\big[X\big]_{\frac{p}{\gamma}}E\big[Y^{p}\big]^{1 - \frac{\gamma - \beta}{p}}\\
&\leq \big[X\big]_{\frac{p}{\gamma}}\big((\gamma - \beta)x^{\frac{p}{\gamma-\beta}} + (p + \beta - \gamma)E\big[Y^{p}\big]\big)
\end{split}
\end{equation}
provided $x=1$ if $\beta=\gamma$, since $1^{\infty} = \lim_{q\uparrow\infty} 1^{q} = 1$. By means of these bounds, we derive the quantitative $L^{p}$-estimate of Theorem~\ref{th:stability moment estimate} for $p\geq 2$, on which our main results rely. To allow for infinite values, let us generalise the definitions of $[\cdot]_{p}$ and $[\cdot]_{\infty}$ in~\eqref{eq:essential functionals} to any random variable $X$. 

Further, let $\mathscr{L}_{loc}^{p}(\R^{m\times d})$ denote the linear space of all $\R^{m\times d}$-valued Borel measurable locally $p$-fold integrable maps on $[t_{0},\infty)$ and $\mathscr{L}_{loc}^{p}(\R_{+}^{m\times d})$ stand for the convex cone of all maps in $\mathscr{L}_{loc}^{p}(\R^{m\times d})$ with non-negative coordinates.

The linear space of all $\R^{m\times d}$-valued $(\mathscr{F}_{t})_{t\in [t_{0},\infty)}$-progressively measurable processes on $[t_{0},\infty)\times\Omega$ is denoted by $\mathscr{S}(\R^{m\times d})$ and for $p\in [1,\infty]$ we let $\mathscr{S}_{loc}^{p}(\R^{m\times d})$ be the subspace of all $\kappa\in\mathscr{S}(\R^{m\times d})$ with locally integrable paths satisfying
\begin{equation*}
\int_{t_{0}}^{t}\big| \big[\kappa_{s}^{(i,j)}\big]_{p}\big|\,ds < \infty\quad\text{for all $t\in [t_{0},\infty)$}
\end{equation*}
and any $(i,j)\in\{1,\dots,m\}\times\{1,\dots,d\}$, where $\kappa^{(i,j)}$ is the $(i,j)$-entry of $\kappa$. Similarly, let $\mathscr{S}_{2,\mathrm{loc}}^{p}(\R^{m\times d})$ be the linear space of all $\kappa\in\mathscr{S}(\R^{m\times d})$ with locally square-integrable paths for which $[\kappa^{(i,j)}]_{p}$ is locally square-integrable for all $(i,j)\in\{1,\dots,m\}\times\{1,\dots,d\}$.

For instance, let $X$ be an $\mathscr{F}_{t_{0}}$-measurable random variable and $\hat{\kappa}$ be a measurable map on $[t_{0},\infty)$ with values in $\R_{+}^{m\times d}$. Then $\kappa\in\mathscr{S}(\R^{m\times d})$ defined by $\kappa_{s}(\omega):=\hat{\kappa}(s)X(\omega)$ lies in $\mathscr{S}_{2,\mathrm{loc}}^{p}(\R^{m\times d})$ if and only if $X\in\mathscr{L}_{+}^{p}(\Omega,\mathscr{F},P)$ and $\hat{\kappa}\in\mathscr{L}_{loc}^{2}(\R^{m\times d})$.

Finally, let $\mathscr{S}(\R_{+}^{m\times d})$ and $\mathscr{S}_{loc}^{p}(\R_{+}^{m\times d})$ stand for the convex cones of all processes in $\mathscr{S}(\R^{m\times d})$ and $\mathscr{S}_{loc}^{p}(\R^{m\times d})$ with non-negative entries, respectively. These concepts lead to sharp $L^{p}$-estimates for $p\in [2,\infty)$, as will be shown.

\subsection{Admissible spaces of probability measures and notions of solutions and stability}\label{se:2.2}

From now on, we let $\mathscr{P}$ be a separable metrisable topological space in $\mathscr{P}(\R^{m})$ that is \emph{admissible} in the sense of Definition~2.1 in~\cite{KalMeyPro21}. That is, for every metric space $S$, each probability space $(\tilde{\Omega},\tilde{\mathscr{F}},\tilde{P})$ and any continuous process $X:S\times\tilde{\Omega}\rightarrow\R^{m}$ satisfying $\tilde{P}_{X_{s}}\in\mathscr{P}$ for any $s\in S$, the map
\[
S\rightarrow\mathscr{P},\quad s\mapsto\tilde{P}_{X_{s}}
\]
is Borel measurable. Thereby, we write $\tilde{P}_{\tilde{X}}$ for the law $\tilde{P}\circ\tilde{X}^{-1}$ of any random vector $\tilde{X}:\tilde{\Omega}\rightarrow\R^{m}$ under $\tilde{P}$. Sufficient conditions for a metrisable topological space in $\mathscr{P}(\R^{m})$ to be admissible and examples of such spaces are given in Section~2.2 in~\cite{KalMeyPro21}.

In particular, our main application is included. Namely, for $p\in [1,\infty)$ the Polish space $\mathscr{P}_{p}(\R^{m})$ of all $\mu\in\mathscr{P}(\R^{m})$ admitting a finite $p$-th moment $\int_{\R^{m}}|x|^{p}\,\mu(dx)$, endowed with the \emph{$p$-th Wasserstein metric} defined via
\begin{equation}\label{eq:Wasserstein metric}
\nu_{p}(\mu,\nu):= \bigg(\inf_{\theta\in\mathscr{P}(\mu,\nu)} \int_{\R^{m}\times\R^{m}} |x-y|^{p}\,\theta(dx,dy)\bigg)^{\frac{1}{p}}
\end{equation}
is admissible. In this context, $\mathscr{P}(\mu,\nu)$ denotes the convex set of all Borel probability measures $\theta$ on $\R^{m}\times\R^{m}$ with first and second marginal distributions $\mu$ and $\nu$, respectively, for any $\mu,\nu\in\mathscr{P}(\R^{m})$.

Next, let $\mathscr{A}$ represent the progressive $\sigma$-field, consisting of all sets $A$ in $[t_{0},\infty)\times\Omega$ for which $\mathbbm{1}_{A}$ is progressively measurable. Then we shall call a map
\begin{equation*}
F:[t_{0},\infty)\times\Omega\times\R^{m}\times\mathscr{P}\rightarrow\R^{m\times d},\quad (s,\omega,x,\mu)\mapsto F_{s}(x,\mu)(\omega)
\end{equation*}
\emph{admissible} if it is $\mathscr{A}\otimes\mathscr{B}(\R^{m})\otimes\mathscr{B}(\mathscr{P})$-measurable. This property, which we also considered in Section~2.3 in~\cite{KalMeyPro21}, ensures that for each $X\in\mathscr{S}(\R^{m})$ and every Borel measurable map $\mu:[t_{0},\infty)\rightarrow\mathscr{P}$, the process
\begin{equation*}
[t_{0},\infty)\times\Omega\rightarrow\R^{m\times d},\quad (s,\omega)\mapsto F_{s}(X_{s}(\omega),\mu(s))(\omega)
\end{equation*}
is progressively measurable. Here and subsequently, we assume that the drift $\B$ and the diffusion $\Sigma$ of the McKean-Vlasov equation~\eqref{eq:McKean-Vlasov} are admissible, as described.

\begin{Definition}\label{de:solution}
A \emph{solution} to~\eqref{eq:McKean-Vlasov} is an $\R^{m}$-valued adapted continuous process $X$ such that $P_{X}$ is $\mathscr{P}$-valued, $\int_{t_{0}}^{t}|\B_{s}(X_{s},P_{X_{s}})| + |\Sigma_{s}(X_{s},P_{X_{s}})|^{2}\,ds < \infty$ for each $t\in [t_{0},\infty)$ and
\begin{equation*}
X_{t} = X_{t_{0}} + \int_{t_{0}}^{t}\B_{s}(X_{s},P_{X_{s}})\,ds + \int_{t_{0}}^{t}\Sigma_{s}(X_{s},P_{X_{s}})\,dW_{s}\quad\text{for all $t\in [t_{0},\infty)$ a.s.}
\end{equation*}
\end{Definition}

We readily check that for $\B$ and $\Sigma$ to be deterministic, it is necessary and sufficient that there are two Borel measurable maps $b$ and $\sigma$ on $[t_{0},\infty)\times\R^{m}\times\mathscr{P}$ with values in $\R^{m}$ and $\R^{m\times d}$, respectively, such that
\begin{equation}\label{eq:deterministic case}
\B_{s}(x,\mu) = b(s,x,\mu)\quad\text{and}\quad\Sigma_{s}(x,\mu) = \sigma(s,x,\mu)
\end{equation}
for all $(s,x,\mu)\in [t_{0},\infty)\times\R^{m}\times\mathscr{P}$. In this deterministic setting we may introduce \emph{weak and strong solutions} and write~\eqref{eq:McKean-Vlasov} formally as follows:
\begin{equation}\label{eq:deterministic McKean-Vlasov}
dX_{t} = b(t,X_{t},P_{X_{t}})\,dt + \sigma(t,X_{t},P_{X_{t}})\,dW_{t}\quad\text{for $t\in [t_{0},\infty)$}.
\end{equation}
Namely, any $\R^{m}$-valued $\mathscr{F}_{t_{0}}$-measurable random vector $\xi$ and the Brownian motion $W$ induce a filtration by $\null_{\xi}\mathscr{E}_{t}^{0}:=\sigma(\xi)\vee\sigma(W_{s} - W_{t_{0}}:s\in [t_{0},t])$ for all $t\in [t_{0},\infty)$. Then a solution $X$ to~\eqref{eq:deterministic McKean-Vlasov} with $X_{t_{0}} = \xi$ a.s.~is \emph{strong} if it is adapted to the right-continuous filtration of the augmented filtration of $(\null_{\xi}\mathscr{E}_{t}^{0})_{t\in [t_{0},\infty)}$.

A \emph{weak solution} is an $\R^{m}$-valued adapted continuous process $X$ defined on some filtered probability space $(\tilde{\Omega},\tilde{\mathscr{F}},(\tilde{\mathscr{F}}_{t})_{t\in\R_{+}},\tilde{P})$ that satisfies the usual conditions and on which there is a standard $d$-dimensional $(\tilde{\mathscr{F}}_{t})_{t\in\R_{+}}$-Brownian motion $\tilde{W}$ such that
\begin{equation*}
\tilde{P}_{X}\in\mathscr{P},\quad\int_{t_{0}}^{t}|b(s,X_{s},\tilde{P}_{X_{s}})| + |\sigma(s,X_{s},\tilde{P}_{X_{s}})|^{2}\,ds < \infty\quad\text{for any $t\in [t_{0},\infty)$}
\end{equation*}
and $X_{t} = X_{t_{0}} + \int_{t_{0}}^{t}b(s,X_{s},\tilde{P}_{X_{s}})\,ds + \int_{t_{0}}^{t}\sigma(s,X_{s},\tilde{P}_{X_{s}})\,d\tilde{W}_{s}$ for all $t\in [t_{0},\infty)$ a.s. In such a case, we shall say that $X$ solves~\eqref{eq:deterministic McKean-Vlasov} weakly on $(\tilde{\Omega},\tilde{\mathscr{F}},(\tilde{\mathscr{F}}_{t})_{t\in\R_{+}},\tilde{P})$ relative to $\tilde{W}$.

In general, we measure the regularity of the random coefficients $\B$ and $\Sigma$ with respect to the variable $\mu\in\mathscr{P}$ by means of an $\R_{+}$-valued Borel measurable functional $\vartheta$ on $\mathscr{P}\times\mathscr{P}$ and for $p\in [1,\infty)$ we assume that there is $c_{p,\mathscr{P}} > 0$ such that
\begin{equation}\label{eq:domination condition}
\vartheta(P_{X},P_{\tilde{X}}) \leq c_{p,\mathscr{P}}E\big[|X - \tilde{X}|^{p}\big]^{\frac{1}{p}}
\end{equation}
for any two $\R^{m}$-valued random vectors $X, \tilde{X}$ satisfying $P_{X},P_{\tilde{X}}\in\mathscr{P}$. For example, this condition is satisfied if $\vartheta$ is \emph{dominated} by the $p$-th Wasserstein metric $\vartheta_{p}$, given by~\eqref{eq:Wasserstein metric}, as follows:
\begin{equation}\label{eq:general domination condition}
\vartheta(\mu,\nu) \leq c_{p,\mathscr{P}}\vartheta_{p}(\mu,\nu)
\end{equation}
for any $\mu,\nu\in\mathscr{P}$. Thereby, we extend the definition of $\nu_{p}(\mu,\nu)$ for all $\mu,\nu\in\mathscr{P}(\R^{m})$, by allowing infinite values. Note that if in addition $\mathscr{P}$ lies in $\mathscr{P}_{p}(\R^{m})$ and satisfies
\begin{equation*}
\mu\circ\varphi^{-1}\in\mathscr{P}\quad\text{for all $\mu\in\mathscr{P}$}
\end{equation*}
and any $\R^{m}$-valued bounded uniformly continuous map $\varphi$ on $\R^{m}$ with $|\varphi| \leq |\cdot|$ such that $\vartheta$ is a pseudometric generating its topology, then $\mathscr{P}$ is automatically admissible, by Corollary~2.4 in~\cite{KalMeyPro21}.

\begin{Example}
Suppose that $\phi:\R\cup\{-\infty\}\rightarrow\R_{+}$ and $\varphi:\R^{m}\times\R^{m}\rightarrow\R$ are measurable, $\rho:\R_{+}\rightarrow\R_{+}$ is increasing and there is $c > 0$ such that
\begin{equation*}
\varphi(x,y) \leq \rho(|x-y|)\quad\text{and}\quad \rho(v + w)/c \leq \rho(v) + \rho(w)
\end{equation*}
for any $x,y\in\R^{m}$ and all $v,w\in\R_{+}$. Under the condition that $\mathscr{P}$ is included in the set $\mathscr{P}_{\rho}(\R^{m})$ of all $\mu\in\mathscr{P}(\R^{m})$ for which $\int_{\R^{m}}\rho(|x|)\,\mu(dx)$ is finite, we may take
\begin{equation*}
\vartheta(\mu,\nu) = \phi\bigg(\inf_{\theta\in\mathscr{P}(\mu,\nu)}\int_{\R^{m}\times\R^{m}}\varphi(x,y)\,\theta(dx,dy)\bigg)\quad\text{for all $\mu,\nu\in\mathscr{P}$}
\end{equation*}
and the following three statements hold:
\begin{enumerate}[(1)]
\item Whenever $f,g:\R^{m}\rightarrow\R$ are measurable and satisfy $\varphi(x,y) = f(x) - g(y) \leq \rho(|x-y|)$ for every $x,y\in\R^{m}$, then
\begin{equation*}
\vartheta(\mu,\nu) = \phi\bigg(\int_{\R^{m}}f(x)\,\mu(dx) - \int_{\R^{m}}g(y)\,\nu(dy)\bigg)\quad\text{for any $\mu,\nu\in\mathscr{P}$}.
\end{equation*}
\item For the choice $\phi(v) = v^{1/p}$ and $\rho(v) = v^{p}$ for any $v\in\R_{+}$ and $\varphi(x,y) = |x-y|^{p}$ for all $x,y\in\R^{m}$  we get that $\mathscr{P}_{\rho}(\R^{m}) = \mathscr{P}_{p}(\R^{m})$ and $\vartheta = \vartheta_{p}$.
\item The inclusion $\mathscr{P}\subset\mathscr{P}_{p}(\R^{m})$ and the domination condition~\eqref{eq:general domination condition} are valid as soon as $\phi(v)\leq (v^{+})^{1/p}$ for each $v\in\R\cup\{-\infty\}$ and $\rho(v) = c_{p,\mathscr{P}}^{p}v^{p}$ for any $v\in\R_{+}$.
\end{enumerate}
\end{Example}

We notice that even in the case $(\mathscr{P},\vartheta) = (\mathscr{P}_{2}(\R^{m}),\vartheta_{2})$ for any solution $X$ to~\eqref{eq:McKean-Vlasov} the function $[t_{0},\infty)\rightarrow\R_{+}$, $s\mapsto E[|X_{s}|^{2}]$ is not necessarily locally integrable.  For instance, let for the moment the following partial and complete affine growth conditions hold:
\begin{equation*}
x'\B(x,\mu) \leq |x|\big(\kappa  + \chi\vartheta_{2}(\mu,\delta_{0})\big)\quad\text{and}\quad |\Sigma(x,\mu)| \leq \kappa + \chi\vartheta_{2}(\mu,\delta_{0})
\end{equation*}
for every $(x,\mu)\in\R^{m}\times\mathscr{P}_{2}(\R^{m})$ and some $\kappa,\chi\in\R_{+}$, where $\delta_{0}$ is the Dirac measure in $0\in\R^{m}$. Then It{\^o}'s formula immediately yields that
\begin{align*}
X_{t}^{2} - X_{t_{0}}^{2} - 2\int_{t_{0}}^{t}X_{s}'\Sigma_{s}(X_{s},P_{X_{s}})&\,dW_{s} = \int_{t_{0}}^{t}2X_{s}'\B_{s}(X_{s},P_{X_{s}}) + |\Sigma_{s}(X_{s},P_{X_{s}})|^{2}\,ds\\
&\leq \int_{t_{0}}^{t} 2|X_{s}|\big(\kappa + \chi E\big[|X_{s}|^{2}\big]^{\frac{1}{2}}\big) + \big(\kappa + \chi E\big[|X_{s}|^{2}\big]^{\frac{1}{2}}\big)^{2}\,ds
\end{align*}
for all $t\in [t_{0},\infty)$ a.s., since $\vartheta_{p}(\mu,\delta_{0})^{p} = \int_{\R^{m}}|x|^{p}\,\mu(dx)$ for any $\mu\in\mathscr{P}_{p}(\R^{m})$. Although the first appearing Lebesgue integral is finite, the second may be infinite, since the condition $P_{X}\in\mathscr{P}$ in Definition~\ref{de:solution} is merely equivalent to the square-integrability of $X$.

Certainly, this observation depends on the positivity of the variable $\chi$. Since if $\chi = 0$, then Lemmas~\ref{le:2nd moment growth estimate} and~\ref{le:p-th moment growth estimate} show that $E[|X|^{2}]$ is locally bounded. Thus, as in~\cite{KalMeyPro21}, we shall state all uniqueness, stability and existence results under a local integrability condition, which takes \emph{growth in the measure variable} into account.

For this purpose, we let $\Theta$ denote an $[0,\infty]$-valued Borel measurable functional on $[t_{0},\infty)\times\mathscr{P}^{2}\times\mathscr{P}(\R^{m})$ when $\mathscr{P}(\R^{m})$ is endowed with the Prokhorov metric and the induced Borel $\sigma$-field.

\begin{Definition}
\emph{Pathwise uniqueness} holds for~\eqref{eq:McKean-Vlasov} (with respect to $\Theta$) if every two solutions $X$ and $\tilde{X}$ with $X_{t_{0}} = \tilde{X}_{t_{0}}$ a.s.~(and for which the measurable function
\begin{equation}\label{eq:integrable function}
[t_{0},\infty)\rightarrow\R_{+},\quad s\mapsto\Theta(s,P_{X_{s}},P_{\tilde{X}_{s}},P_{X_{s} - \tilde{X}_{s}})
\end{equation}
is locally integrable) are indistinguishable.
\end{Definition}

To consider a relevant class of functionals such as $\Theta$ let $\mathrm{R}_{c}$ be the cone of all $\R_{+}$-valued continuous functions on $\R_{+}$ that are positive on $(0,\infty)$ and vanish at the origin.

\begin{Example}\label{ex:functional}
Let $\eta,\lambda\in\mathscr{L}_{loc}^{1}(\R_{+})$ and $\rho,\varrho\in\mathrm{R}_{c}$ be such that $\rho$ is concave, $\varrho$ is increasing and
\begin{equation*}
\Theta(s,\mu,\tilde{\mu},\nu) = \lambda(s)\varrho\big(\vartheta(\mu,\tilde{\mu})\big) + \eta(s)\int_{\R^{m}}\rho(|x|^{p})\,\nu(dx)
\end{equation*}
for any $s\in [t_{0},\infty)$, all $\mu,\tilde{\mu}\in\mathscr{P}$ and every $\nu\in\mathscr{P}(\R^{m})$. Then $\Theta(s,\mu,\tilde{\mu},\nu) < \infty$ as soon as $\nu\in\mathscr{P}_{p}(\R^{m})$, and for any two continuous processes $X$ and $\tilde{X}$ we have
\begin{equation*}
\Theta(\cdot,P_{X},P_{\tilde{X}},P_{X-\tilde{X}}) = \lambda\varrho\big(\vartheta(P_{X},P_{\tilde{X}})^{p}\big) + \eta E\big[\rho(|X-\tilde{X}|^{p})\big],
\end{equation*}
which yields a locally integrable function if $E[|X-\tilde{X}|^{p}]$ is locally bounded, for example.
\end{Example}

Based on the stability concepts of Definition 2.11~in~\cite{KalMeyPro21} that involve first moments, we formulate \emph{generalised notions of stability} for~\eqref{eq:McKean-Vlasov} in a \emph{global meaning}. In this regard, a shift of the stochastic drift and diffusion coefficients is not required, as the explicit argumentation preceeding the stability definitions in~\cite{KalMeyPro21} explains.

\begin{Definition}\label{de:stability}
Let $\alpha > 0$ and $p\in [1,\infty)$.
\begin{enumerate}[(i)]
\item Equation~\eqref{eq:McKean-Vlasov} is \emph{stable in $p$-th moment} (relative to $\Theta$) if for any two solutions $X$ and $\tilde{X}$ (for which~\eqref{eq:integrable function} is locally integrable) it holds that
\begin{equation*}
\sup_{t\in [t_{0},\infty)} E\big[|X_{t} - \tilde{X}_{t}|^{p}\big] < \infty,
\end{equation*}
provided $E[|X_{t_{0}} - \tilde{X}_{t_{0}}|^{p}] < \infty$. If additionally $\lim_{t\uparrow\infty} E[|X_{t} - \tilde{X}_{t}|^{p}] = 0$, then we refer to \emph{asymptotic stability in $p$-th moment}.
\item We say that~\eqref{eq:McKean-Vlasov} is \emph{$\alpha$-exponentially stable in $p$-th moment} (relative to $\Theta$) if there exist $\lambda < 0$ and $c\in\R_{+}$ such that any two solutions $X$ and $\tilde{X}$ satisfy
\begin{equation}\label{eq:exponential p-th moment stability}
E\big[|X_{t} - \tilde{X}_{t}|^{p}\big] \leq ce^{\lambda(t-t_{0})^{\alpha}}E\big[|X_{t_{0}} - \tilde{X}_{t_{0}}|^{p}\big]
\end{equation}
for each $t\in [t_{0},\infty)$ as soon as $E[|X_{t_{0}} - \tilde{X}_{t_{0}}|^{p}] < \infty$ (and~\eqref{eq:integrable function} is locally integrable). In this case, $\lambda$ is a \emph{$p$-th moment $\alpha$-Lyapunov exponent} for~\eqref{eq:McKean-Vlasov}.
\item We call~\eqref{eq:McKean-Vlasov} \emph{pathwise $\alpha$-exponentially} (relative to an initial $p$-th moment and $\Theta$) if there is $\lambda < 0$ such that
\begin{equation*}
\limsup_{t\uparrow\infty}\frac{1}{t^{\alpha}}\log\big(|X_{t} - \tilde{X}_{t}|\big) \leq \lambda\quad\text{a.s.}
\end{equation*}
(whenever $E[|X_{t_{0}} - \tilde{X}_{t_{0}}|^{p}] < \infty$ and~\eqref{eq:integrable function} is locally integrable), in which case $\lambda$ is a \emph{pathwise $\alpha$-Lyapunov exponent} for~\eqref{eq:McKean-Vlasov}.
\end{enumerate}
\end{Definition}

\begin{Remark}\label{re:exponential moment stability}
Suppose that $E[|X-\tilde{X}|^{p}]$ is locally bounded. Then from~\eqref{eq:exponential p-th moment stability} we directly infer that
\begin{equation*}
\limsup_{t\uparrow\infty} \frac{1}{t^{\alpha}}\log\big(E\big[|X_{t} - \tilde{X}_{t}|^{p}\big]\big) \leq \lambda.
\end{equation*}
Conversely, the latter bound yields the former for $\lambda + \varepsilon$ instead of $\lambda$ for any $\varepsilon > 0$. For a more general reasoning see Remark 2.12 in~\cite{KalMeyPro21}.
\end{Remark}

\section{Main results}\label{se:3}

\subsection{A quantitative second moment estimate and pathwise uniqueness}\label{se:3.1}

By comparing solutions to~\eqref{eq:McKean-Vlasov} with possibly different drift and diffusion coefficients, we obtain pathwise uniqueness. In this regard, let $\tilde{\B}$ and $\tilde{\Sigma}$ be two admissible maps on $[t_{0},\infty)\times\Omega\times\R^{m}\times\mathscr{P}$, taking their values in $\R^{m}$ and $\R^{m\times d}$, respectively.

First, let us introduce an \emph{uniform error and continuity condition} for the coefficients $(\B,\Sigma)$ and $(\tilde{\B},\tilde{\Sigma})$ that is only partially restrictive for the drift coefficients $\B$ and $\tilde{\B}$ and involves the cone $\mathrm{R}_{c}$ introduced just before Example~\ref{ex:functional}:
\begin{enumerate}[label=(C.\arabic*), ref=C.\arabic*, leftmargin=\widthof{(C.1)} + \labelsep]
\item\label{con:1} There are $\alpha\in (0,1]$, $\varepsilon,\lambda\in\mathscr{S}_{loc}^{1}(\R_{+})$, $\eta\in\mathscr{S}_{loc}^{1/(1-\alpha)}(\R_{+})$ and $\rho,\varrho\in\mathrm{R}_{c}$ such that $\rho^{1/\alpha}$ is concave, $\varrho$ is increasing and
\begin{equation*}
2(x-\tilde{x})'\big(\B(x,\mu) - \tilde{\B}(\tilde{x},\tilde{\mu})\big) + |\Sigma(x,\mu) - \tilde{\Sigma}(\tilde{x},\tilde{\mu})|^{2} \leq \varepsilon + \eta\rho(|x-\tilde{x}|^{2}) + \lambda\varrho\big(\vartheta(\mu,\tilde{\mu})^{2}\big)
\end{equation*}
for any $x,\tilde{x}\in\R^{m}$ and all $\mu,\tilde{\mu}\in\mathscr{P}$ a.e.~on $[t_{0},\infty)$ a.s.
\end{enumerate}

In the case that $(\B,\Sigma)=(\tilde{\B},\tilde{\Sigma})$ and $\varepsilon$ vanishes,~\eqref{con:1} is simply a \emph{partial uniform continuity condition} for the coefficients of~\eqref{eq:McKean-Vlasov}. Moreover, we may interpret the expression $\varepsilon$ as \emph{error bound} for the differences $\B-\tilde{\B}$ and $\Sigma-\tilde{\Sigma}$.

For the succeeding $L^{2}$-estimate based on Bihari's inequality we recall that for any $\rho\in\mathrm{R}_{c}$ the function $\Phi_{\rho}\in C^{1}((0,\infty))$ given by
\begin{equation}\label{eq:rho-function 1}
\Phi_{\rho}(w) := \int_{1}^{w}\frac{1}{\rho(v)}\,dv
\end{equation}
is a strictly increasing $C^{1}$-diffeomorphism onto the interval $(\Phi_{\rho}(0),\Phi_{\rho}(\infty))$. Let $D_{\rho}$ be the set of all $(v,w)\in\R_{+}^{2}$ with $\Phi_{\rho}(v) + w < \Phi_{\rho}(\infty)$ and note that $\Psi_{\rho}:D_{\rho}\rightarrow\R_{+}$ given by
\begin{equation}\label{eq:rho-function 2}
\Psi_{\rho}(v,w) := \Phi_{\rho}^{-1}\big(\Phi_{\rho}(v) + w\big)
\end{equation}
is a continuous extension of a locally Lipschitz continuous function that is increasing in each coordinate. Under~\eqref{con:1}, we use the functions $\gamma,\delta\in\mathscr{L}_{loc}^{1}(\R_{+})$ defined by
\begin{equation*}
\gamma:=\alpha[\eta]_{\frac{1}{1-\alpha}} + \beta E[\lambda]\quad\text{and}\quad \delta:=(1-\alpha)[\eta]_{\frac{1}{1-\alpha}} + (1-\beta)E[\lambda]
\end{equation*}
for fixed $\beta\in (0,1]$ to give a \emph{quantitative $L^{2}$-bound}. Thereby, we require until the end of this section that the $L^{p}$-norm bound~\eqref{eq:domination condition} is valid for $p=2$.

\begin{Proposition}\label{pr:specific abstract second moment estimate}
Let~\eqref{con:1} hold, $X$ and $\tilde{X}$ be solutions to~\eqref{eq:McKean-Vlasov} with respective coefficients $(\B,\Sigma)$ and $(\tilde{\B},\tilde{\Sigma})$ so that $E[|Y_{t_{0}}|^{2}] < \infty$ for $Y:=X-\tilde{X}$ and $E[\lambda]\varrho(\vartheta(P_{X},P_{\tilde{X}})^{2})$ be locally integrable. Define $\rho_{0},\varrho_{0}\in C(\R_{+})$ via
\begin{equation*}
\rho_{0}(v) := \rho(v)^{\frac{1}{\alpha}}\quad\text{and}\quad\varrho_{0}(v) := \rho_{0}(v)\vee\varrho(c_{2,\mathscr{P}}^{2}v)^{\frac{1}{\beta}}
\end{equation*}
and assume that $\Phi_{\rho_{0}}(\infty) = \infty$ or $E[\eta\rho(|Y|^{2})]$ is locally integrable. 
Then $E[|Y|^{2}]$ is locally bounded and
\begin{equation*}
\sup_{s\in [t_{0},t]}E\big[|Y_{s}|^{2}\big] \leq\Psi_{\varrho_{0}}\bigg(E\big[|Y_{t_{0}}|^{2}\big] + \int_{t_{0}}^{t}E[\varepsilon_{s}] + \delta(s)\,ds,\int_{t_{0}}^{t}\gamma(s)\,ds\bigg)
\end{equation*}
for each $t\in [t_{0},t_{0}^{+})$, where $t_{0}^{+}$ represents the supremum over all $t\in [t_{0},\infty)$ for which
\begin{equation*}
\bigg(E\big[|Y_{t_{0}}|^{2}\big] + \int_{t_{0}}^{t}E[\varepsilon_{s}] + \delta(s)\,ds,\int_{t_{0}}^{t}\gamma(s)\,ds\bigg)\in D_{\varrho_{0}}.
\end{equation*}
\end{Proposition}

\begin{Remark}
From $\Phi_{\varrho_{0}}(\infty) = \infty$ it follows that $\Phi_{\rho_{0}}(\infty) = \infty$ and $D_{\varrho_{0}} = \R_{+}^{2}$. Thus,  $E[|Y|^{2}]$ is bounded in this case if $E[\varepsilon]$, $\gamma$ and $\delta$ are integrable. Moreover, the conditions
\begin{equation*}
\Phi_{\varrho_{0}}(0) = -\infty,\quad Y_{t_{0}}=0\quad\text{a.s.}\quad\text{and}\quad E[\varepsilon]=\delta = 0\quad\text{a.e.}
\end{equation*}
imply $t_{0}^{+} = \infty$ and $Y= 0$ a.s. This fact will be used to derive pathwise uniqueness.
\end{Remark}

\begin{Example}\label{ex:modulus of continuity}
Assume that $c_{2,\mathscr{P}}=\alpha=\beta=1$ and $\rho(v) = \varrho(v) = \hat{\alpha} v(|\log(v)| + 1)$ for any $v\in\R_{+}$ and some $\hat{\alpha}\in (0,1]$. Then $\Phi_{\varrho_{0}}(0) = - \infty$, $\Phi_{\varrho_{0}}(\infty) = \infty$ and we have
\begin{equation*}
\Psi_{\varrho_{0}}(v,w) =
\begin{cases}
e^{(1+\log(v))e^{\hat{\alpha}w} -1} & \text{if $v\geq 1$},\\
e^{\frac{e^{\hat{\alpha} w}}{1-\log(v)} -1 }& \text{if $1 > v \geq \exp(1-e^{\hat{\alpha} w})$},\\
e^{1 - (1-\log(v))e^{-\hat{\alpha} w}} & \text{if $v < \exp(1 - \hat{\alpha} w)$}
\end{cases}
\end{equation*}
for every $v,w\in\R_{+}$, which leads to an explicit $L^{2}$-estimate in Proposition~\ref{pr:specific abstract second moment estimate}.
\end{Example}

To deduce pathwise uniqueness from the comparison, we restrict~\eqref{con:1} to the case when $(\B,\Sigma)$ $=(\tilde{\B},\tilde{\Sigma})$, $\alpha=1$, $\epsilon=0$ and $\eta$ is deterministic. Further, if $\B$ and $\Sigma$ are independent of $\mu\in\mathscr{P}$, then this condition will be imposed on compact sets only:
\begin{enumerate}[label=(C.\arabic*), ref=C.\arabic*, leftmargin=\widthof{(C.3)} + \labelsep]
\setcounter{enumi}{1}
\item\label{con:2} There are $\eta\in\mathscr{L}_{loc}^{1}(\R_{+})$, $\lambda\in\mathscr{S}_{loc}^{1}(\R_{+})$ and $\rho,\varrho\in\mathrm{R}_{c}$ so that $\rho$ is concave, $\varrho$ is increasing and
\begin{equation*}
2(x-\tilde{x})'\big(\B(x,\mu) - \B(\tilde{x},\tilde{\mu})\big) + |\Sigma(x,\mu) - \Sigma(\tilde{x},\tilde{\mu})|^{2} \leq \eta\rho(|x-\tilde{x}|^{2}) + \lambda\varrho\big(\vartheta(\mu,\tilde{\mu})^{2}\big)
\end{equation*}
for any $x,\tilde{x}\in\R^{m}$ and all $\mu,\tilde{\mu}\in\mathscr{P}$ a.e.~on $[t_{0},\infty)$ a.s.
\item\label{con:3} $\B$ and $\Sigma$ are independent of $\mu\in\mathscr{P}$ and for each $n\in\N$ there are $\eta_{n}\in\mathscr{L}_{loc}^{1}(\R_{+})$ and a concave function $\rho_{n}\in\mathrm{R}_{c}$ such that
\begin{equation*}
2(x-\tilde{x})'\big(\hat{\B}(x) - \hat{\B}(\tilde{x})\big) + |\hat{\Sigma}(x) - \hat{\Sigma}(\tilde{x})|^{2} \leq \eta_{n}\rho_{n}(|x-\tilde{x}|^{2})
\end{equation*}
for every $x,\tilde{x}\in\R^{m}$ with $|x|\vee|\tilde{x}|\leq n$ a.e.~on $[t_{0},\infty)$ a.s., where $\hat{\B}:=\B(\cdot,\mu_{0})$ and $\hat{\Sigma}:=\Sigma(\cdot,\mu_{0})$ for fixed $\mu_{0}\in\mathscr{P}$.
\end{enumerate}

Under~\eqref{con:2}, pathwise uniqueness for~\eqref{eq:McKean-Vlasov} follows with respect to the $\R_{+}$-valued Borel measurable functional $\Theta$ on $[t_{0},\infty)\times\mathscr{P}^{2}\times\mathscr{P}_{2}(\R^{m})$ defined via
\begin{equation*}
\Theta(s,\mu,\tilde{\mu},\nu) := E\big[\lambda_{s}\big]\varrho\big(\vartheta(\mu,\tilde{\mu})^{2}\big) + \int_{\R^{m}}\rho(|y|^{2})\,\nu(dy).
\end{equation*}
If $\Phi_{\rho}(\infty) = \infty$, then the second term plays no role and we consider $\Theta:[t_{0},\infty)\times\mathscr{P}^{2}\rightarrow\R_{+}$ given by $\Theta(\cdot,\mu,\tilde{\mu}):=E[\lambda]\varrho(\vartheta(\mu,\tilde{\mu})^{2})$ instead, as this yields a weaker local integrability condition.
\begin{Corollary}\label{co:pathwise uniqueness}
The following two assertions hold:
\begin{enumerate}[(i)]
\item If~\eqref{con:2} holds and $\int_{0}^{1}(\rho(v)\vee\varrho(c_{2,\mathscr{P}}^{2}v))^{-1}\,dv = \infty$, then pathwise uniqueness for~\eqref{eq:McKean-Vlasov} relative to $\Theta$ is valid.
\item If~\eqref{con:3} is satisfied and $\int_{0}^{1}\rho_{n}(v)^{-1}\,dv = \infty$ for all $n\in\N$, then pathwise uniqueness for the SDE~\eqref{eq:McKean-Vlasov} follows.
\end{enumerate}
\end{Corollary}

\begin{Remark}\label{re:pathwise uniqueness}
If $\B$ and $\Sigma$ are deterministic, in which case~\eqref{eq:deterministic case} holds, then pathwise uniqueness for~\eqref{eq:deterministic McKean-Vlasov} in the standard sense follows from the corollary, by stating the uniform continuity condition~\eqref{con:2} when $\lambda$ is independent of $\omega\in\Omega$ and lies in $\mathscr{L}_{loc}^{1}(\R_{+})$ and by using the domination condition~\eqref{eq:general domination condition} instead of the $L^{p}$-bound~\eqref{eq:domination condition} for $p=2$.
\end{Remark}

As application we consider the case that $\B$ and $\Sigma$ are \emph{integral maps}. Thereby, an $\R^{m\times d}$-valued map on $[t_{0},\infty)\times\Omega\times\R^{m}\times\R^{m}$ will be called admissible, as introduced in Section~\ref{se:2.2}, by viewing it as a map on $[t_{0},\infty)\times\Omega\times\R^{2m}\times\mathscr{P}$ that is independent of the measure variable $\mu\in\mathscr{P}$.

\begin{Example}
Let $\null_{0}\B$ and $\null_{0}\Sigma$ be maps on $[t_{0},\infty)\times\Omega\times\R^{m}\times\R^{m}$ with values in $\R^{m}$ and $\R^{m\times d}$ that are admissible such that $\null_{0}\B(x,\cdot)$ and $\null_{0}\Sigma(x,\cdot)$ are $\mu$-integrable,
\begin{equation}
\B(x,\mu) = \int_{\R^{m}}\null_{0}\B(x,y)\,\mu(dy)\quad\text{and}\quad \Sigma(x,\mu) = \int_{\R^{m}}\null_{0}\Sigma(x,y)\,\mu(dy)
\end{equation}
for all $(x,\mu)\in [t_{0},\infty)\times\mathscr{P}$. Then the following two assertions hold:
\begin{enumerate}[(1)]
\item Suppose that there are $\eta\in\mathscr{L}_{loc}^{1}(\R_{+})$, $\lambda\in\mathscr{S}_{loc}^{1}(\R_{+})$ and concave functions $\rho,\varrho\in\mathrm{R}_{c}$ such that $\varrho$ is increasing and
\begin{equation*}
2(x-\tilde{x})'\big(\null_{0}\B(x,y) - \null_{0}\B(\tilde{x},\tilde{y})\big) + |\null_{0}\Sigma(x,y) - \null_{0}\Sigma(\tilde{x},\tilde{y})|^{2}\leq \eta\rho(|x-\tilde{x}|^{2}) + \lambda\varrho(|y-\tilde{y}|^{2})
\end{equation*}
for all $x,\tilde{x},y,\tilde{y}\in\R^{m}$ a.e.~on $[t_{0},\infty)$ a.s. Then~\eqref{con:2} is valid for the choice $\vartheta=\vartheta_{2}$.
\item If in fact there are $\eta\in\mathscr{L}_{loc}^{1}(\R_{+}^{2})$, $\hat{\eta}_{1}\in\mathscr{L}_{loc}^{2}(\R_{+})$ and $\hat{\eta}_{2}\in\mathscr{S}_{loc}^{2}(\R_{+})$ such that
\begin{align*}
(x-\tilde{x})'\big(\null_{0}\B(x,y) - \null_{0}\B(\tilde{x},\tilde{y})\big) &\leq |x-\tilde{x}|\big(\eta_{1}|x-\tilde{x}| + \eta_{2}|y-\tilde{y}|\big)\\
\text{and}\quad |\null_{0}\Sigma(x,y) - \null_{0}\Sigma(\tilde{x},\tilde{y})| &\leq \hat{\eta}_{1}|x-\tilde{x}| + \hat{\eta}_{2}|y-\tilde{y}|
\end{align*}
for any $x,\tilde{x},y,\tilde{y}\in\R^{m}$ a.e.~on $[t_{0},\infty)$ a.s., then~\eqref{con:2} follows for $\vartheta = \vartheta_{1}$ and $\rho(v)$ $=\varrho(v) = v$ for each $v\in\R_{+}$. In this case,
\begin{equation*}
\Theta(\cdot,\mu,\tilde{\mu}) = (\eta_{2} + 2[\hat{\eta}_{2}]_{2}^{2})\vartheta_{1}(\mu,\tilde{\mu})^{2}
\end{equation*}
for all $\mu,\tilde{\mu}\in\mathscr{P}$ and, under the hypothesis that $\vartheta=\vartheta_{1}$, pathwise uniqueness for~\eqref{eq:McKean-Vlasov} with respect to $\Theta$ holds.
\end{enumerate}
\end{Example}

\subsection{An explicit moment estimate and moment stability}\label{se:3.2}

In this section we compare two solutions with varying drift and diffusion coefficients in the $L^{p}$-norm for $p\in [2,\infty)$. The resulting estimate implies standard, asymptotic and exponential stability in $p$-th moment. 

To this end, we require a \emph{uniform error and mixed H\oe lder continuity condition} for $(\B,\tilde{\B})$ and $(\Sigma,\tilde{\Sigma})$ that is only \emph{partially restrictive for the drift coefficients}:
\begin{enumerate}[label=(C.\arabic*), ref=C.\arabic*, leftmargin=\widthof{(C.4)} + \labelsep]
\setcounter{enumi}{3}
\item\label{con:4} There are $l\in\N$, $\alpha,\beta\in [0,1]^{l}$ with $\alpha + \beta\in [0,1]^{l}$, two $\R_{+}^{l}$-valued measurable maps $\zeta,\hat{\zeta}$ on $[t_{0},\infty)$ and for each $k\in\{1,\dots,l\}$ there are
\begin{equation*}
\null_{k}\eta\in\mathscr{S}_{loc}^{\frac{p}{1-\alpha_{k}}}(\R)\quad\text{and}\quad\null_{k}\hat{\eta}\in\mathscr{S}_{2,\mathrm{loc}}^{\frac{p}{1-\alpha_{k}}}(\R_{+})
\end{equation*}
such that $(x-\tilde{x})'(\B(x,\mu) - \tilde{\B}(\tilde{x},\tilde{\mu}))\leq \sum_{k=1}^{l}\zeta_{k} \,\null_{k}\eta|x-\tilde{x}|^{1 + \alpha_{k}}\vartheta(\mu,\tilde{\mu})^{\beta_{k}}$ and
\begin{equation*}
|\Sigma(x,\mu) - \tilde{\Sigma}(\tilde{x},\tilde{\mu})| \leq  \sum_{k=1}^{l}\hat{\zeta}_{k}\,\null_{k}\hat{\eta}|x-\tilde{x}|^{\alpha_{k}}\vartheta(\mu,\tilde{\mu})^{\beta_{k}}
\end{equation*}
for any $x,\tilde{x}\in\R^{m}$ and all $\mu,\tilde{\mu}\in\mathscr{P}$ a.e.~on $[t_{0},\infty)$ a.s. Further, for each $k\in\{1,\dots,l\}$ we have $\zeta_{k}=\hat{\zeta}_{k}=1$, if $\alpha_{k} + \beta_{k} = 1$, and the functions
\begin{equation}\label{eq:local integrability condition}
\zeta_{k}^{\frac{p}{1-\alpha_{k}-\beta_{k}}}\big[\null_{k}\eta\big]_{\frac{p}{1-\alpha_{k}}}\quad\text{and}\quad \hat{\zeta}_{k}^{\frac{2p}{1-\alpha_{k}-\beta_{k}}}\big[\null_{k}\hat{\eta}\big]_{\frac{p}{1-\alpha_{k}}}^{2}
\end{equation}
are locally integrable.
\end{enumerate}

\begin{Remark}
It is feasible to take $\zeta_{k} = \hat{\zeta}_{k} = 1$ for any $k\in\{1,\dots,l\}$ in~\eqref{con:4}. In this case, the functions in~\eqref{eq:local integrability condition} are automatically locally integrable and we allow for the setting 
\begin{equation*}
\null_{k}\eta = \null_{k}\kappa \quad\text{and}\quad \null_{k}\hat{\eta} =  \null_{k}\hat{\kappa}\quad\text{for each $k\in\{1,\dots,l\}$}
\end{equation*}
and some $\null_{k}\kappa\in\mathscr{L}_{loc}^{1}(\R)$ and $\null_{k}\hat{\kappa}\in\mathscr{L}_{loc}^{2}(\R_{+})$. However, the possibility to choose $\zeta$ and $\hat{\zeta}$ leads to the error estimate~\eqref{eq:error estimate} in Theorem~\ref{th:strong existence}, the announced strong existence result.
\end{Remark}

\begin{Example}\label{ex:error and Hoelder condition}
For $\alpha,\beta\in ]0,1]$ assume that there exist $\R_{+}$-valued measurable functions $\delta,\hat{\delta}$ on $[t_{0},\infty)$ and processes
\begin{equation*}
\null_{1}\eta\in\mathscr{S}_{loc}^{\frac{p}{1-\alpha}}(\R),\quad \null_{2}\eta\in\mathscr{S}_{2,loc}^{\frac{p}{1-\alpha}}(\R_{+}),\quad\varepsilon,\null_{2}\eta\in\mathscr{S}_{loc}^{p}(\R_{+})\quad\text{and}\quad\hat{\varepsilon},\null_{2}\hat{\eta}\in\mathscr{S}_{2,loc}^{p}(\R_{+})
\end{equation*}
such that $(x-\tilde{x})'(\B(x,\mu) - \tilde{\B}(\tilde{x},\tilde{\mu})) \leq |x-\tilde{x}|(\delta \varepsilon + \null_{1}\eta |x-\tilde{x}|^{\alpha} + \null_{2}\eta\vartheta(\mu,\tilde{\mu})^{\beta})$ and
\begin{equation*}
|\Sigma(x,\mu) - \tilde{\Sigma}(\tilde{x},\tilde{\mu})|\leq \hat{\delta} \hat{\varepsilon} + \null_{1}\hat{\eta}|x-\tilde{x}|^{\alpha} + \null_{2}\hat{\eta}\vartheta(\mu,\tilde{\mu})^{\beta}
\end{equation*}
for every $x,\tilde{x}\in\R^{m}$ and any $\mu,\tilde{\mu}\in\mathscr{P}$ a.e.~on $[t_{0},\infty)$. Further, let $\delta^{p}[\varepsilon]_{p}$ and $\hat{\delta}^{2p}[\hat{\varepsilon}]_{p}^{2}$ be locally integrable. Then~\eqref{con:4} holds for $l=3$ if we replace the coefficients
\begin{equation*}
\alpha,\quad \beta,\quad \zeta,\quad\hat{\zeta}\quad (\null_{1}\eta,\null_{2}\eta,\null_{3}\eta)\quad\text{and}\quad (\null_{1}\hat{\eta},\null_{2}\hat{\eta},\null_{3}\hat{\eta}),
\end{equation*}
appearing there, by $(0,\alpha,0)$, $(0,0,\beta)$, $(\delta,1,1)$, $(\hat{\delta},1,1)$, $(\varepsilon,\null_{1}\eta,\null_{2}\eta)$ and $(\hat{\varepsilon},\null_{1}\hat{\eta},\null_{2}\hat{\eta})
$.
\end{Example}

Generally, we note in~\eqref{con:4} that for any $k\in\{1,\dots,l\}$ with $\alpha_{k} = \beta_{k} = 0$ the coefficients $\zeta_{k}$, $\null_{k}\eta$, $\hat{\zeta}_{k}$ and $\null_{k}\hat{\eta}$ serve as \emph{error terms} for $\B-\tilde{\B}$ and $\Sigma-\tilde{\Sigma}$, respectively. Further, under this condition, we introduce two coefficients $\gamma_{p,\mathscr{P}}\in\mathscr{L}_{loc}^{1}(\R)$ and $\delta_{p,\mathscr{P}}\in\mathscr{L}_{loc}^{1}(\R_{+})$ by
\begin{equation}\label{eq:specific stability coefficient}
\begin{split}
\gamma_{p,\mathscr{P}} &:= \sum_{k=1}^{l}(p-1+\alpha_{k} + \beta_{k})c_{p,\mathscr{P}}^{\beta_{k}}\big[\null_{k}\eta\big]_{\frac{p}{1-\alpha_{k}}}\\
&\quad + c_{p}\sum_{j,k=1}^{l}(p-2 + \alpha_{j} + \alpha_{k} + \beta_{j} + \beta_{k})c_{p,\mathscr{P}}^{\beta_{j} + \beta_{k}}\big[\null_{j}\hat{\eta}\,\null_{k}\hat{\eta}\big]_{\frac{p}{2-\alpha_{j} - \alpha_{k}}}
\end{split}
\end{equation}
and
\begin{equation}\label{eq:drift stability coefficient}
\begin{split}
&\delta_{p,\mathscr{P}} := \sum_{k=1}^{l}(1-\alpha_{k} - \beta_{k})c_{p,\mathscr{P}}^{\beta_{k}}\zeta_{k}^{\frac{p}{1-\alpha_{k} - \beta_{k}}}\big[\null_{k}\eta\big]_{\frac{p}{1-\alpha_{k}}}\\
&\quad + c_{p}\sum_{j,k=1}^{l}(2-\alpha_{j}-\alpha_{k}-\beta_{j}-\beta_{k})c_{p,\mathscr{P}}^{\beta_{j} + \beta_{k}}(\hat{\zeta}_{j}\hat{\zeta}_{k})^{\frac{p}{2-\alpha_{j}-\alpha_{k}-\beta_{j}-\beta_{k}}}\big[\null_{j}\hat{\eta}\,\null_{k}\hat{\eta}\big]_{\frac{p}{2-\alpha_{j}-\alpha_{k}}}
\end{split}
\end{equation}
with $c_{p} := (p-1)/2$. Thereby, we observe that the term $[\null_{j}\hat{\eta}\null_{k}\hat{\eta}]_{p/(2-\alpha_{j}-\alpha_{k})}$ is indeed locally integrable for any $j,k\in\{1,\dots,l\}$, since H\oe lder's inequality gives
\begin{equation*}
\int_{t_{0}}^{t}\big[\null_{j}\hat{\eta}_{s}\,\null_{k}\hat{\eta}_{s}\big]_{\frac{p}{2-\alpha_{j} - \alpha_{k}}}\,ds \leq \int_{t_{0}}^{t} \big[\null_{j}\hat{\eta}_{s}\big]_{\frac{p}{1-\alpha_{j}}} \big[\null_{k}\hat{\eta}_{s}\big]_{\frac{p}{1-\alpha_{k}}}\,ds < \infty\quad\text{for all $t\in [t_{0},\infty)$.}
\end{equation*}
Moreover, since $\zeta_{k}=\hat{\zeta}_{k}=1$ for all $k\in\{1,\dots,l\}$ with $\alpha_{k} + \beta_{k} = 1$ and $1^{\infty} = \lim_{q\uparrow\infty} 1^{q} = 1$, by convention, Young's inequality yields that
\begin{equation*}
(2-\alpha_{j}-\alpha_{k}-\beta_{j}-\beta_{k})(\hat{\zeta}_{j}\hat{\zeta}_{k})^{\frac{p}{2-\alpha_{j}-\alpha_{k}-\beta_{j}-\beta_{k}}} \leq (1-\alpha_{j}-\beta_{j})\hat{\zeta}_{j}^{\frac{p}{1-\alpha_{j}-\beta_{j}}} + (1-\alpha_{k}-\beta_{k})\hat{\zeta}_{k}^{\frac{p}{1-\alpha_{k}-\beta_{k}}}
\end{equation*}
for any $j,k\in\{1,\dots,l\}$. This clarifies the local integrability of the expressions appearing inside the second sum in~\eqref{eq:drift stability coefficient}.

By means of the functions $\gamma_{p,\mathscr{P}}$ and $\delta_{p,\mathscr{P}}$ we get an \emph{explicit $L^{p}$-comparison estimate} under a local integrability condition involving the $\R_{+}$-valued Borel measurable functional $\Theta$ on $[t_{0},\infty)\times\mathscr{P}^{2}$ given by
\begin{equation}\label{eq:integrability functional}
\Theta(\cdot,\mu,\tilde{\mu}):=\sum_{\substack{k=1,\\ \beta_{k} > 0}}^{l}\zeta_{k}\big[\null_{k}\eta\big]_{\frac{p}{1-\alpha_{k}}}\vartheta(\mu,\tilde{\mu})^{\beta_{k}} 
+ \hat{\zeta}_{k}^{2}\big[\null_{k}\hat{\eta}\big]_{\frac{p}{1-\alpha_{k}}}^{2}\vartheta(\mu,\tilde{\mu})^{2\beta_{k}}.
\end{equation}
For instance, let $\mu,\tilde{\mu}$ be two $\mathscr{P}$-valued Borel measurable functions on $[t_{0},\infty)$ for which the function $[t_{0},\infty)\rightarrow\R_{+}$, $s\mapsto\vartheta(\mu,\tilde{\mu})(s)$ is locally bounded. Then $\Theta(\cdot,\mu,\tilde{\mu})$ is locally integrable, because the functions
\begin{equation*}
\zeta_{k}[\null_{k}\eta]_{\frac{p}{1-\alpha_{k}}}\quad\text{and}\quad \hat{\zeta}_{k}^{2}[\null_{k}\hat{\eta}]_{\frac{p}{1-\alpha_{k}}}^{2}
\end{equation*}
possess this property for each $k\in\{1,\dots,l\}$, as another application of Young's inequality shows. In particular, if $\beta = 0$, then there is no dependence on the measure variable to consider and $\Theta = 0$.

\begin{Proposition}\label{pr:specific stability moment estimate}
Let~\eqref{con:4} hold, $X$ and $\tilde{X}$ be solutions to~\eqref{eq:McKean-Vlasov} with respective coefficients $(\B,\Sigma)$ and $(\tilde{\B},\tilde{\Sigma})$ such that $E[|Y_{t_{0}}|^{p}] < \infty$ for $Y:=X-\tilde{X}$ and $\Theta(\cdot,P_{X},P_{\tilde{X}})$ be locally integrable. Then
\begin{equation}\label{eq:specific stability moment estimate}
E\big[|Y_{t}|^{p}\big] \leq e^{\int_{t_{0}}^{t}\gamma_{p,\mathscr{P}}(s)\,ds}E\big[|Y_{t_{0}}|^{p}\big] + \int_{t_{0}}^{t}e^{\int_{s}^{t}\gamma_{p,\mathscr{P}}(\tilde{s})\,d\tilde{s}}\delta_{p,\mathscr{P}}(s)\,ds
\end{equation}
for all $t\in [t_{0},\infty)$. In particular, if $\gamma_{p,\mathscr{P}} ^{+}$ and $\delta_{p,\mathscr{P}}$ are integrable, then $E[|Y|^{p}]$ is bounded. If in addition $\gamma_{p,\mathscr{P}}^{-}$ fails to be integrable, then
\begin{equation*}
\lim_{t\uparrow\infty} E\big[|Y_{t}|^{p}\big] = 0.
\end{equation*}
\end{Proposition}

\begin{Remark}
The term $\delta_{p,\mathscr{P}}$ contains all the coefficients $\zeta_{k}$ and $\hat{\zeta}_{k}$, where $k\in\{1,\dots,l\}$, and is based on all H\oe lder exponents in $(0,1)$ appearing in~\eqref{con:4} in the following sense: $\delta_{p,\mathscr{P}}(s) = 0$ for fixed $s\in [t_{0},\infty)$ if and only if both conditions
\begin{equation*}
\zeta_{k}(s) = 0\quad\text{or}\quad \null_{k}\eta_{s} \leq 0\quad\text{a.s.}\quad\text{and}\quad \hat{\zeta}_{k}(s) = 0\quad\text{or}\quad \null_{k}\hat{\eta}_{s} = 0\quad\text{a.s.}
\end{equation*}
hold for every $k\in\{1,\dots,l\}$ with $\alpha_{k} + \beta_{k} < 1$.
\end{Remark}

Until the end of this section let $(\B,\Sigma)=(\tilde{\B},\tilde{\Sigma})$. Then~\eqref{con:4} turns into a \emph{mixed H\oe lder continuity condition} for $\B$ and $\Sigma$ if all the error terms disappear, that is, $\alpha_{k} + \beta_{k} > 0$ for any $k\in\{1,\dots,l\}$. Noteworthy, even if these expressions are in place, stability still follows.

\begin{Corollary}\label{co:moment stability}
Let~\eqref{con:4} be valid. Then~\eqref{eq:McKean-Vlasov} is (asymptotically) stable in $p$-th moment with respect to $\Theta$ if $\gamma_{p,\mathscr{P}}^{+}$ and $\delta_{p,\mathscr{P}}$ are integrable (and $\int_{t_{0}}^{\infty}\gamma_{p,\mathscr{P}}^{-}(s)\,ds = \infty)$.
\end{Corollary}

To analyse the \emph{$L^{p}$-boundedness} and the \emph{rate of $L^{p}$-convergence for solutions} in the succeeding Corollary~\ref{co:exponential moment stability}, we strengthen~\eqref{con:4} to a \emph{partial Lipschitz condition on $\B$} and a \emph{complete Lipschitz condition on $\Sigma$}:
\begin{enumerate}[label=(C.\arabic*), ref=C.\arabic*, leftmargin=\widthof{(C.5)} + \labelsep]
\setcounter{enumi}{4}
\item\label{con:5} There are $\eta_{1}\in\mathscr{L}_{loc}^{1}(\R)$, $\null_{2}\eta\in\mathscr{S}_{loc}^{p}(\R_{+})$, $\null_{1}\hat{\eta}\in\mathscr{S}_{2,loc}^{\infty}(\R_{+})$ and $\null_{2}\hat{\eta}\in\mathscr{P}_{2,loc}^{p}(\R_{+})$ such that
\begin{equation*}
(x-\tilde{x})'\big(\B(x,\mu) - \B(\tilde{x},\tilde{\mu})\big) \leq |x-\tilde{x}|\big(\eta_{1}|x-\tilde{x}| + \null_{2}\eta\vartheta(\mu,\tilde{\mu})\big)
\end{equation*}
and $|\Sigma(x,\mu) - \Sigma(\tilde{x},\tilde{\mu})| \leq \null_{1}\hat{\eta}|x-\tilde{x}| + \null_{2}\hat{\eta}\vartheta(\mu,\tilde{\mu})$ for any $x,\tilde{x}\in \R^{m}$ and all $\mu,\tilde{\mu}\in\mathscr{P}$ a.e.~on $[t_{0},\infty)$ a.s.
\end{enumerate}

Let us suppose that~\eqref{con:5} holds, in which case~\eqref{con:4} follows for $l=2$, $\alpha=(1,0)$ and $\beta=(0,1)$, as Example~\ref{ex:error and Hoelder condition} shows. Thus, the function $\delta_{p,\mathscr{P}}$ in~\eqref{eq:drift stability coefficient} is identically zero and the formula for the \emph{stability coefficient} $\gamma_{p,\mathscr{P}}$ in~\eqref{eq:specific stability coefficient} reduces to
\begin{equation}\label{eq:specific stability coefficient 2}
\frac{\gamma_{p,\mathscr{P}}}{p} = \eta_{1} + c_{p,\mathscr{P}}[\null_{2}\eta]_{p} + c_{p}\big([\null_{1}\hat{\eta}]_{\infty}^{2} + 2c_{p,\mathscr{P}}[\null_{1}\hat{\eta}\null_{2}\hat{\eta}]_{p} + c_{p,\mathscr{P}}^{2}[\null_{2}\hat{\eta}]_{p}^{2}\big).
\end{equation}
Thereby, we recall the fact that $\null_{1}\hat{\eta} = \hat{\kappa}$ for some $\hat{\kappa}\in\mathscr{L}_{loc}^{2}(\R_{+})$ is possible in~\eqref{con:5}. Moreover, the functional~\eqref{eq:integrability functional} is of the form
\begin{equation}\label{eq:integrability functional 2}
\Theta(\cdot,\mu,\tilde{\mu}) = [\null_{2}\eta]_{p}\vartheta(\mu,\tilde{\mu}) + [\null_{2}\hat{\eta}]_{p}^{2}\vartheta(\mu,\tilde{\mu}) ^{2}
\end{equation}
for all $\mu,\tilde{\mu}\in\mathscr{P}$. Hence, by means of the following \emph{upper bound on $\gamma_{p,\mathscr{P}}$ that involves sums of power functions} we derive \emph{exponential moment stability}.
\begin{enumerate}[label=(C.\arabic*), ref=C.\arabic*, leftmargin=\widthof{(C.6)} + \labelsep]
\setcounter{enumi}{5}
\item\label{con:6} Condition~\eqref{con:5} is satisfied and there are $l\in\N$, $\alpha\in (0,\infty)^{l}$ and $\hat{\lambda},s\in\R^{l}$ with $\alpha_{1} < \cdots < \alpha_{l}$ and $\hat{\lambda}_{l} < 0 $ such that
\begin{equation*}
\gamma_{p,\mathscr{P}}(s) \leq \sum_{k=1}^{l}\hat{\lambda}_{k}\alpha_{k}(s-s_{k})^{\alpha_{k}-1}\quad\text{for a.e.~$s\in [t_{1},\infty)$}
\end{equation*}
for some $t_{1}\in [t_{0},\infty)$ with $\max_{k\in\{1,\dots,l\}} s_{k} \leq t_{1}$.
\end{enumerate}

Based on the fact that the preceding condition implies the existence of some $\hat{t}_{1}\in [t_{0},\infty)$ such that $\gamma_{p,\mathscr{P}} < 0$ a.e.~on $[\hat{t}_{1},\infty)$, we state the just mentioned stability properties. 

\begin{Corollary}\label{co:exponential moment stability}
The following two assertions hold:
\begin{enumerate}[(i)]
\item Suppose that~\eqref{con:5} is valid and $\gamma_{p,\mathscr{P}}^{+}$ is integrable. Then for the difference $Y$ of any two solutions $X$ and $\tilde{X}$ to~\eqref{eq:McKean-Vlasov} we have
\begin{equation*}
\sup_{t\in [t_{0},\infty)} e^{\int_{t_{0}}^{t}\gamma_{p,\mathscr{P}}^{-}(s)\,ds}E\big[|Y_{t}|^{p}\big] < \infty,
\end{equation*}
assuming that $E[|Y_{t_{0}}|^{p}] < \infty$ and $\Theta(\cdot,P_{X},P_{\tilde{X}})$ is locally integrable. Moreover, if in addition $\int_{t_{0}}^{\infty}\gamma_{p,\mathscr{P}}^{-}(s)\,ds = \infty$, then 
\begin{equation*}
\lim_{t\uparrow\infty} e^{\alpha\int_{t_{0}}^{t}\gamma_{p,\mathscr{P}}^{-}(s)\,ds}E\big[|Y_{t}|^{p}\big] = 0\quad\text{for all $\alpha\in [0,1)$}.
\end{equation*}
\item Let~\eqref{con:6} be satisfied. Then~\eqref{eq:McKean-Vlasov} is $\alpha_{l}$-exponentially stable in $p$-th moment with respect to $\Theta$ with any moment $\alpha_{l}$-Lyapunov exponent in $(\hat{\lambda}_{l},0)$, and $\hat{\lambda}_{l}$ is a Lyapunov exponent provided
\begin{equation*}
\max_{k\in\{1,\dots,l\}} \hat{\lambda}_{k} \leq 0 \quad\text{and}\quad s_{l} \leq t_{0}.
\end{equation*}
\end{enumerate}
\end{Corollary}

To demonstrate how the coefficients in~\eqref{eq:specific stability coefficient} and~\eqref{eq:drift stability coefficient} can be computed and facilitate access to the preceding results, let us consider certain sums of affine and integral maps.

\begin{Example}\label{ex:stochastic integral maps}
Suppose that $\kappa\in\mathscr{S}(\R^{m})$, $\eta\in\mathscr{S}(\R^{m\times m})$ and $\hat{\kappa},\hat{\eta}\in\mathscr{S}(\R^{m\times d})$. Further, we take $N\in\N$ and $\hat{m}\in\N^{N}$ and for each $n\in\{1,\dots,N\}$ let
\begin{equation*}
\null_{n}\eta\in\mathscr{S}(\R_{+}^{m}),\quad \null_{n}\hat{\eta}\in\mathscr{S}(\R^{m\times\hat{m}_{n}})
\end{equation*}
and $f_{n}$ and $g_{n}$ be two Borel measurable maps on $\R^{m}\times\R^{m}$ with values in $\R^{m}$ and $\R^{\hat{m}_{n}\times d}$, respectively, such that $f_{n}(x,\cdot)$ and $g_{n}(x,\cdot)$ are $\mu$-integrable and
\begin{equation}\label{eq:stochastic integral maps 1}
\B(x,\mu) = \kappa + \eta x + \sum_{n=1}^{N} \mathrm{diag}(\null_{n}\eta)\int_{\R^{m}}f_{n}(x,y)\,\mu(dy)
\end{equation}
and
\begin{equation}\label{eq:stochastic integral maps 2}
\Sigma(x,\mu) = \hat{\kappa} + \mathrm{diag}(x)\hat{\eta} +  \sum_{n=1}^{N} \null_{n}\eta\int_{\R^{m}}g_{n}(x,y)\,\mu(dy)
\end{equation}
for all $(x,\mu)\in\R^{m}\times\mathscr{P}$. Then~\eqref{con:4} is satisfied for the subsequent choice $\vartheta=\vartheta_{1}$ under the following three conditions:
\begin{enumerate}[(1)]
\item There is $\zeta_{\eta}\in\mathscr{L}_{loc}^{1}(\R)$ such that $x'\eta x \leq \zeta_{\eta}|x|^{2}$ for all $x\in\R^{m}$ a.e.~on $[t_{0},\infty)$ a.s.~and $[|\hat{\eta}|]_{\infty}^{2}$ is locally integrable.
\item For $\alpha,\beta\in ]0,1]$ there exists $\lambda_{1},\dots,\lambda_{N}\in\R^{m}$ and $\upsilon_{1},\dots,\upsilon_{N}\in\R_{+}^{m}$ such that
\begin{equation}\label{eq:main example estimate 1}
\mathrm{sgn}(x_{i} - \tilde{x}_{i})\big(f_{n,i}(x,y) - f_{n,i}(\tilde{x},\tilde{y})\big) \leq \lambda_{n,i}|x-\tilde{x}|^{\alpha} + \upsilon_{n,i}|y-\tilde{y}|^{\beta}
\end{equation}
for all $x,\tilde{x},y,\tilde{y}\in\R^{m}$ and every $(n,i)\in\{1,\dots,N\}\times\{1,\dots,m\}$. Furthermore, $\null_{1}\tilde{\eta}:=\sum_{n=1}^{N}\null_{n}\eta'\lambda_{n}$ and $\null_{2}\tilde{\eta}:=\sum_{n=1}^{N}\null_{n}\eta'\upsilon_{n}$ lie in $\mathscr{S}_{loc}^{\frac{p}{1-\alpha}}(\R)$ and $\mathscr{S}_{loc}^{p}(\R_{+})$, respectively.
\item For $\hat{\alpha},\hat{\beta}\in ]0,1]$ there are $\hat{\lambda}_{1},\dots,\hat{\lambda}_{N},\hat{\upsilon}_{1},\dots,\hat{\upsilon}_{N}\in\R_{+}$ such that
\begin{equation}\label{eq:main example estimate 2}
|g_{n}(x,y) - g_{n}(\tilde{x},\tilde{y})|\leq \hat{\lambda}_{n}|x-\tilde{x}|^{\hat{\alpha}} + \hat{\upsilon}_{n}|y-\tilde{y}|^{\hat{\beta}}
\end{equation}
for all $x,\tilde{x},y,\tilde{y}\in\R^{m}$ and each $n\in\{1\dots,N\}$. Moreover, $\null_{1}\overline{\eta}:=\sum_{n=1}^{N}|\null_{n}\hat{\eta}|\hat{\lambda}_{n}$ and $\null_{2}\overline{\eta}:=\sum_{n=1}^{N}|\null_{n}\hat{\eta}|\hat{\upsilon}_{n}$ belong to $\mathscr{S}_{2,loc}^{\frac{p}{1-\hat{\alpha}}}(\R_{+})$ and $\mathscr{S}_{2,loc}^{p}(\R_{+})$, respectively.
\end{enumerate}
Thus, under these requirements, Proposition~\ref{pr:specific stability moment estimate} and Corollaries~\ref{co:moment stability} and~\ref{co:exponential moment stability} entail the following assertions:
\begin{enumerate}[(1)]
\setcounter{enumi}{3}
\item The bound~\eqref{eq:specific stability moment estimate} holds for the difference $Y$ of any two solutions $X$ and $\tilde{X}$ to~\eqref{eq:McKean-Vlasov} for which $E[|Y_{t_{0}}|^{p}] < \infty$ and $ [\null_{2}\tilde{\eta}]_{p}\vartheta(P_{X},P_{\tilde{X}})^{\beta} + [\null_{2}\overline{\eta}]_{p}^{2}\vartheta(P_{X},P_{\tilde{X}})^{2\hat{\beta}}
$ is locally integrable, where
\begin{align*}
\gamma_{p,\mathscr{P}} &= p\zeta_{\eta} + (p-1+\alpha)\big[\null_{1}\tilde{\eta}\big]_{\frac{p}{1-\alpha}} + (p-1+\beta)c_{p,\mathscr{P}}^{\beta}\big[\null_{2}\tilde{\eta}\big]_{p}\\
&\quad + c_{p}\bigg(p\big[|\hat{\eta}|\big]_{\infty}^{2} + \big(p-2(1-\hat{\alpha})\big)\big[\null_{1}\overline{\eta}\big]_{\frac{p}{1-\hat{\alpha}}}^{2} + \big(p-2(1-\hat{\beta})\big)c_{p,\mathscr{P}}^{2\hat{\beta}}\big[\null_{2}\overline{\eta}\big]_{p}^{2}\bigg)\\
&\quad + 2c_{p}\bigg((p-1+\hat{\alpha})\big[|\hat{\eta}|\null_{1}\overline{\eta}\big]_{\frac{p}{1-\hat{\alpha}}} + (p - 1 + \hat{\beta})c_{p,\mathscr{P}}^{\hat{\beta}}\big[|\hat{\eta}|\null_{2}\overline{\eta}\big]_{p}\bigg)\\
&\quad + 2c_{p}(p-2+\hat{\alpha} + \hat{\beta})c_{p,\mathscr{P}}^{\hat{\beta}}\big[\null_{1}\overline{\eta}\null_{2}\overline{\eta}\big]_{\frac{p}{2-\hat{\alpha}}}
\end{align*}
and
\begin{align*}
\delta_{p,\mathscr{P}} &= (1-\alpha)\big[\null_{1}\tilde{\eta}\big]_{\frac{p}{1-\alpha}} + (1-\beta)c_{p,\mathscr{P}}^{\beta}\big[\null_{2}\tilde{\eta}\big]_{p} + 2c_{p}(1-\hat{\alpha})\bigg(\big[\null_{1}\overline{\eta}\big]_{\frac{p}{1-\hat{\alpha}}}^{2} + \big[|\hat{\eta}|\null_{1}\overline{\eta}\big]_{\frac{p}{1-\hat{\alpha}}}\bigg)\\
&\quad + 2c_{p}c_{p,\mathscr{P}}^{\hat{\beta}}\bigg((1-\hat{\beta})\bigg(c_{p,\mathscr{P}}^{\hat{\beta}}\big[\null_{2}\overline{\eta}\big]_{p}^{2} + \big[|\hat{\eta}|\null_{2}\overline{\eta}\big]_{p}\bigg) + (2-\hat{\alpha}-\hat{\beta})\big[\null_{1}\overline{\eta}\null_{2}\overline{\eta}\big]_{\frac{p}{2-\hat{\alpha}}}\bigg).
\end{align*}
\item Suppose that $(\zeta_{\eta} + [\null_{1}\tilde{\eta}]_{\frac{p}{1-\alpha}}\mathbbm{1}_{\{1\}}(\alpha))^{+}$, $[\null_{1}\tilde{\eta}]_{\frac{p}{1-\alpha}}\mathbbm{1}_{]0,1[}(\alpha)$,
\begin{equation*}
[\null_{2}\tilde{\eta}]_{p},\quad [|\hat{\eta}|]_{\infty}^{2},\quad [\null_{1}\overline{\eta}]_{\frac{p}{1-\hat{\alpha}}}^{2}\quad\text{and}\quad [\null_{2}\overline{\eta}]_{p}^{2}
\end{equation*}
are integrable. Then $E[|Y|^{p}]$ is bounded and we have $\lim_{t\uparrow\infty} E[|Y_{t}|^{p}] = 0$ as soon as the negative part of $\zeta_{\eta} + [\null_{1}\tilde{\eta}]_{\frac{p}{1-\alpha}}\mathbbm{1}_{\{1\}}(\alpha)$ fails to be integrable.
\item Let $\null_{1}\tilde{\eta}\leq 0$ in case that $\alpha < 1$ and assume additionally that $\beta=\hat{\alpha}=\hat{\beta}=1$, in which case~\eqref{con:5} is satisfied and we may define an $\R_{+}$-valued Borel measurable functional $\Theta$ on $[t_{0},\infty)\times\mathscr{P}^{2}$ via
\begin{equation*}
\Theta(\cdot,\mu,\tilde{\mu}) := \big[\null_{2}\tilde{\eta}\big]_{p}\vartheta(\mu,\tilde{\mu}) + \big[\null_{2}\overline{\eta}\big]_{p}^{2}\vartheta(\mu,\tilde{\mu})^{2}.
\end{equation*}
\begin{enumerate}[(i)]
\item If the positive part of $\zeta_{\eta} + [\null_{1}\tilde{\eta}]_{\frac{p}{1-\alpha}}\mathbbm{1}_{\{1\}}(\alpha)$ is integrable (and its negative part fails to be integrable), then~\eqref{eq:McKean-Vlasov} is (asymptotically) stable in $p$-th moment relative to $\Theta$.
\item If there exist $\tilde{\lambda} < 0$ and $\tilde{\alpha} > 0$ such that
\begin{align*}
&\zeta_{\eta} + \big[\null_{1}\tilde{\eta}\big]_{\frac{p}{1-\alpha}}\mathbbm{1}_{\{1\}}(\alpha) + c_{p,\mathscr{P}}\big[\null_{2}\tilde{\eta}\big]_{p} + c_{p}\bigg(\big[|\hat{\eta}|\big]_{\infty}^{2} + \big[\null_{1}\overline{\eta}\big]_{\infty}^{2} + c_{p,\mathscr{P}}^{2}\big[\null_{2}\overline{\eta}\big]_{p}^{2}\bigg)\\
&\quad + 2c_{p}\bigg(\big[|\hat{\eta}|\null_{1}\overline{\eta}\big]_{\infty} + c_{p,\mathscr{P}}\big[|\hat{\eta}|\null_{2}\overline{\eta}\big]_{p} + c_{p,\mathscr{P}}\big[\null_{1}\overline{\eta}\null_{2}\overline{\eta}\big]_{p}\bigg) \leq \frac{\tilde{\lambda}}{p}\tilde{\alpha}(s - t_{0})^{\tilde{\alpha} - 1}
\end{align*}
for a.e.~$s\in [t_{0},\infty)$, then~\eqref{eq:McKean-Vlasov} is $\tilde{\alpha}$-exponentially stable in $p$-th moment with respect to $\Theta$ with Lyapunov exponent $\tilde{\lambda}$.
\end{enumerate}
\end{enumerate}
\end{Example}

\subsection{Pathwise stability and moment growth bounds}\label{se:3.3}

In the first part of this section we establish pathwise exponential stability for~\eqref{eq:McKean-Vlasov}. In this regard, we restrict the partial Lipschitz condition~\eqref{con:5} to the case that all the regularity coefficients are deterministic and a certain growth estimate holds:
\begin{enumerate}[label=(C.\arabic*), ref=C.\arabic*, leftmargin=\widthof{(C.8)} + \labelsep]
\setcounter{enumi}{7}
\item\label{con:8} There are $\eta\in\mathscr{L}_{loc}^{1}(\R^{2})$ and $\hat{\eta}\in\mathscr{L}_{loc}^{2}(\R_{+}^{2})$ with $\eta_{2}\geq 0$ such that
\[
(x-\tilde{x})'\big(\B(x,\mu) - \B(\tilde{x},\tilde{\mu})\big) \leq |x-\tilde{x}|\big(\eta_{1}|x-\tilde{x}| + \eta_{2}\vartheta(\mu,\tilde{\mu})\big)
\]
and $|\Sigma(x,\mu) - \Sigma(\tilde{x},\tilde{\mu})| \leq \hat{\eta}_{1}|x-\tilde{x}| + \hat{\eta}_{2}\vartheta(\mu,\tilde{\mu})$ for every $x,\tilde{x}\in\R^{m}$ and all $\mu,\tilde{\mu}\in\mathscr{P}$ a.e.~on $[t_{0},\infty)$ a.s. Further, there is $\hat{\delta} > 0$ so that
\begin{equation*}
\sup_{t\in [t_{0},\infty)} \int_{t}^{t+\hat{\delta}} h(s)\,ds < \infty\quad\text{for each $h\in\{\eta_{2},\hat{\eta}_{1}^{2},\hat{\eta}_{2}^{2}\}$}.
\end{equation*}
\end{enumerate}

\begin{Remark}
The preceding growth estimate is automatically satisfied if the regularity coefficients $\eta_{2}$, $\hat{\eta}_{1}$ and $\hat{\eta}_{2}$ are in fact locally bounded.
\end{Remark}

Given $q\in [2,\infty)$ we shall just for the next two results assume that~\eqref{eq:domination condition} holds when $p$ is replaced by $pq$ but may fail to be valid for $p$. Thus, under~\eqref{con:8}, the stability coefficient $\gamma_{pq,\mathscr{P}}$ and the functional $\Theta$, which we considered in~\eqref{eq:specific stability coefficient 2} and~\eqref{eq:integrability functional 2} for $p$ instead of $pq$, can simply be written in the form 
\begin{equation}\label{eq:specific stability coefficient 3}
\frac{\gamma_{pq,\mathscr{P}}}{pq} = \eta_{1} + c_{pq,\mathscr{P}}\eta_{2} + c_{pq}\big(\hat{\eta}_{1}^{2} + 2c_{pq,\mathscr{P}}\hat{\eta}_{1}\hat{\eta}_{2} + c_{pq,\mathscr{P}}^{2}c_{pq}\hat{\eta}_{2}^{2}\big) 
\end{equation}
and $\Theta(\cdot,\mu,\tilde{\mu}) = \eta_{2}\vartheta(\mu,\tilde{\mu}) + \hat{\eta}_{2}^{2}\vartheta(\mu,\tilde{\mu})^{2}$ for every $\mu,\tilde{\mu}\in\mathscr{P}$.  In addition, we impose the following abstract condition on $\gamma_{pq,\mathscr{P}}$ to deduce a \emph{general pathwise stability bound} from an application of Theorem~\ref{th:pathwise stability}, a pathwise result for random It{\^o} processes.
\begin{enumerate}[label=(C.\arabic*), ref=C.\arabic*, leftmargin=\widthof{(C.9)} + \labelsep]
\setcounter{enumi}{8}
\item\label{con:9} Condition~\eqref{con:8} is valid and there exist $\hat{\varepsilon}\in (0,1)$ and some strictly increasing sequence $(t_{n})_{n\in\N}$ in $[t_{0},\infty)$ such that $\gamma_{pq,\mathscr{P}} \leq 0$ a.e.~on $[t_{1},\infty)$,
\begin{equation*}
\sup_{n\in\N} (t_{n+1}-t_{n}) < \hat{\delta},\quad \lim_{n\uparrow\infty} t_{n} = \infty
\end{equation*}
and $\sum_{n=1}^{\infty}\exp(\varepsilon(pq)^{-1}\int_{t_{1}}^{t_{n}}\gamma_{pq,\mathscr{P}}(s)\,ds) < \infty$ for each $\varepsilon\in (0,\hat{\varepsilon})$.
\end{enumerate}

\begin{Proposition}\label{pr:pathwise stability}
Let~\eqref{con:9} be valid and $X$ and $\tilde{X}$ be two solutions to~\eqref{eq:McKean-Vlasov} for which $\Theta(\cdot,P_{X},P_{\tilde{X}})$ is locally integrable. Then for $Y:=X-\tilde{X}$ we have
\begin{equation*}
\limsup_{t\uparrow\infty}\frac{1}{\varphi(t)}\log(|Y_{t}|) \leq \frac{1}{pq}\limsup_{n\uparrow\infty}\frac{1}{\varphi(t_{n})}\int_{t_{1}}^{t_{n}}\gamma_{pq,\mathscr{P}}(s)\,ds\quad\text{a.s.}
\end{equation*}
for each increasing function $\varphi:[t_{1},\infty)\rightarrow\R_{+}$ that is positive on $(t_{1},\infty)$, under the condition that $E[|Y_{t_{0}}|^{pq}] < \infty$ or $\eta_{2} = \hat{\eta}_{2} = 0$.
\end{Proposition}

Since the following condition, which involves the same sum of power functions as in~\eqref{con:6}, implies~\eqref{con:9}, we obtain pathwise exponential stability.
\begin{enumerate}[label=(C.\arabic*), ref=C.\arabic*, leftmargin=\widthof{(C.10)} + \labelsep]
\setcounter{enumi}{9}
\item\label{con:10} Condition~\eqref{con:8} is satisfied and there exist $l\in\N$, $\alpha\in (0,\infty)^{l}$, $\hat{\lambda},s\in\R^{l}$ and $t_{1}\in [t_{0},\infty)$ with $\alpha_{1} < \cdots < \alpha_{l}$ and $\hat{\lambda}_{l} < 0 $ such that
\begin{equation*}
\max_{k\in\{1,\dots,l\}} s_{k} \leq t_{1}\quad\text{and}\quad \gamma_{pq,\mathscr{P}}(s) \leq \sum_{k=1}^{l}\hat{\lambda}_{k}\alpha_{k}(s-s_{k})^{\alpha_{k}-1}\quad\text{for a.e.~$s\in [t_{1},\infty)$}
\end{equation*}
\end{enumerate}

\begin{Corollary}\label{co:pathwise stability}
Under~\eqref{con:10}, the following two statements hold:
\begin{enumerate}[(i)]
\item The McKean-Vlasov SDE~\eqref{eq:McKean-Vlasov} is pathwise $\alpha_{l}$-exponentially with Lyapunov exponent $\hat{\lambda}_{l}/(pq)$ relative to an initial $pq$-th moment and $\Theta$.
\item If $\B$ and $\Sigma$ are actually independent of $\mu\in\mathscr{P}$, then the SDE~\eqref{eq:McKean-Vlasov} is pathwise $\alpha_{l}$-exponentially with Lyapunov exponent $\hat{\lambda}_{l}/(pq)$, without any restrictions.
\end{enumerate}
\end{Corollary}

As possible application let us consider Example~\ref{ex:stochastic integral maps} that involves integral maps.

\begin{Example}
Assume that $\B$ and $\Sigma$ admit the representations~\eqref{eq:stochastic integral maps 1} and~\eqref{eq:stochastic integral maps 2}. In this case,~\eqref{con:8} follows for $\vartheta=\vartheta_{1}$ from condition~(1) in Example~\ref{ex:stochastic integral maps} and the following two requirements:
\begin{enumerate}[(1)]
\setcounter{enumi}{1}
\item The two partial and complete H\oe lder continuity estimates~\eqref{eq:main example estimate 1} and~\eqref{eq:main example estimate 2} hold for each $(n,i)\in\{1,\dots,N\}\times\{1,\dots,m\}$ when $\hat{\alpha} = \beta = \hat{\beta} = 1$.
\item There are $\tilde{\eta}_{1}\in\mathscr{L}_{loc}^{1}(\R)$ and $\R_{+}$-valued measurable locally bounded functions $\tilde{\eta}_{2}$, $\overline{\eta}_{1}$ and $\overline{\eta}_{2}$ on $[t_{0},\infty)$ such that
\begin{equation*}
\sum_{n=1}^{N}\null_{n}\eta'\lambda_{n}\leq \tilde{\eta}_{1},\quad \sum_{n=1}^{N}\null_{n}\eta'\upsilon_{n} \leq \tilde{\eta}_{2},\quad \sum_{n=1}^{N}|\null_{n}\hat{\eta}|\hat{\lambda}_{n}\leq \overline{\eta}_{1}\quad\text{and}\quad\sum_{n=1}^{N}|\null_{n}\hat{\eta}|\hat{\upsilon}_{n}\leq\overline{\eta}_{2}
\end{equation*}
a.e.~on $[t_{0},\infty)$ a.s. Further, $\tilde{\eta}_{1}\leq 0$ in the case that $\alpha < 1$.
\end{enumerate}
Based on these coefficients, our considerations in Example~\ref{ex:stochastic integral maps} show that for the stability coefficient $\gamma_{pq,\mathscr{P}}$ in~\eqref{eq:specific stability coefficient 3} and the functional $\Theta$ we have
\begin{align*}
\frac{\gamma_{pq,\mathscr{P}}}{pq} &= \zeta_{\eta} + \tilde{\eta}_{1}\mathbbm{1}_{\{1\}}(\alpha) + c_{pq,\mathscr{P}}\tilde{\eta}_{2} + c_{pq}\bigg(\big[|\hat{\eta}|\big]_{\infty}^{2} + \overline{\eta}_{1}^{2} + c_{pq,\mathscr{P}}^{2}\overline{\eta}_{2}^{2}\bigg)\\
&\quad + 2c_{pq}\bigg(\big[|\hat{\eta}|\big]_{\infty}\overline{\eta}_{1} + c_{pq,\mathscr{P}}\big[|\overline{\eta}|\big]_{\infty}\overline{\eta}_{2} + c_{pq,\mathscr{P}}\overline{\eta}_{1}\overline{\eta}_{2}\bigg)
\end{align*}
and $\Theta(\cdot,\mu,\tilde{\mu}) = \tilde{\eta}_{2}\vartheta(\mu,\tilde{\mu}) + \overline{\eta}_{2}^{2}\vartheta(\mu,\tilde{\mu})^{2}$ for all $\mu,\tilde{\mu}\in\mathscr{P}$. Thus, if there are $\tilde{\lambda} < 0$ and $\tilde{\alpha} > 0$ such that 
\begin{equation*}
\gamma_{pq,\mathscr{P}} \leq \tilde{\lambda}\tilde{\alpha}(s-t_{0})^{\tilde{\alpha} - 1}\quad\text{for a.e.~$s\in [t_{0},\infty)$},
\end{equation*}
then Corollary~\ref{co:pathwise stability} yields the following two assertions:
\begin{enumerate}[(1)]
\setcounter{enumi}{3}
\item Equation~\eqref{eq:McKean-Vlasov} pathwise $\tilde{\alpha}$-exponentially stable with Lyapunov exponent $\tilde{\lambda}/(pq)$ with respect to an initial $pq$-th moment and $\Theta$.
\item In the case that $\tilde{\eta}_{2}=\overline{\eta}_{2} = 0$ when $\B$ and $\Sigma$ are independent of $\mu\in\mathscr{P}$ the SDE~\eqref{eq:McKean-Vlasov} is pathwise $\tilde{\alpha}$-exponentially stable with Lyapunov exponent $\tilde{\lambda}/(pq)$.
\end{enumerate}
\end{Example}

Now we deduce a second and a $p$-th moment estimate for solutions to~\eqref{eq:McKean-Vlasov}. As the first bound implies that their second moment functions are locally bounded, local integrability in terms of the functional $\Theta$ in Corollary~\ref{co:pathwise uniqueness} is always satisfied.

Similarly, the second bound ensures that the $p$-th moment function of any solution is locally bounded and thus, all local integrability requirements with respect to $\Theta$ in Corollaries~\ref{co:moment stability},~\ref{co:exponential moment stability} and~\ref{co:pathwise stability} hold automatically.

Let us give two \emph{growth conditions} on $(\B,\Sigma)$ that are only \emph{partially restrictive} for the drift $\B$. The first is required for the second moment estimate and includes different classes of growth behaviour, by using the cone $\mathrm{R}_{c}$, introduced in Section~\ref{se:2.2}. The second yields the $p$-th moment estimate and is of affine nature:
\begin{enumerate}[label=(C.\arabic*), ref=C.\arabic*, leftmargin=\widthof{(C.11)} + \labelsep]
\setcounter{enumi}{10}
\item\label{con:11} There are $\alpha\in (0,1]$, $\kappa,\chi\in\mathscr{S}_{loc}^{1}(\R_{+})$, $\upsilon\in\mathscr{S}_{\mathrm{loc}}^{1/(1-\alpha)}(\R_{+})$ and $\phi,\varphi\in\mathrm{R}_{c}$ such that $\phi^{1/\alpha}$ is concave, $\varphi$ is increasing and
\begin{equation*}
2x'\B(x,\mu) + |\Sigma(x,\mu)|^{2} \leq \kappa + \upsilon\phi(|x|^{2}) + \chi\varphi(\vartheta(\mu,\delta_{0})^{2})
\end{equation*}
for all $(x,\mu)\in\R^{m}\times\mathscr{P}$ a.e.~on $[t_{0},\infty)$ a.s.
\item\label{con:12} There are $l\in\N$, $\alpha,\beta\in [0,1]^{l}$ with $\alpha + \beta\in [0,1]^{l}$, two $\R_{+}^{l}$-valued measurable maps $\kappa,\hat{\kappa}$ on $[t_{0},\infty)$ and for each $k\in\{1,\dots,l\}$ there are
\begin{equation*}
\null_{k}\upsilon\in\mathscr{S}_{loc}^{\frac{p}{1-\alpha_{k}}}(\R)\quad\text{and}\quad\null_{k}\hat{\upsilon}\in\mathscr{S}_{2,\mathrm{loc}}^{\frac{p}{1-\alpha_{k}}}(\R_{+})
\end{equation*}
such that $x'\B(x,\mu)\leq \sum_{k=1}^{l}\kappa_{k}\,\null_{k}\upsilon|x|^{1 + \alpha_{k}}\vartheta(\mu,\delta_{0})^{\beta_{k}}$ and
\begin{equation*}
|\Sigma(x,\mu)| \leq \sum_{k=1}^{l}\hat{\kappa}_{k}\,\null_{k}\hat{\upsilon}|x|^{\alpha_{k}}\vartheta(\mu,\delta_{0})^{\beta_{k}}
\end{equation*}
for any $(x,\mu)\in\R^{m}\times\mathscr{P}$ a.e.~on $[t_{0},\infty)$ a.s. Moreover, for any $k\in\{1,\dots,l\}$ it holds that $\kappa_{k}=\hat{\kappa}_{k} = 1$, if $\alpha_{k} + \beta_{k} = 1$, and the functions
\begin{equation*}
\kappa_{k}^{\frac{p}{1-\alpha_{k} - \beta_{k}}}[\null_{k}\upsilon]_{\frac{p}{1-\alpha_{k}}}\quad\text{and}\quad \hat{\kappa}_{k}^{\frac{2p}{1-\alpha_{k} - \beta_{k}}}[\null_{k}\hat{\upsilon}]_{\frac{p}{1-\alpha_{k}}}^{2}
\end{equation*}
are locally integrable.
\end{enumerate}

\begin{Remark}
As in~\eqref{con:4} we could have $\kappa_{k}=\hat{\kappa}_{k}=1$ for all $k\in\{1,\dots,l\}$, in which case the additional local integrability requirement is redundant. By admitting more general choices of $\kappa$ and $\hat{\kappa}$, the growth estimate~\eqref{eq:growth estimate} in Theorem~\ref{th:strong existence} can be derived.
\end{Remark}

Provided~\eqref{con:11} is valid and $\beta\in (0,1]$, we define $f,g\in\mathscr{L}_{loc}^{1}(\R_{+})$ by
\begin{equation*}
f:=\alpha\big[\upsilon\big]_{\frac{1}{1-\alpha}} + \beta E\big[\chi\big]\quad\text{and}\quad g:=(1-\alpha)\big[\upsilon\big]_{\frac{1}{1-\alpha}} + (1-\beta)E\big[\chi\big]
\end{equation*}
and apply Proposition~\ref{pr:abstract second moment estimate} to state the following \emph{quantitative $L^{2}$-estimate}, which becomes explicit in the setting of Example~\ref{ex:modulus of continuity}.

\begin{Lemma}\label{le:2nd moment growth estimate}
Let~\eqref{eq:domination condition} hold for $p=2$,~\eqref{con:11} be valid, $X$ be a solution to~\eqref{eq:McKean-Vlasov} such that $E[|X_{t_{0}}|^{2}] < \infty$ and $E[\chi]\varphi(\vartheta(P_{X},\delta_{0})^{2})$ is locally integrable and $\varphi_{0}\in C(\R_{+})$ given by
\begin{equation*}
\varphi_{0}(v):=\phi(v)^{\frac{1}{\alpha}}\vee\varphi(c_{2,\mathscr{P}}^{2}v)^{\frac{1}{\beta}}
\end{equation*}
satisfy $\Phi_{\varphi_{0}}(\infty)=\infty$. Then $E[|X|^{2}]$ is locally bounded and
\begin{equation*}
\sup_{s\in [t_{0},t]} E\big[|X_{s}|^{2}\big] \leq \Psi_{\varphi_{0}}\bigg(E\big[|X_{t_{0}}|^{2}\big] + \int_{t_{0}}^{t}E\big[\kappa_{s}\big] + g(s)\,ds,\int_{t_{0}}^{t}f(s)\,ds\bigg)
\end{equation*}
for any $t\in [t_{0},\infty)$. In particular, if $E[\kappa]$, $f$ and $g$ are integrable, then $E[|X|^{2}]$ is bounded.
\end{Lemma}

We recall the constant $c_{p} = (p-1)/2$ and let the familiar estimate~\eqref{eq:domination condition} be valid. If the growth condition~\eqref{con:12} is satisfied, then we may let $f_{p,\mathscr{P}}\in\mathscr{L}_{loc}^{1}(\R)$ and $g_{p,\mathscr{P}}\in\mathscr{L}_{loc}^{1}(\R_{+})$ be given by
\begin{equation}\label{eq:growth coefficient}
\begin{split}
f_{p,\mathscr{P}} &:= \sum_{k=1}^{l}(p-1+\alpha_{k} + \beta_{k})c_{p,\mathscr{P}}^{\beta_{k}}\big[\null_{k}\upsilon]_{\frac{p}{1-\alpha_{k}}}\\
&\quad + c_{p}\sum_{j,k=1}^{l}(p-2 + \alpha_{j} + \alpha_{k} + \beta_{j} + \beta_{k})c_{p,\mathscr{P}}^{\beta_{j} + \beta_{k}}\big[\null_{j}\hat{\upsilon}\,\null_{k}\hat{\upsilon}\big]_{\frac{p}{2-\alpha_{j} - \alpha_{k}}}
\end{split}
\end{equation}
and
\begin{equation}\label{eq:drift growth coefficient}
\begin{split}
&g_{p,\mathscr{P}} := \sum_{k=1}^{l}(1-\alpha_{k} - \beta_{k})c_{p,\mathscr{P}}^{\beta_{k}}\kappa_{k}^{\frac{p}{1-\alpha_{k}-\beta_{k}}}\big[\null_{k}\upsilon\big]_{\frac{p}{1-\alpha_{k}}}\\
&\quad + c_{p}\sum_{j,k=1}^{l}(2-\alpha_{j}-\alpha_{k}-\beta_{j}-\beta_{k})c_{p,\mathscr{P}}^{\beta_{j} + \beta_{k}}(\hat{\kappa}_{j}\hat{\kappa}_{k})^{\frac{p}{2-\alpha_{j}-\alpha_{k}-\beta_{j}-\beta_{k}}}\big[\null_{j}\hat{\upsilon}\,\null_{k}\hat{\upsilon}\big]_{\frac{p}{2-\alpha_{j}-\alpha_{k}}}.
\end{split}
\end{equation}
Further, we define an $\R_{+}$-valued Borel measurable functional $\Theta$ on $[t_{0},\infty)\times\mathscr{P}^{2}$ by
\begin{equation*}
\Theta(\cdot,\mu,\tilde{\mu}):=\sum_{\substack{k=1,\\ \beta_{k} > 0}}^{l}\kappa_{k}\big[\null_{k}\upsilon\big]_{\frac{p}{1-\alpha_{k}}}\vartheta(\mu,\tilde{\mu})^{\beta_{k}} 
+ \hat{\kappa}_{k}^{2}\big[\null_{k}\hat{\upsilon}\big]_{\frac{p}{1 - \alpha_{k}}}^{2}\vartheta(\mu,\tilde{\mu})^{2\beta_{k}}.
\end{equation*}
These formulas are in essence the same as those for the stability coefficients in~\eqref{eq:specific stability coefficient} and~\eqref{eq:drift stability coefficient} and the functional in~\eqref{eq:integrability functional}, as the subsequent \emph{explicit $L^{p}$-growth estimate} and the $L^{p}$-comparison estimate in Proposition~\ref{pr:specific stability moment estimate} are implied by Theorem~\ref{th:stability moment estimate}.

\begin{Lemma}\label{le:p-th moment growth estimate}
Let~\eqref{con:12} hold and $X$ be a solution to~\eqref{eq:McKean-Vlasov} such that $E[|X_{t_{0}}|^{p}]$ $< \infty$ and $\Theta(\cdot,P_{X},\delta_{0})$ is locally integrable. Then
\begin{equation}\label{eq:p-th moment growth estimate}
E\big[|X_{t}|^{p}\big] \leq e^{\int_{t_{0}}^{t}f_{p,\mathscr{P}}(s)\,ds}E\big[|X_{t_{0}}|^{p}\big] + \int_{t_{0}}^{t}e^{\int_{s}^{t}f_{p,\mathscr{P}}(\tilde{s})\,d\tilde{s}}g_{p,\mathscr{P}}(s)\,ds
\end{equation}
for all $t\in [t_{0},\infty)$. In particular, if $f_{p,\mathscr{P}}^{+}$ and $g_{p,\mathscr{P}}$ are integrable, then $E[|X|^{p}]$ is bounded and from $\int_{t_{0}}^{\infty}f_{p,\mathscr{P}}^{-}(s)\,ds = \infty$ it follows that $\lim_{t\uparrow\infty} E[|X_{t}|^{p}] = 0$.
\end{Lemma}

At last, we recast Example~\ref{ex:stochastic integral maps} when the integer $N$ is replaced by $lN$ for some $l\in\N$.

\begin{Example}
Let $\kappa\in\mathscr{S}(\R^{m})$, $\eta\in\mathscr{S}(\R^{m\times m})$ and $\hat{\kappa},\hat{\eta}\in\mathscr{S}(\R^{m\times d})$ and $\hat{m}\in\N^{l\times N}$. For every $(k,n)\in\{1,\dots,l\}\times\{1,\dots,N\}$ we assume that
\begin{equation*}
\null_{k,n}\eta\in\mathscr{S}(\R_{+}^{m}),\quad \null_{k,n}\hat{\eta}\in\mathscr{S}(\R^{m\times\hat{m}_{k,n}})
\end{equation*}
and $f_{k,n}$ and $g_{k,n}$ are two Borel measurable maps on $\R^{m}\times\R^{m}$ taking their values in $\R^{m}$ and $\R^{\hat{m}_{k,n}\times d}$, respectively, such that $f_{k,n}(x,\cdot)$ and $g_{k,n}(x,\cdot)$ are $\mu$-integrable and
\begin{equation*}
\B(x,\mu) = \kappa + \eta x + \sum_{k=1}^{l}\sum_{n=1}^{N}\mathrm{diag}(\null_{k,n}\eta)\int_{\R^{m}}f_{k,n}(x,y)\,\mu(dy)
\end{equation*}
and
\begin{equation*}
\Sigma(x,\mu) = \hat{\kappa} + \mathrm{diag}(x)\hat{\eta} + \sum_{k=1}^{l}\sum_{n=1}^{N}\null_{k,n}\hat{\eta}\int_{\R^{m}}g_{k,n}(x,y)\,\mu(dy)
\end{equation*}
for every $(x,\mu)\in\R^{m}\times\mathscr{P}$. In this case,~\eqref{con:12} follows for $\vartheta=\vartheta_{1}$ from condition (1) in Example~\ref{ex:stochastic integral maps} and the following three conditions:
\begin{enumerate}[(1)]
\setcounter{enumi}{1}
\item $E[|\kappa|^{p}]^{1/p}$ and $E[|\hat{\kappa}|^{p}]^{2/p}$ are locally integrable.
\item For given $\alpha,\beta\in [0,1]^{l}$ with $\alpha + \beta\in [0,1]^{l}$ and every $k\in\{1,\dots,l\}$ there are $c_{k,1},\dots,c_{k,N}\in\R^{m}$ satisfying
\begin{equation*}
\mathrm{sgn}(x_{i})f_{k,n,i}(x,y)\leq c_{k,n,i}|x|^{\alpha_{k}}|y|^{\beta_{k}}
\end{equation*}
for any $x,y\in\R^{m}$ and all $(n,i)\in\{1,\dots,N\}\times\{1,\dots,m\}$ so that $\null_{k}\upsilon:=\sum_{n=1}^{N}\null_{k,n}\eta'c_{k,n}$ lies in $\mathscr{S}_{loc}^{\frac{p}{1-\alpha_{k}}}(\R)$.
\item Further, there exist $\hat{c}_{k,1},\dots,\hat{c}_{k,N}\in\R_{+}$ satisfying
\begin{equation*}
|g_{k,n}(x,y)| \leq \hat{c}_{k,n}|x|^{\alpha_{k}}|y|^{\beta_{k}}
\end{equation*}
for all $x,y\in\R^{m}$ and any $n\in\{1,\dots,N\}$ and $\null_{k}\hat{\upsilon}:=\sum_{n=1}^{N}|\null_{k,n}\hat{\eta}|\hat{c}_{k,n}$ belongs to $\mathscr{S}_{2,loc}^{\frac{p}{1-\alpha_{k}}}(\R_{+})$.
\end{enumerate}
Thus, if $\vartheta=\vartheta_{1}$ and these four requirements are met, then Lemma~\ref{le:p-th moment growth estimate} implies the following two assertions:
\begin{enumerate}[(1)]
\setcounter{enumi}{4}
\item For any solution $X$ to~\eqref{eq:McKean-Vlasov} such that $E[|X_{t_{0}}|^{p}] < \infty$ the estimate~\eqref{eq:p-th moment growth estimate} is valid if the local integrability of
\begin{equation*}
\sum_{\substack{k=1,\\ \beta_{k} > 0}}^{l}\big[\null_{k}\upsilon\big]_{\frac{p}{1-\alpha_{k}}}\vartheta(P_{X},\delta_{0})^{\beta_{k}} + \big[\null_{k}\hat{\upsilon}\big]_{\frac{p}{1 - \alpha_{k}}}^{2}\vartheta(P_{X},\delta_{0})^{2\beta_{k}}
\end{equation*}
holds. Thereby, the formulas~\eqref{eq:growth coefficient} and~\eqref{eq:drift growth coefficient} apply when $l$, $\alpha$, $\beta$ and $c_{p,\mathscr{P}}$ are replaced by $l+2$, $(\alpha,0,1)$, $(\beta,0,0)$ and $1$, respectively, with the choice 
\begin{equation*}
\kappa_{k}=\hat{\kappa}_{k} = 1\quad\text{for all $k\in\{1,\dots,l+2\}$},
\end{equation*}
$(\null_{k+1}\nu,\null_{k+1}\hat{\nu}) = (|\kappa|,|\hat{\kappa}|)$ and $(\null_{k+2}\nu,\null_{k+2}\hat{\nu}) = (\zeta_{\eta},|\hat{\eta}|)$.
\item Define $\underline{\upsilon}\in\mathscr{L}_{loc}^{1}(\R)$ by $\underline{\nu}:= \zeta_{\eta} + \sum_{k=1,\,\alpha_{k} = 1}^{l}[\null_{k}\upsilon]_{\infty}$. Then $E[|X|^{p}]$ is bounded provided
\begin{equation*}
\underline{\nu}^{+},\quad [|\hat{\eta}|]_{\infty}^{2},\quad \sum_{\substack{k=1,\\ \alpha_{k} < 1}}^{l}[\null_{k}\nu]_{\frac{p}{1-\alpha_{k}}},\quad \sum_{k=1}^{l}[\null_{k}\hat{\upsilon}]_{\frac{p}{1-\alpha_{k}}}^{2},\quad E\big[|\kappa|^{p}\big]^{\frac{1}{p}}\quad\text{and}\quad E\big[|\hat{\kappa}|^{p}\big]^{\frac{2}{p}}
\end{equation*}
are integrable. If additionally $\int_{t_{0}}^{\infty}\underline{\nu}^{-}(s)\,ds = \infty$, then $\lim_{t\uparrow\infty} E[|X_{t}|^{p}] = 0$.
\end{enumerate}
\end{Example}

\subsection{Strong solutions with locally bounded moment functions}\label{se:3.4}

Now we suppose that $b$ and $\sigma$ are two Borel measurable maps on $[t_{0},\infty)\times\R^{m}\times\mathscr{P}$ with values in $\R^{m}$ and $\R^{m\times d}$, respectively, and $\xi:\Omega\rightarrow\R^{m}$ is $\mathscr{F}_{t_{0}}$-measurable. The aim of this section is to deduce a strong solution $X$ to~\eqref{eq:deterministic McKean-Vlasov} with $X_{t_{0}}=\xi$ a.s.~such that the measurable $p$-th moment function
\begin{equation*}
[t_{0},\infty)\rightarrow [0,\infty],\quad t\mapsto E\big[|X_{t}|^{p}\big]
\end{equation*}
is finite and locally bounded for $p\in [2,\infty)$. Namely, we use the preceding comparison and growth results to construct the law of the solution as local uniform limit of a Picard iteration. In this setting, $b(s,\cdot,\mu)$ is not required to be Lipschitz continuous or of affine growth for any $(s,\mu)\in [t_{0},\infty)\times\mathscr{P}$.

For a Borel measurable map $\mu:[t_{0},\infty)\rightarrow\mathscr{P}$ we define two measurable maps $b_{\mu}$ and $\sigma_{\mu}$ on $[t_{0},\infty)\times\R^{m}$ with values in $\R^{m}$ and $\R^{m\times d}$, respectively, by $b_{\mu}(t,x):=b(t,x,\mu(t))$ and $\sigma_{\mu}(t,x):=\sigma(t,x,\mu(t))$. These two coefficients induce the SDE
\begin{equation}\label{eq:mu-SDE}
dX_{t} = b_{\mu}(t,X_{t})\,dt + \sigma_{\mu}(t,X_{t})\,dW_{t}\quad\text{for $t\in [t_{0},\infty)$}.
\end{equation}
To obtain strong solutions for this equation, we introduce a growth as well as a continuity and boundedness condition and a spatial Osgood condition on compact sets on $(b,\sigma)$:
\begin{enumerate}[label=(D.\arabic*), ref=D.\arabic*, leftmargin=\widthof{(D.3)} + \labelsep]
\item\label{con:det.1} There exist $\kappa,\upsilon,\chi\in\mathscr{L}_{loc}^{1}(\R_{+})$ and $\phi,\varphi\in\mathrm{R}_{c}$ so that $\phi$ is concave, $\int_{1}^{\infty}\phi(v)^{-1}\,dv = \infty$ and
\begin{equation*}
2x'b(\cdot,x,\mu) + |\sigma(\cdot,x,\mu)|^{2} \leq \kappa + \upsilon\phi(|x|^{2}) + \chi\varphi(\vartheta(\mu,\delta_{0})^{2})
\end{equation*}
for every $(x,\mu)\in\R^{m}\times\mathscr{P}$ a.e.~on $[t_{0},\infty)$.
\item\label{con:det.2} $b(s,\cdot,\mu)$ and $\sigma(s,\cdot,\mu)$ are continuous for any $\mu\in\mathscr{P}$ for a.e.~$s\in [t_{0},\infty)$ and for each $n\in\N$ there is $c_{n} > 0$ such that
\begin{equation*}
|b(s,x,\mu)|\vee|\sigma(s,x,\mu)|\leq c_{n}
\end{equation*}
for any $(x,\mu)\in\R^{m}\times\mathscr{P}(\R^{m})$ with $|x|\vee \vartheta(\mu,\delta_{0})\leq n$ for a.e.~$s\in [t_{0},t_{0}+n]$.
\item\label{con:det.3} For every $n\in\N$ there are $\eta_{n}\in\mathscr{L}_{loc}^{1}(\R_{+})$ and a concave function $\rho_{n}\in\mathrm{R}_{c}$ such that $\int_{0}^{1}\rho_{n}(v)^{-1}\,dv = \infty$ and
\begin{equation*}
2(x-\tilde{x})'\big(b(\cdot,x,\mu) - b(\cdot,\tilde{x},\mu)\big) + |\sigma(\cdot,x,\mu) - \sigma(\cdot,\tilde{x},\mu)|^{2} \leq \eta_{n}\rho_{n}(|x-\tilde{x}|^{2})
\end{equation*} 
for all $x,\tilde{x}\in\R^{m}$ with $|x|\vee|\tilde{x}|\leq n$ and each $\mu\in\mathscr{P}$ a.e.~on $[t_{0},\infty)$.
\end{enumerate}

We write $B_{b,loc}(\mathscr{P})$ for the set of all $\mathscr{P}$-valued Borel measurable maps $\mu$ on $[t_{0},\infty)$ such that $\sup_{s\in [t_{0},t]}\vartheta(\mu(s),\delta_{0}) < \infty$ for any $t\in [t_{0},\infty)$. Then a local weak existence result from~\cite{IkeWat89}, Corollary~\ref{co:pathwise uniqueness} and Lemma~\ref{le:2nd moment growth estimate} allow for a concise analysis of the SDE~\eqref{eq:mu-SDE}.

\begin{Proposition}\label{pr:results on mu-SDE}
For $\mu\in B_{b,loc}(\mathscr{P})$ the following three assertions hold:
\begin{enumerate}[(i)]
\item Under~\eqref{con:det.3}, we have pathwise uniqueness for~\eqref{eq:mu-SDE}.
\item Let~\eqref{con:det.1} and~\eqref{con:det.2} be satisfied. Then there is a weak solution $X$ to~\eqref{eq:mu-SDE} with $\mathscr{L}(X_{t_{0}}) = \mathscr{L}(\xi)$.
\item Assume that~\eqref{con:det.1}-\eqref{con:det.3} are valid. Then~\eqref{eq:mu-SDE} admits a unique strong solution $X^{\xi,\mu}$ such that $X_{t_{0}}^{\xi,\mu} = \xi$ a.s.
\end{enumerate}
\end{Proposition}

Now consider the convex space $B_{b,loc}(\mathscr{P}_{p}(\R^{m}))$ of all $\mathscr{P}_{p}(\R^{m})$-valued Borel measurable maps $\mu$ on $[t_{0},\infty)$ such that $\vartheta_{p}(\mu,\delta_{0})$ is locally bounded, endowed with the topology of local uniform convergence.

Then $B_{b,loc}(\mathscr{P}_{p}(\R^{m}))$ is completely metrisable, as $\vartheta_{p}$ is complete, and a sequence $(\mu_{n})_{n\in\N}$ in this space converges locally uniformly to some $\mu\in B_{b,loc}(\mathscr{P}_{p}(\R^{m}))$ if and only if $\lim_{n\uparrow\infty}\sup_{s\in [t_{0},t]}\vartheta_{p}(\mu_{n},\mu)(s) = 0$ for all $t\in [t_{0},\infty)$.

To deduce a strong solution to~\eqref{eq:deterministic McKean-Vlasov} as local uniform limit of a Picard iteration in $B_{b,loc}(\mathscr{P}_{p}(\R^{m}))$, we replace the spatial Osgood condition~\eqref{con:det.3} on compacts sets by a Lipschitz condition, which implies the former:
\begin{enumerate}[label=(D.\arabic*), ref=D.\arabic*, leftmargin=\widthof{(D.4)} + \labelsep]
\setcounter{enumi}{3}
\item\label{con:det.4} There are $\eta\in\mathscr{L}_{loc}^{1}(\R^{2})$ and $\hat{\eta}\in\mathscr{L}_{loc}^{2}(\R_{+}^{2})$ with $\eta_{2}\geq 0$ such that
\begin{align*}
(x-\tilde{x})'\big(b(\cdot,x,\mu) - b(\cdot,\tilde{x},\tilde{\mu})\big)&\leq |x-\tilde{x}|\big(\eta_{1}|x-\tilde{x}| + \eta_{2}\vartheta(\mu,\tilde{\mu})\big)\\
\quad\text{and}\quad |\sigma(\cdot,x,\mu) - \sigma(\cdot,\tilde{x},\tilde{\mu})|&\leq \hat{\eta}_{1}|x-\tilde{x}| + \hat{\eta}_{2}\vartheta(\mu,\tilde{\mu})
\end{align*}
for all $x,\tilde{x}\in\R^{m}$ and every $\mu,\tilde{\mu}\in\mathscr{P}$ a.e.~on $[t_{0},\infty)$.
\end{enumerate}

Assuming that~\eqref{con:det.4} is satisfied, from which~\eqref{con:4} follows for $(\B,\Sigma) = (b,\sigma)$, $l=2$, $\alpha=(1,0)$ and $\beta=(0,1)$, we define $\gamma_{p}\in\mathscr{L}_{loc}^{1}(\R)$ and $\delta\in\mathscr{L}_{loc}^{1}(\R_{+})$ by
\begin{equation}\label{eq:error coefficients}
\begin{split}
\gamma_{p}&:=p \eta_{1} + (p-1)\eta_{2} + c_{p}\big(p\hat{\eta}_{1}^{2} + 2(p-1)\hat{\eta}_{1}\hat{\eta}_{2} + (p-2)\hat{\eta}_{2}^{2}\big)\\
\text{and}\quad \delta&:=\eta_{2} + 2 c_{p}\big(\hat{\eta}_{1}\hat{\eta}_{2} + \hat{\eta}_{2}^{2}\big).
\end{split}
\end{equation}
So, $\gamma_{p}$ and $\delta$ are defined according to the formulas in~\eqref{eq:specific stability coefficient} and~\eqref{eq:drift stability coefficient}, respectively, for the choice $l=2$, $\alpha=(1,0)$, $\beta = (0,0)$ and $\zeta_{2} = \hat{\zeta}_{2} = 1$, as we will use the $L^{p}$-estimate of Proposition~\ref{pr:specific stability moment estimate} when $(\B,\Sigma) = (b_{\mu},\sigma_{\mu})$ and $(\tilde{\B},\tilde{\Sigma}) = (b_{\tilde{\mu}},\sigma_{\tilde{\mu}})$ for $\mu,\tilde{\mu}\in B_{b,loc}(\mathscr{P})$.

If in addition the following affine growth condition for $(b,\sigma)$ is valid, from which~\eqref{con:det.1} follows, by Young's inequality, as $\alpha+\beta\in [0,1]$, then an estimate in $B_{b,loc}(\mathscr{P}_{p}(\R^{m}))$ for the Picard iteration holds.
\begin{enumerate}[label=(D.\arabic*), ref=D.\arabic*, leftmargin=\widthof{(D.5)} + \labelsep]
\setcounter{enumi}{4}
\item\label{con:det.5} There are $l\in\N$, $\alpha,\beta\in [0,1]^{l}$ with $\alpha + \beta\in [0,1]^{l}$, $\upsilon\in\mathscr{L}_{loc}^{1}(\R^{l})$ and $\hat{\upsilon}\in\mathscr{L}_{loc}^{2}(\R_{+}^{l})$ such that
\begin{equation*}
x'b(\cdot,x,\mu)\leq \sum_{k=1}^{l}\upsilon_{k}|x|^{1+\alpha_{k}}\vartheta(\mu,\delta_{0})^{\beta_{k}}\quad\text{and}\quad |\sigma(\cdot,x,\mu)|\leq \sum_{k=1}^{l}\hat{\upsilon}_{k}|x|^{\alpha_{k}}\vartheta(\mu,\delta_{0})^{\beta_{k}}
\end{equation*}
for every $(x,\mu)\in\R^{m}\times\mathscr{P}$ a.e.~on $[t_{0},\infty)$.
\end{enumerate}

Since the previous condition ensures that~\eqref{con:12} is satisfied for $(\B,\Sigma) = (b,\sigma)$, we may utilise the coefficients  $f_{p,\mathscr{P}}\in\mathscr{L}_{loc}^{1}(\R)$ and $g_{p,\mathscr{P}}\in\mathscr{L}_{loc}^{1}(\R_{+})$ given by~\eqref{eq:growth coefficient} and~\eqref{eq:drift growth coefficient} for the choice $\kappa_{k} = \hat{\kappa}_{k}= 1$ for all $k\in\{1,\dots,l\}$. That is,
\begin{equation}\label{eq:growth coefficient 2}
\begin{split}
f_{p,\mathscr{P}} &= p\sum_{\substack{k=1,\\ \alpha_{k}=1}}^{l}\upsilon_{k} +  \sum_{\substack{k=1,\\ \alpha_{k} < 1}}^{l}(p-1+\alpha_{k} + \beta_{k})c_{p,\mathscr{P}}^{\beta_{k}}\upsilon_{k}^{+}\\
&\quad + c_{p}\sum_{j,k=1}^{l} (p-2+\alpha_{j} + \alpha_{k} + \beta_{j} + \beta_{k})c_{p,\mathscr{P}}^{\beta_{j}+\beta_{k}}\hat{\upsilon}_{j}\hat{\upsilon}_{k}
\end{split}
\end{equation}
and
\begin{equation}\label{eq:drift growth coefficient 2}
g_{p,\mathscr{P}} = \sum_{k=1}^{l}(1-\alpha_{k}-\beta_{k})c_{p,\mathscr{P}}^{\beta_{k}}\upsilon_{k}^{+} + c_{p}\sum_{j,k=1}^{l}(2-\alpha_{j}-\alpha_{k}-\beta_{j}-\beta_{k})c_{p,\mathscr{P}}^{\beta_{j}+\beta_{k}}\hat{\upsilon}_{j}\hat{\upsilon}_{k}.
\end{equation}
As a result, we obtain a \emph{unique strong solution} to~\eqref{eq:deterministic McKean-Vlasov} with initial value condition $\xi$ together with a \emph{semi-explicit error estimate} and an \emph{explicit growth estimate}.

\begin{Theorem}\label{th:strong existence}
Assume that~\eqref{con:det.1} and $p=2$ or~\eqref{con:det.5} hold and let~\eqref{con:det.2},~\eqref{con:det.4} be valid. Further, let $\mathscr{P}_{p}(\R^{m})\subset\mathscr{P}$, $\mu_{0}\in B_{b,loc}(\mathscr{P})$, $E[|\xi|^{p}] < \infty$ and $\Theta:[t_{0},\infty)\times\mathscr{P}^{2}\rightarrow\R_{+}$ be given by
\begin{equation}\label{eq:local integrability functional}
\Theta(\cdot,\mu,\tilde{\mu}):= \eta_{2}\vartheta(\mu,\tilde{\mu}) + \hat{\eta}_{2}^{2}\vartheta(\mu,\tilde{\mu})^{2}.
\end{equation}
\begin{enumerate}[(i)]
\item We have pathwise uniqueness for~\eqref{eq:deterministic McKean-Vlasov} with respect to $\Theta$ and there is a unique strong solution $X^{\xi}$ such that $X_{t_{0}}^{\xi} = \xi$ a.s.~and $E[|X^{\xi}|^{p}]$ is locally bounded.
\item The map $[t_{0},\infty)\rightarrow\mathscr{P}_{p}(\R^{m})$, $t\mapsto\mathscr{L}(X_{t}^{\xi})$ is the local uniform limit of the sequence $(\mu_{n})_{n\in\N}$ recursively defined via $\mu_{n}:=\mathscr{L}(X^{\xi,\mu_{n-1}})$ and
\begin{equation}\label{eq:error estimate}
\sup_{s\in [t_{0},t]}\vartheta_{p}(\mu_{n}(s),\mathscr{L}(X_{s}^{\xi})) \leq \Delta(t)\sum_{i=n}^{\infty}\bigg(\frac{c_{p,\mathscr{P}}^{i}}{i!}\bigg)^{\frac{1}{p}}\bigg(\int_{t_{0}}^{t}e^{\int_{s}^{t}\gamma_{p}^{+}(\tilde{s})\,d\tilde{s}}\delta(s)\,ds\bigg)^{\frac{i}{p}}
\end{equation}
for any $t\in [t_{0},\infty)$ with $\Delta(t):=\sup_{s\in [t_{0},t]}c_{p,\mathscr{P}}^{-1}\vartheta(\mathscr{L}(X^{\xi,\mu_{0}}),\mu_{0})(s)$.
\item Suppose that~\eqref{con:det.5} is satisfied. If $\mu_{0}$ lies in the closed and convex set $M_{p}$ of all $\mu\in B_{b,loc}(\mathscr{P}_{p}(\R^{m}))$ such that
\begin{equation}\label{eq:growth estimate}
\vartheta_{p}(\mu(t),\delta_{0})^{p} \leq e^{\int_{t_{0}}^{t}f_{p,\mathscr{P}}(s)\,ds}E\big[|\xi|^{p}\big] + \int_{t_{0}}^{t}e^{\int_{s}^{t}f_{p,\mathscr{P}}(\tilde{s})\,d\tilde{s}}g_{p,\mathscr{P}}(s)\,ds
\end{equation}
for any $t\in [t_{0},\infty)$, then so does $\mu_{n}$ for each $n\in\N$.
\end{enumerate}
\end{Theorem}

\begin{Remark}
While the choice $\mu_{0} = \delta_{0}$ yields $\Delta(t)\leq \sup_{s\in [t_{0},t]} E[|X_{s}^{\xi,\delta_{0}}|^{p}]^{1/p}$ for any $t\in [t_{0},\infty)$, we have $\mu_{n} = \mu_{0}$ for all $n\in\N$ whenever $\mu_{0} = \mathscr{L}(X^{\xi})$.
\end{Remark}

\begin{Example}\label{ex:deterministic integral maps}
Let $b_{0}$ and $\sigma_{0}$ be Borel measurable maps on $[t_{0},\infty)\times\R^{m}\times\R^{m}$ with values in $\R^{m}$ and $\R^{m\times d}$, respectively, such that $b_{0}(s,x,\cdot)$, $\sigma_{0}(s,x,\cdot)$ are $\mu$-integrable,
\begin{equation*}\label{eq:integral maps}
b(s,x,\mu) = \int_{\R^{m}}b_{0}(s,x,y)\,\mu(dy)\quad\text{and}\quad \sigma(s,x,\mu) = \int_{\R^{m}}\sigma_{0}(s,x,y)\,\mu(dy)
\end{equation*}
for any $(s,x,\mu)\in [t_{0},\infty)\times\R^{m}\times\mathscr{P}$. Further, let $\mathscr{P}\subset\mathscr{P}_{1}(\R^{m})$ and $\vartheta\geq\vartheta_{1}$, which allows the choice $(\mathscr{P},\vartheta) = (\mathscr{P}_{p}(\R^{m}),\vartheta_{p})$. Then the following three assertions hold:
\begin{enumerate}[(1)]
\item If there are $\eta\in\mathscr{L}_{loc}^{1}(\R^{2})$ and $\hat{\eta}\in\mathscr{L}_{loc}^{2}(\R_{+}^{2})$ with $\eta_{2}\geq 0$ such that
\begin{align*}
(x-\tilde{x})'\big(b_{0}(\cdot,x,y) - b_{0}(\cdot,\tilde{x},\tilde{y})\big) &\leq |x-\tilde{x}|\big(\eta_{1}|x-\tilde{x}| + \eta_{2}|y-\tilde{y}|\big)\\
\text{and}\quad|\sigma_{0}(\cdot,x,y) - \sigma_{0}(\cdot,\tilde{x},\tilde{y})| &\leq \hat{\eta}_{1}|x-\tilde{x}| + \hat{\eta}_{2}|y-\tilde{y}|
\end{align*}
for every $x,\tilde{x},y,\tilde{y}\in\R^{m}$ a.e.~on $[t_{0},\infty)$, then~\eqref{con:det.4} is satisfied.
\item Given $l\in\N$ and $\alpha,\beta\in [0,1]^{l}$ with $\alpha + \beta\in [0,1]^{l}$, suppose that there are $\upsilon\in\mathscr{L}_{loc}^{1}(\R^{l})$ and $\hat{\upsilon}\in\mathscr{L}_{loc}^{2}(\R_{+}^{l})$ such that
\begin{equation*}
x'b_{0}(\cdot,x,y)\leq \sum_{k=1}^{l} \upsilon_{k}|x|^{1+\alpha_{k}}|y|^{\beta_{k}}\quad\text{and}\quad |\sigma_{0}(\cdot,x,y)|\leq \sum_{k=1}^{l} \hat{\upsilon}_{k}|x|^{\alpha_{k}}|y|^{\beta_{k}}
\end{equation*}
for any $x,y\in\R^{m}$ a.e.~on $[t_{0},\infty)$. Then $\sigma_{0}(s,x,\cdot)$ is automatically $\mu$-integrable for all $(x,\mu)\in\R^{m}\times\mathscr{P}$ for a.e.~$s\in [t_{0},\infty)$ and~\eqref{con:det.5} follows.
\item Let $b_{0}(s,\cdot,y)$ and $\sigma_{0}(s,\cdot,y)$ be continuous for all $y\in\R^{m}$ for a.e.~$s\in [t_{0},\infty)$ and assume that for each $n\in\N$ there is $c_{n} > 0$ such that
\begin{equation*}
|b_{0}(s,x,y)| \leq c_{n}(1 + |y|)\quad\text{for all $x,y\in\R^{m}$}
\end{equation*}
with $|x|\leq n$ for a.e.~$s\in [t_{0},t_{0}+n]$. Further, let the estimate for $\sigma_{0}$ in (2) be valid such that $\hat{\upsilon}$ is locally bounded. Then~\eqref{con:det.2} is valid.
\end{enumerate}

For instance, let $N\in\N$, $\hat{b}\in\mathscr{L}_{loc}^{1}(\R_{+}^{N})$ and $\hat{c}\in\mathscr{L}_{loc}^{1}(\R^{N})$ and suppose that $\hat{\alpha}\in ]0,\infty[^{N}$ and $f_{1},\dots,f_{N}$ are $\R^{m}$-valued Lipschitz continuous maps on $\R^{m}\times\R^{m}$ such that
\begin{equation*}
b_{0}(\cdot,x,y) = - x\bigg(\hat{b}_{1}|x|^{\hat{\alpha}_{1} - 1} + \cdots  + \hat{b}_{N} |x|^{\hat{\alpha}_{N} - 1}\bigg) + \hat{c}_{1}f_{1}(x,y) + \cdots + \hat{c}_{N}f_{N}(x,y) 
\end{equation*}
for any $x,y\in\R^{m}$ with $x\neq 0$ and $b_{0}(\cdot,0,\cdot) = \hat{c}_{1}f_{1}(0,\cdot) + \cdots + \hat{c}_{N}f_{N}(0,\cdot)$. Then all the conditions for $b_{0}$ in (1) and (2) are satisfied and the conditions for $b_{0}$ in (3) hold if $\hat{b}$ and $\hat{c}$ are in fact locally bounded. In particular,
\begin{equation*}
\text{if $f_{1},\dots,f_{N}$ are independent of the first variable $x\in\R^{m}$,}
\end{equation*}
then we may take $\eta = 0$, $l=1$ and $(\alpha,\beta) = (0,1)$ in (1) and (2). In the general case, each statement of Theorem~\ref{th:strong existence} holds as soon as $\mathscr{P}_{p}(\R^{m})\subset\mathscr{P}$ and the requirements in (1)-(3) are met. Thereby, the coefficients
\begin{equation*}
\gamma_{p},\quad \delta,\quad f_{p,\mathscr{P}},\quad g_{p,\mathscr{P}}\quad\text{and}\quad\Theta
\end{equation*}
remain exactly as specified in the formulas~\eqref{eq:error coefficients}-\eqref{eq:local integrability functional}.
\end{Example}

\section{Moment and pathwise asymptotic estimations for random It{\^o} processes}\label{se:4}

\subsection{Auxiliary moment bounds}\label{se:4.1}

From now on, let $\hat{\B}\in\mathscr{S}(\R^{m})$ and $\hat{\Sigma}\in\mathscr{S}(\R^{m\times d})$ satisfy $\int_{t_{0}}^{t}|\hat{\B}_{s}| + |\hat{\Sigma}_{s}|^{2}\,ds < \infty$ for all $t\in [t_{0},\infty)$ and $Y$ be a random It{\^o} process with drift $\hat{\B}$ and diffusion $\hat{\Sigma}$, as introduced in~\cite{KalMeyPro21}[Section 4.1]. That is, $Y$ is an $\R^{m}$-valued adapted continuous process such that
\begin{equation*}
Y_{t} = Y_{t_{0}} + \int_{t_{0}}^{t}\hat{\B}_{s}\,ds + \int_{t_{0}}^{t}\hat{\Sigma}_{s}\,dW_{s}\quad\text{for all $t\in [t_{0},\infty)$ a.s.}
\end{equation*}
Our aim is to give quantitative $L^{p}$-estimates for $Y$ in a random seminorm when $p\in [2,\infty)$. For this purpose, we fix $\hat{m}\in\N$ and let $V$ be an $\R^{\hat{m}\times m}$-valued adapted locally absolutely continuous process. For $A\in\R^{\hat{m}\times m}$ we define a seminorm on $\R^{m}$ by
\begin{equation}\label{eq:seminorm}
|x|_{A} := |Ax|
\end{equation}
and note that the induced matrix seminorm defined via $|B|_{A} :=\sup_{x\in\R^{d}\setminus\{0\}} |Bx|_{A}$ satisfies $|B|_{A} = |AB|$ for any $B\in\R^{m\times d}$. In particular, the random (matrix) seminorm $|\cdot|_{V}$ is an adapted continuous process.

The continuous map $\psi_{d}:\R^{d\times \hat{m}}\times\R^{\hat{m}}\setminus\{0\}\rightarrow\R^{d\times d}$ given by $\psi_{d}(A,x) := Ax(Ax)'/|x|^{2}$ satisfies $|\psi_{d}(A,x)| = |x|_{A}^{2}/|x|^{2} \leq |A|^{2}$ for any $A\in\R^{d\times\hat{m}}$ and all $x\in\R^{\hat{m}}\setminus\{0\}$. Thus, the map
\begin{equation*}
\R^{d\times\hat{m}}\times\R^{\hat{m}}\rightarrow\R^{d\times d},\quad (A,x)\mapsto (p-2)|x|_{A}^{p-2}\psi_{d}(A,x)
\end{equation*}
is continuous once we set $\psi_{d}(\cdot,0):=0$. Bearing these considerations in mind, we consider the following representation.

\begin{Lemma}\label{le:auxiliary radial p-th power identity}
Any $u\in C([t_{0},\infty))$ that is locally of bounded variation satisfies
\begin{align*}
&u(t)|Y_{t}|_{V_{t}}^{p} = u(t_{0})|Y_{t_{0}}|_{V_{t_{0}}}^{p} + \int_{t_{0}}^{t}|Y_{s}|_{V_{s}}^{p}\,du(s) + p\int_{t_{0}}^{t}u(s)|Y_{s}|_{V_{s}}^{p-2}(V_{s}Y_{s})'V_{s}\hat{\Sigma}_{s}\,dW_{s}\\
&\quad + p\int_{t_{0}}^{t}u(s)|Y_{s}|_{V_{s}}^{p-2}\bigg((V_{s}Y_{s})'\big(\dot{V}_{s}Y_{s} + V_{s}\hat{\B}_{s}\big) + \frac{1}{2}|\hat{\Sigma}_{s}|_{V_{s}}^{2} + \bigg(\frac{p}{2}-1\bigg)|\psi_{d}((V_{s}\hat{\Sigma}_{s})',V_{s}Y_{s})|\bigg)\,ds
\end{align*}
for any $t\in [t_{0},\infty)$ a.s.
\end{Lemma}

\begin{proof}
The function $\varphi:(\R^{\hat{m}})^{m}\times\R^{m}\rightarrow\R_{+}$ defined by $\varphi(a_{1},\dots,a_{m},x) := |\sum_{j=1}^{m}a_{j}x_{j}|^{p}$ is twice continuously differentiable with first-order derivatives with respect to the $j$-th and the last variable given by
\begin{equation*}
D_{a_{j}}\varphi(a_{1},\dots,a_{m},x) = p|x|_{A}^{p-2}(Ax)'x_{j}\quad\text{and}\quad D_{x}\varphi(a_{1},\dots,a_{m},x) = p|x|_{A}^{p-2}(Ax)'A
\end{equation*}
for each $j\in\{1,\dots,m\}$, all $a_{1},\dots,a_{m}\in\R^{\hat{m}}$ and any $x\in\R^{m}$, where $A\in\R^{\hat{m}\times m}$ is of the form $A=(a_{1},\dots,a_{m})$. Its second-order derivative relative to the last variable can be written in the form
\begin{equation*}
D_{x}^{2}\varphi(a_{1},\dots,a_{m},x) = p|x|_{A}^{p-2}\big(A'A + (p-2)\psi_{m}(A',Ax)\big),
\end{equation*}
which in turn gives us that
\begin{equation*}
\mathrm{tr}(D_{x}^{2}\varphi(a_{1},\dots,a_{m},x)BB') = p|x|_{A}^{p-2}\big(|B|_{A}^{2} + (p-2)|\psi_{d}((AB)',Ax)|\big)
\end{equation*}
for any $B\in\R^{m\times d}$. Moreover, from It{\^o}'s formula we know that
\begin{align*}
|Y_{t}|_{V_{t}}^{p} &- |Y_{t_{0}}|_{V_{t_{0}}}^{p} = \sum_{j=1}^{m}\int_{t_{0}}^{t}D_{a_{j}}\varphi(\null_{1}V_{s},\dots,\null_{m}V_{s},Y_{s})\,d\null_{j}V_{s}\\
&\quad + \int_{t_{0}}^{t}D_{x}\varphi(\null_{1}V_{s},\dots,\null_{m}V_{s},Y_{s})\,dY_{s} + \frac{1}{2}\int_{t_{0}}^{t}\mathrm{tr}\big(D_{x}^{2}\varphi(\null_{1}V_{s},\dots,\null_{m}V_{s},Y_{s})\hat{\Sigma}_{s}\hat{\Sigma}_{s}'\big)\,ds
\end{align*}
for each $t\in [t_{0},\infty)$ a.s., where $\null_{j}V$ stands for the $j$-th column of the process $V$ for all $j\in\{1,\dots,m\}$. Hence, It{\^o}'s product rule completes the verification.
\end{proof}

We recall the constant $c_{p}=(p-1)/2$ and employ the just considered bound on the map $\psi_{d}$ to get an auxiliary $L^{p}$-estimate.

\begin{Lemma}\label{le:auxiliary p-th moment estimate}
Let $E[|Y_{t_{0}}|_{V_{t_{0}}}^{p}] < \infty$ and $\tau$ be a stopping time. Assume there is $Z\in\mathscr{S}(\R)$ such that $\int_{t_{0}}^{t}|Z_{s}|\,ds < \infty$ for any $t\in [t_{0},\infty)$ and
\begin{equation*}
(V_{s}Y_{s})'\big(\dot{V}_{s}Y_{s} + V_{s}\hat{\B}_{s}) + c_{p}|\hat{\Sigma}_{s}|_{V_{s}}^{2}\leq Z_{s}
\end{equation*}
for a.e.~$s\in [t_{0},\infty)$ with $s < \tau$ a.s. If $u:[t_{0},\infty)\rightarrow\R_{+}$ is locally absolutely continuous and $E[\int_{t_{0}}^{t\wedge\tau}|Y_{s}|_{V_{s}}^{p-2}|\dot{u}(s)|Y_{s}|_{V_{s}}^{2} + u(s)p Z_{s}|\,ds] < \infty$, then
\begin{equation*}
E\big[u(t\wedge\tau)|Y_{t}^{\tau}|^{p}_{V_{t}^{\tau}}\big] \leq u(t_{0})E\big[|Y_{t_{0}}|_{V_{t_{0}}}^{p}\big] + E\bigg[\int_{t_{0}}^{t\wedge\tau}|Y_{s}|_{V_{s}}^{p-2}\big(\dot{u}(s)|Y_{s}|_{V_{s}}^{2} + u(s)p Z_{s}\big)\,ds\bigg]
\end{equation*}
for every $t\in [t_{0},\infty)$.
\end{Lemma}

\begin{proof}
By the preceding lemma, the claimed inequality holds when $\tau$ is replaced by the stopping time
\begin{equation}\label{eq:hitting time for Ito processes}
\tau_{k}:=\inf\bigg\{t\in [t_{0},\infty)\,\bigg|\, |Y_{t}|_{V_{t}}\geq k\text { or } \int_{t_{0}}^{t}|Z_{s}| + |\hat{\Sigma}_{s}|_{V_{s}}^{2}\,ds\geq k\bigg\}\wedge\tau
\end{equation}
for each fixed $k\in\N$, because from $|(V\hat{\Sigma})'(VY)|\leq |\hat{\Sigma}|_{V} |Y|_{V}$ we readily deduce that $\int_{t_{0}}^{\cdot\wedge\tau_{k}}u(s)|Y_{s}|_{V_{s}}^{p-2}(V_{s}Y_{s})'V_{s}\hat{\Sigma}_{s}\,dW_{s}$ is a square-integrable martingale. Hence, Fatou's lemma and dominated convergence give the asserted bound, since $\sup_{k\in\N}\tau_{k} = \tau.$
\end{proof}

Now we apply a Burkholder-Davis-Gundy inequality for stochastic integrals driven by $W$ from~\cite{Mao08}[Theorem 7.3]. For $q\in [2,\infty)$ set $\overline{w}_{q}:=(q^{q+1}/(2(q-1)^{q-1}))^{q/2}$, if $q > 2$, and $\overline{w}_{q}:=4$, if $q = 2$. Then
\begin{equation}\label{eq:BDG-inequality}
E\bigg[\sup_{\tilde{s}\in [t_{0},t]}\bigg|\int_{t_{0}}^{\tilde{s}}X_{s}\,dW_{s}\bigg|^{q}\bigg] \leq \overline{w}_{q}E\bigg[\bigg(\int_{t_{0}}^{t}|X_{s}|^{2}\,ds\bigg)^{\frac{q}{2}}\bigg]
\end{equation}
for every $X\in\mathscr{S}(\R^{m\times d})$ and all $t\in [t_{0},\infty)$ satisfying $\int_{t_{0}}^{t}|X_{s}|^{2}\,ds < \infty$. The next result is an auxiliary moment estimate in a supremum seminorm.

\begin{Proposition}\label{pr:auxiliary p-th moment estimate in supremum seminorm}
Let $q\in [1,\infty)$, $\tau$ be a stopping time and $Z\in\mathscr{S}(\R)$ be such that $\int_{t_{0}}^{t}|Z_{s}|\,ds < \infty$ for each $t\in [t_{0},\infty)$ and
\begin{equation*}
(V_{s}Y_{s})'\big(\dot{V}_{s}Y_{s} + V_{s}\hat{\B}_{s}\big) \leq Z_{s}
\end{equation*}
for a.e.~$s\in [t_{0},\infty)$ with $s < \tau$ a.s. Then $\hat{Z}:= Z + c_{p}|\hat{\Sigma}|_{V}$ and any locally absolutely continuous function $u:[t_{0},\infty)\rightarrow\R_{+}$ satisfy
\begin{align*}
E\bigg[\bigg(\sup_{s\in [t_{0},t]} u(s\wedge\tau)|Y_{s}^{\tau}|_{V_{s}^{\tau}}^{p} &- u(t_{0})|Y_{t_{0}}|_{V_{t_{0}}}^{p}\bigg)^{q}\bigg]^{\frac{1}{q}}\\
&\leq E\bigg[\bigg(\int_{t_{0}}^{t\wedge\tau}|Y_{s}|_{V_{s}}^{p-2}\big(\dot{u}(s)|Y_{s}|_{V_{s}}^{2} + u(s)p\hat{Z}_{s}\big)^{+}\,ds\bigg)^{q}\bigg]^{\frac{1}{q}}\\
&\quad +  pE\bigg[\overline{w}_{q_{0}}\bigg(\int_{t_{0}}^{t\wedge\tau}u(s)^{2}|Y_{s}|_{V_{s}}^{2p-2}|\hat{\Sigma}_{s}|_{V_{s}}^{2}\,ds\bigg)^{\frac{q_{0}}{2}}\bigg]^{\frac{1}{q_{0}}}
\end{align*}
for each $t\in [t_{0},\infty)$ with $q_{0}:= q\vee 2$.
\end{Proposition}

\begin{proof}
Because $\sup_{\tilde{s}\in [t_{0},t]}\int_{t_{0}}^{\tilde{s}}\kappa(s)\,ds \leq \int_{t_{0}}^{t}\kappa^{+}(s)\,ds$ for every $\kappa\in\mathscr{L}_{loc}^{1}(\R)$, we infer from Lemma~\ref{le:auxiliary radial p-th power identity} that the stopping time~\eqref{eq:hitting time for Ito processes} satisfies
\begin{equation}\label{eq:auxiliary supremum radial p-th power identity}
\begin{split}
\sup_{s\in [t_{0},t]}u(s\wedge\tau_{k})|Y_{s}^{\tau_{k}}|_{V_{s}^{\tau_{k}}}^{p} &\leq u(t_{0})|Y_{t_{0}}|_{V_{t_{0}}}^{p} + \int_{t_{0}}^{t\wedge\tau_{k}}|Y_{s}|_{V_{s}}^{p-2}\big(\dot{u}(s)|Y_{s}|_{V_{s}}^{2} + u(s)p\hat{Z}_{s}\big)^{+}\,ds\\
&\quad + \sup_{s\in [t_{0},t]} I_{s}^{\tau_{k}}\quad\text{a.s.}
\end{split}
\end{equation}
for any fixed $k\in\N$ and $t\in [t_{0},\infty)$, where $I$ denotes a continuous local martingale with $I_{t_{0}} = 0$ that is indistinguishable from the stochastic integral
\begin{equation*}
p\int_{t_{0}}^{\cdot}u(s)|Y_{s}|_{V_{s}}^{p-2}(V_{s}Y_{s})'V_{s}\hat{\Sigma}_{s}\,dW_{s}.
\end{equation*}
Thus, H\oe lder's inequality,~\eqref{eq:BDG-inequality} and the fact that $|(V\hat{\Sigma})'(VY)| \leq |\hat{\Sigma}|_{V}|Y|_{V}$ yield that
\begin{equation}\label{eq:auxiliary stochastic integral moment estimate}
\overline{w}_{q_{0}}^{-1}E\bigg[\sup_{s\in [t_{0},t]} |I_{s}^{\tau_{k}}|^{q}\bigg]^{\frac{q_{0}}{q}} \leq E\bigg[\bigg(\int_{t_{0}}^{t\wedge\tau_{k}}p^{2}u(s)^{2}|Y_{s}|_{V_{s}}^{2p-2}|\hat{\Sigma}_{s}|_{V_{s}}^{2}\,ds\bigg)^{\frac{q_{0}}{2}}\bigg].
\end{equation}
For this reason, the claimed inequality follows when $\tau$ is replaced by $\tau_{k}$ from~\eqref{eq:auxiliary supremum radial p-th power identity},~\eqref{eq:auxiliary stochastic integral moment estimate} and the triangle inequality in the $L^{q}$-norm. Since $\sup_{k\in\N} \tau_{k} = \tau$, monotone convergence completes the proof.
\end{proof}

\subsection{Quantitative moment estimates}\label{se:4.2}

First, we derive an $L^{2}$-estimate from Section~\ref{se:4.1} and Bihari's inequality. For this purpose, let $l\in\N$ and consider the following assumption on the It{\^o} process $Y$:
\begin{enumerate}[label=(A.\arabic*), ref=A.\arabic*, leftmargin=\widthof{(A.1)} + \labelsep]
\item\label{con:assumption 1} There are $\alpha\in (0,1]^{l}$, $\kappa\in\mathscr{S}_{loc}^{1}(\R_{+})$, a measurable map $\theta:[t_{0},\infty)\rightarrow\R_{+}^{l}$, $c\in\R_{+}^{l}$ and for each $k\in\{1,\dots,l\}$ there are
\begin{equation*}
\rho_{k},\varrho_{k}\in\mathrm{R}_{c},\quad\null_{k}\eta\in\mathscr{S}_{loc}^{\frac{1}{1-\alpha_{k}}}(\R_{+})\quad\text{and}\quad\null_{k}\lambda\in\mathscr{S}_{loc}^{1}(\R_{+})
\end{equation*}
such that $\rho_{k}^{1/\alpha_{k}}$ is concave, $\varrho_{k}$ is increasing, $\theta_{k}(s) \leq c_{k}E\big[|Y_{s}|_{V_{s}}^{2}\big]^{1/2}$ for a.e.~$s\in [t_{0},\infty)$ with $E[\null_{k}\lambda_{s}] > 0$ and
\begin{equation*}
2(VY)'\big(\dot{V}Y + V\hat{\B}\big) + |\hat{\Sigma}|_{V}^{2}\leq \kappa + \sum_{k=1}^{l}\null_{k}\eta\rho_{k}(|Y|_{V}^{2}) + \null_{k}\lambda\varrho_{k}\circ\theta_{k}^{2}
\end{equation*}
a.e.~on $[t_{0},\infty)$ a.s.
\end{enumerate}
Whenever~\eqref{con:assumption 1} holds and $\beta\in (0,1]^{l}$, then the functionals in~\eqref{eq:essential functionals} gives rise to two functions $\gamma,\delta\in\mathscr{L}_{loc}^{1}(\R_{+})$ defined by
\begin{equation*}
\gamma := \sum_{k=1}^{l}\alpha_{k}\big[\null_{k}\eta\big]_{\frac{1}{1-\alpha_{k}}} + \beta_{k}E\big[\null_{k}\lambda\big]\quad\text{and}\quad \delta := \sum_{k=1}^{l}(1-\alpha_{k})\big[\null_{k}\eta\big]_{\frac{1}{1-\alpha_{k}}} + (1-\beta_{k})E\big[\null_{k}\lambda\big].
\end{equation*}
For each $\rho\in\mathrm{R}_{c}$ let us recall the domain $D_{\rho}$ and the functions $\Phi_{\rho}\in C^{1}((0,\infty))$ and $\Psi_{\rho}\in C(D_{\rho},\R_{+})$ given by~\eqref{eq:rho-function 1} and~\eqref{eq:rho-function 2}, which allow for a general bound.

\begin{Proposition}\label{pr:abstract second moment estimate}
Let~\eqref{con:assumption 1} hold, $E[|Y_{t_{0}}|_{V_{t_{0}}}^{2}] < \infty$, $\sum_{k=1}^{l}E[\null_{k}\lambda]\varrho_{k}\circ\theta_{k}^{2}$ be locally integrable and $\rho_{0},\varrho_{0}\in C(\R_{+})$ be given by
\begin{equation*}
\rho_{0}(v) := \max_{k\in\{1,\dots,l\}}\rho_{k}(v)^{\frac{1}{\alpha_{k}}}\quad\text{and}\quad\varrho_{0}(v) := \rho_{0}(v)\vee\max_{k\in\{1,\dots,l\}}\varrho_{k}(c_{k}^{2}v)^{\frac{1}{\beta_{k}}}.
\end{equation*}
If $\Phi_{\rho_{0}}(\infty) = \infty$ or $\sum_{k=1}^{l}E[\null_{k}\eta\rho_{k}(|Y|_{V}^{2})]$ is locally integrable, then $E[|Y|_{V}^{2}]$ is locally bounded and
\begin{equation*}
\sup_{s\in [t_{0},t]} E\big[|Y_{s}|_{V_{s}}^{2}\big] \leq \Psi_{\varrho_{0}}\bigg(E\big[|Y_{t_{0}}|_{V_{t_{0}}}^{2}\big] + \int_{t_{0}}^{t}E[\kappa_{s}] + \delta(s)\,ds,\int_{t_{0}}^{t} \gamma(s)\,ds\bigg)
\end{equation*}
for all $t\in [t_{0},t_{0}^{+})$, where $t_{0}^{+}$ stands for the supremum over all $t\in [t_{0},\infty)$ for which
\begin{equation*}
\bigg(E\big[|Y_{t_{0}}|_{V_{t_{0}}}^{2}\big] + \int_{t_{0}}^{t}E[\kappa_{s}] + \delta(s)\,ds,\int_{t_{0}}^{t} \gamma(s)\,ds\bigg)\in D_{\varrho_{0}}.
\end{equation*}
\end{Proposition}

\begin{proof}
We take the stopping time $\tau_{n}:=\inf\{t\in [t_{0},\infty)\,|\, |Y_{t}|_{V_{t}}\geq n\}$ for given $n\in\N$, define $\hat{\kappa} := E[\kappa] + \sum_{k=1}^{l}E[\null_{k}\lambda]\varrho_{k}\circ\theta_{k}^{2}$ and observe that
\begin{equation}\label{eq:abstract second moment auxiliary estimate 1}
E\big[|Y_{t}^{\tau_{n}}|_{V_{t}^{\tau_{n}}}^{2}\big] \leq E\big[|Y_{t_{0}}|_{V_{t_{0}}}^{2}\big] + \int_{t_{0}}^{t}\hat{\kappa}(s) + \sum_{k=1}^{l}E\big[\null_{k}\eta_{s}\rho_{k}(|Y_{s}|_{V_{s}}^{2})\mathbbm{1}_{\{\tau_{n} > s\}}\big]\,ds
\end{equation}
for fixed $t\in [t_{0},\infty)$, according to Lemma~\ref{le:auxiliary p-th moment estimate}. Moreover, for any stopping time $\tau$ for which $E[|Y^{\tau}|_{V^{\tau}}^{2}] < \infty$ we infer from~\eqref{eq:essential inequality} and the concavity of $\rho^{1/\alpha_{k}}$ that
\begin{equation}\label{eq:abstract second moment auxiliary estimate 2}
E\big[\null_{k}\eta_{s}\rho_{k}(|Y_{s}|_{V_{s}}^{2})\mathbbm{1}_{\{\tau > s\}}\big] \leq \big[\null_{k}\eta_{s}\big]_{\frac{1}{1-\alpha_{k}}}\big(1-\alpha_{k} + \alpha_{k}\rho_{k}\big(E\big[|Y_{s}^{\tau}\big|_{V_{s}^{\tau}}^{2}\big]\big)^{\frac{1}{\alpha_{k}}}\big)
\end{equation}
for any $s\in [t_{0},t]$ and every $k\in\{1,\dots,l\}$. So, if $\Phi_{\rho_{0}}(\infty) = \infty$, then an application of Bihari's inequality to~\eqref{eq:abstract second moment auxiliary estimate 1} and Fatou's lemma show that $E[|Y|_{V}^{2}]$ is locally bounded.

In this case, we may take $\tau = \infty$ in~\eqref{eq:abstract second moment auxiliary estimate 2} to see that $\sum_{k=1}^{l}E[\null_{k}\eta\rho_{k}(|Y|_{V}^{2})]$ is locally integrable. For this reason, we merely suppose that the latter holds. Then
\begin{equation*}
E\big[|Y_{t}|_{V_{t}}^{2}\big] \leq E\big[|Y_{t_{0}}|_{V_{t_{0}}}^{2}\big] + \int_{t_{0}}^{t}(\hat{\kappa} + \hat{\delta})(s) + \sum_{k=1}^{l}\alpha_{k}\big[\null_{k}\eta_{s}\big]_{\frac{1}{1-\alpha_{k}}}\rho_{k}\big(E\big[|Y_{s}|_{V_{s}}^{2}\big]\big)^{\frac{1}{\alpha_{k}}}\,ds
\end{equation*}
for $\hat{\delta} :=\sum_{k=1}^{l}(1-\alpha_{k})[\null_{k}\eta]_{1/(1-\alpha_{k})}$, by~\eqref{eq:abstract second moment auxiliary estimate 1} and Fatou's lemma. Thus, as in the previous case, the function $E[|Y|_{V}^{2}]$ is locally bounded. Lastly, since Young's inequality gives
\begin{equation*}
E[\null_{k}\lambda]\varrho_{k}\circ\theta_{k}^{2} \leq E\big[\null_{k}\lambda\big](1 - \beta_{k} + \beta_{k}\varrho_{k}\big(c_{k}^{2}E\big[|Y|_{V}^{2}\big]\big)^{\frac{1}{\beta_{k}}}\big)
\end{equation*}
a.e.~on $[t_{0},t]$ for all $k\in\{1,\dots,l\}$, the asserted estimate follows from Bihari's inequality. 
\end{proof}

Next, we seek to give an explicit $L^{p}$-estimate for fixed $p\in [2,\infty)$ and impose an abstract mixed power condition on $Y$ involving the constant $c_{p} = (p-1)/2$:
\begin{enumerate}[label=(A.\arabic*), ref=A.\arabic*, leftmargin=\widthof{(A.2)} + \labelsep]
\setcounter{enumi}{1}
\item\label{con:assumption 2} There are $\alpha,\beta\in [0,2]^{l}$ with $\alpha + \beta\in [0,2]^{l}$, two measurable maps $\zeta,\theta:[t_{0},\infty)\rightarrow\R_{+}^{l}$, $c\in\R_{+}^{l}$ and for any $k\in\{1,\dots,l\}$ there are
\begin{equation*}
 \null_{k}\eta\in\mathscr{S}_{loc}^{\frac{p}{2-\alpha_{k}}}(\R)
\end{equation*}
such that $(VY)'(\dot{V}Y + V\hat{\B}) + c_{p}|\hat{\Sigma}|_{V}^{2} \leq \sum_{k=1}^{l}\zeta_{k}\,\null_{k}\eta|Y|_{V}^{\alpha_{k}}\theta_{k}^{\beta_{k}}$ a.e.~on $[t_{0},\infty)$ a.s. Further, for any $k\in\{1,\dots,l\}$ we have $\zeta_{k}=1$, if $\alpha_{k} + \beta_{k} = 2$, the function
\begin{equation*}
\zeta_{k}^{\frac{p}{2-\alpha_{k}-\beta_{k}}}[\null_{k}\eta]_{\frac{p}{2-\alpha_{k}}}
\end{equation*}
is locally integrable and from $\beta_{k} > 0$ it follows that $\theta_{k}(s) \leq c_{k}E[|Y_{s}|_{V_{s}}^{p}]^{1/p}$ for a.e.~$s\in [t_{0},\infty)$ with $\zeta_{k}(s)[\null_{k}\eta_{s}]_{p/(2-\alpha_{k})} > 0$.
\end{enumerate}

\begin{Remark}
The preceding condition implies that the product $\zeta_{k}[\null_{k}\eta]_{p/(2-\alpha_{k})}$ is locally integrable for all $k\in\{1,\dots,l\}$, since Young's inequality entails that
\begin{equation*}
p\zeta_{k} \leq p-2+\alpha_{k}+\beta_{k} + (2-\alpha_{k}-\beta_{k})\zeta_{k}^{\frac{p}{2-\alpha_{k}-\beta_{k}}}.
\end{equation*} 
\end{Remark}

If~\eqref{con:assumption 2} is satisfied, we define $\gamma_{p,c}\in\mathscr{L}_{loc}^{1}(\R)$ and $\delta_{p,c}\in\mathscr{L}_{loc}^{1}(\R_{+})$ by
\begin{equation}\label{eq:stability coefficient}
\gamma_{p,c} := \sum_{k=1}^{l}(p-2 +\alpha_{k} + \beta_{k})c_{k}^{\beta_{k}}\big[\null_{k}\eta\big]_{\frac{p}{2-\alpha_{k}}}
\end{equation}
and
\begin{equation}\label{eq:stability drift coefficient}
\delta_{p,c} := \sum_{k=1}^{l}(2-\alpha_{k} - \beta_{k})c_{k}^{\beta_{k}}\zeta_{k}^{\frac{p}{2-\alpha_{k}-\beta_{k}}}\big[\null_{k}\eta\big]_{\frac{p}{2-\alpha_{k}}},
\end{equation}
which yield an explicit $p$-th moment estimate. As a direct consequence, we can provide sufficient conditions for boundedness and convergence in $L^{p}(\Omega,\mathscr{F},P)$.

\begin{Theorem}\label{th:stability moment estimate}
Let~\eqref{con:assumption 2} hold, $E[|Y_{t_{0}}|_{V_{t_{0}}}^{p}] < \infty$ and $\sum_{k=1,\,\beta_{k} > 0}^{l}\zeta_{k}[\null_{k}\eta]_{p/(2-\alpha_{k})}\theta_{k}^{\beta_{k}}$ be locally integrable. Then
\begin{equation}\label{eq:general p-th moment stability estimate}
E\big[|Y_{t}|_{V_{t}}^{p}\big] \leq e^{\int_{t_{0}}^{t}\gamma_{p,c}(s)\,ds}E\big[|Y_{t_{0}}|_{V_{t_{0}}}^{p}\big] + \int_{t_{0}}^{t}e^{\int_{s}^{t}\gamma_{p,c}(\tilde{s})\,d\tilde{s}}\delta_{p,c}(s)\,ds
\end{equation}
for all $t\in [t_{0},\infty)$. In particular, if $\gamma_{p,c}^{+}$ and $\delta_{p,c}$ are integrable, then $E[|Y|_{V}^{p}]$ is bounded. If additionally $\int_{t_{0}}^{\infty}\gamma_{p,c}^{-}(s)\,ds = \infty$, then $\lim_{t\uparrow\infty} E[|Y_{t}|_{V_{t}}^{p}] = 0$.
\end{Theorem}

\begin{proof}
According to Lemma~\ref{le:auxiliary p-th moment estimate} and~\eqref{con:assumption 2}, the process $\hat{Z}:= (VY)'(\dot{V}Y + V\hat{\B}) + c_{p}|\hat{\Sigma}|_{V}^{2}$
and any stopping time $\tau$ for which $E[|Y^{\tau}|_{V^{\tau}}^{p}]$ is locally bounded satisfy
\begin{equation}\label{eq:stability moment inequality}
E\big[u(t\wedge\tau)|Y_{t}^{\tau}|_{V_{t}^{\tau}}^{p}\big] \leq E\big[|Y_{t_{0}}|_{V_{t_{0}}}^{p}\big] + \int_{t_{0}}^{t}E\big[|Y_{s}|_{V_{s}}^{p-2}\big(\dot{u}(s)|Y_{s}|_{V_{s}}^{2} + u(s)p\hat{Z}_{s})\mathbbm{1}_{\{\tau > s\}}\big]\,ds
\end{equation}
for all $t\in [t_{0},\infty)$ and each locally absolutely continuous function $u:[t_{0},\infty)\rightarrow\R_{+}$. Thereby, we immediately infer from~\eqref{eq:essential inequality} that 
\begin{equation*}
p E\big[|Y_{s}|_{V_{s}}^{p-2}\hat{Z}_{s}\mathbbm{1}_{\{\tau > s\}}\big] \leq  \hat{\delta}_{p}(s) + \hat{\gamma}_{p}(s)E\big[|Y_{s}|_{V_{s}}^{p}\mathbbm{1}_{\{\tau > s\}}\big]
\end{equation*}
for a.e.~$s\in [t_{0},\infty)$ with  $\hat{\gamma}_{p}\in\mathscr{L}_{loc}^{1}(\R)$ and $\hat{\delta}_{p}\in\mathscr{L}_{loc}^{1}(\R_{+})$ given by
\begin{equation*}
\hat{\gamma}_{p}:=\sum_{k=1}^{l}(p-2+\alpha_{k})\zeta_{k}[\null_{k}\eta]_{\frac{p}{2-\alpha_{k}}}\theta_{k}^{\beta_{k}}\quad\text{and}\quad\hat{\delta}_{p}:=\sum_{k=1}^{l}(2-\alpha_{k})\zeta_{k}\big[\null_{k}\eta\big]_{\frac{p}{2-\alpha_{k}}}\theta_{k}^{\beta_{k}}.
\end{equation*}
Thus, if we choose the function $u(t) = \exp(-\int_{t_{0}}^{t}\hat{\gamma}_{p}(s)\,ds)$ for all $t\in [t_{0},\infty)$ and the stopping time $\tau=\inf\{t\in [t_{0},\infty)\,|\, |Y_{t}|_{V_{t}}\geq n\}$ in~\eqref{eq:stability moment inequality} for any $n\in\N$, then
\begin{equation*}
e^{-\int_{t_{0}}^{t}\hat{\gamma}_{p}(s)\,ds}E\big[|Y_{t}|_{V_{t}}^{p}\big] \leq E\big[|Y_{t_{0}}|_{V_{t_{0}}}^{p}\big] + \int_{t_{0}}^{t}e^{-\int_{t_{0}}^{s}\hat{\gamma}_{p}(s_{0})\,ds_{0}}\hat{\delta}_{p}(s)\,ds
\end{equation*}
for each $t\in [t_{0},\infty)$, by an application of Fatou's lemma. In particular, $E[|Y|_{V}^{p}]$ is locally bounded. Thus, a second estimation by means of~\eqref{eq:essential inequality} shows that
\begin{equation*}
pE\big[|Y_{s}|_{V_{s}}^{p-2}\hat{Z}_{s}\big] \leq \delta_{p,c}(s) + \gamma_{p,c}(s)E\big[|Y_{s}|_{V_{s}}^{p}\big]
\end{equation*}
for a.e.~$s\in [t_{0},\infty)$, since $[\null_{k}\eta]_{p/(2-\alpha_{k})}\geq 0$ whenever $\beta_{k} > 0$ for all $k\in\{1,\dots,l\}$. Now we take $u(t) = \exp(-\int_{t_{0}}^{t}\gamma_{p,c}(s)\,ds)$ for any $t\in [t_{0},\infty)$ and $\tau = \infty$ in~\eqref{eq:stability moment inequality} to obtain the asserted estimate after dividing by $u(t)$.
\end{proof}

\begin{Remark}\label{re:moment convergence}
Assume that $\delta_{p,c}=0$ a.e., which is the case whenever $\alpha_{k} + \beta_{k} = 2$ for all $k\in\{1,\dots,l\}$. If $\gamma_{p,c}^{+}$ is integrable, then
\begin{equation*}
\sup_{t\in [t_{0},\infty)} e^{\int_{t_{0}}^{t}\gamma_{p,c}^{-}(s)\,ds}E\big[|Y_{t}|_{V_{t}}^{p}\big] < \infty,
\end{equation*}
because $\beta \gamma_{p,c}^{-} + \gamma_{p,c} = \gamma_{p,c}^{+} - (1-\beta)\gamma_{p,c}^{-}$ for each $\beta\in [0,1]$. Thus, if in addition $\gamma_{p,c}^{-}$ fails to be integrable, then
\begin{equation*}
\lim_{t\uparrow\infty}e^{\alpha\int_{t_{0}}^{t}\gamma_{p,c}^{-}(s)\,ds}E\big[|Y_{t}|_{V_{t}}^{p}\big] = 0\quad\text{for any $\alpha\in [0,1)$}.
\end{equation*}
This describes the rate of convergence.
\end{Remark}

\subsection{Moment estimates in a supremum seminorm}\label{se:4.3}

This section provides a general method to obtain $L^{pq}$-moment estimates in a supremum seminorm for $p,q\in [2,\infty)$ under the following abstract mixed power condition on $Y$, which implies~\eqref{con:assumption 2} when $p$ is replaced by $pq$ and $\zeta_{k}=1$ for all $k\in\{1,\dots,l\}$ holds there:
\begin{enumerate}[label=(A.\arabic*), ref=A.\arabic*, leftmargin=\widthof{(A.3)} + \labelsep]
\setcounter{enumi}{2}
\item\label{con:assumption 3} There are $\alpha,\hat{\alpha},\beta,\hat{\beta}\in [0,2]^{l}$ with $\alpha+\beta,\hat{\alpha}+\hat{\beta}\in [0,2]^{l}$, an $\R_{+}^{l}$-valued measurable map $\theta$ on $[t_{0},\infty)$, $c\in\R_{+}^{l}$ and for each $k\in\{1,\dots,l\}$ there exist
\begin{equation*}
\null_{k}\eta\in\mathscr{S}_{loc}^{\frac{pq}{2-\alpha_{k}}}(\R)\quad\text{and}\quad\null_{k}\hat{\eta}\in\mathscr{S}(\R_{+})
\end{equation*}
such that $(VY)'(\dot{V}Y + V\hat{\B}) + c_{pq}|\hat{\Sigma}|_{V}^{2} \leq \sum_{k=1}^{l}\null_{k}\eta |Y|_{V}^{\alpha_{k}}\theta^{\beta_{k}}$ and
\begin{equation*}
|\hat{\Sigma}|_{V}^{2} \leq \sum_{k=1}^{l}\null_{k}\hat{\eta}|Y|_{V}^{\hat{\alpha}_{k}}\theta_{k}^{\hat{\beta}_{k}}\quad\text{a.e.~on $[t_{0},\infty)$ a.s.}
\end{equation*}
Every $k\in\{1,\dots,l\}$ satisfies $\alpha_{k} = 2$ if and only if $k = l$, and it holds that $\theta_{k}(s)$ $\leq c_{k}E[|Y_{s}|_{V_{s}}^{pq}]^{1/(pq)}$ for a.e.~$s\in [t_{0},\infty)$ with
\begin{equation*}
[\null_{k}\eta_{s}]_{\frac{pq}{2-\alpha_{k}}}\mathbbm{1}_{(0,2]}(\beta_{k}) > 0\quad\text{or}\quad [\null_{k}\hat{\eta}_{s}]_{\frac{pq}{2-\hat{\alpha}_{k}}}\mathbbm{1}_{(0,2]}(\hat{\beta}_{k}) > 0.
\end{equation*}
Further, for any $j\in\{1,\dots,l\}$ with $\alpha_{j} > 0$ and each $k\in\{1,\dots,l\}$ there are $\eta_{j,1},\hat{\eta}_{k,1}\in\mathscr{L}_{loc}^{1}(\R)$ with $\hat{\eta}_{k,1}\geq 0$ and $\null_{j,2}\eta,\null_{k,2}\hat{\eta}\in\mathscr{S}(\R_{+})$ satisfying
\begin{equation}\label{eq:additional condition}
\null_{j}\eta = \eta_{j,1}\,\null_{j,2}\eta\quad\text{and}\quad\null_{k}\hat{\eta} = \hat{\eta}_{k,1}\,\null_{k,2}\hat{\eta}
\end{equation}
such that $\eta_{j,1}^{+}[\null_{j,2}\eta]_{pq/(2-\alpha_{k})}^{q}$ and $\hat{\eta}_{k,1}[\null_{k,2}\hat{\eta}]_{pq/(2-\hat{\alpha}_{k})}^{q/2}$ are locally integrable.
\end{enumerate}

\begin{Remark}
If there are $\zeta\in\mathscr{L}_{loc}^{1}(\R^{l})$ and $\hat{\zeta}\in\mathscr{L}_{loc}^{1}(\R_{+}^{l})$ such that $\zeta_{k} = \null_{k}\eta$ and $\hat{\zeta}_{k} = \null_{k}\hat{\eta}$ for all $k\in\{1,\dots,l\}$, then the condition~\eqref{eq:additional condition} holds automatically.
\end{Remark}

Given~\eqref{con:assumption 3} is satisfied, we introduce $\null_{p}\alpha,\null_{p}\hat{\alpha}\in [0,1]^{l}$ coordinatewise by
\begin{equation*}
\null_{p}\alpha_{k} := \frac{p - 2 + \alpha_{k}}{p}\quad\text{and}\quad\null_{p}\hat{\alpha}_{k} := \frac{2p - 2 + \hat{\alpha}_{k}}{2p}.
\end{equation*}
For any $j\in\{1,\dots,l\}$ with $\alpha_{j} > 0$ and each $k\in\{1,\dots,l\}$ let the two $\R_{+}$-valued continuous functions $c_{j,q}$ and $\hat{c}_{k,q}$ on the set of all $(t_{1},t)\in [t_{0},\infty)^{2}$ with $t_{1}\leq t$ be given by
\begin{equation*}
c_{j,q}(t_{1},t) := \bigg(\int_{t_{1}}^{t}\eta_{j,1}^{+}(s)\,ds\bigg)^{1-\frac{1}{q}}\quad\text{and}\quad \hat{c}_{k,q}(t_{1},t):=\bigg(\int_{t_{1}}^{t}\hat{\eta}_{k,1}(s)\,ds\bigg)^{\frac{1}{2}-\frac{1}{q}}.
\end{equation*}
Further, we define three $[0,\infty]$-valued measurable functions $f_{p,q}$, $g_{p,q}$ and $h_{p,q}$ on the set of all $(t_{1},t)\in [t_{0},\infty)^{2}$ with $t_{1}\leq t$ by
\begin{align*}
f_{p,q}(t_{1},t) &:= p E\bigg[\bigg(\int_{t_{1}}^{t}|Y_{s}|_{V_{s}}^{p-2}\sum_{k=1,\,\alpha_{k}=0}^{l}\null_{k}\eta_{s}^{+}\theta_{k}^{\beta_{k}}(s)\,ds\bigg)^{q}\bigg]^{\frac{1}{q}},\\
g_{p,q}(t_{1},t) &:= f_{p,q}(t_{1},t) + p\sum_{k=1,\,\alpha_{k} > 0}^{l}c_{k,q}(t_{1},t)\bigg(\int_{t_{1}}^{t}\eta_{k,1}^{+}(s)E\big[\null_{k,2}\eta_{s}^{q}|Y_{t_{1}}|_{V_{t_{1}}}^{\null_{p}\alpha_{k}pq}\big]\theta_{k}(s)^{\beta_{k}q}\,ds\bigg)^{\frac{1}{q}}\\
&\quad + p\sum_{k=1}^{l}\hat{c}_{k,q}(t_{1},t)\bigg(\overline{w}_{q}\int_{t_{1}}^{t}\hat{\eta}_{k,1}(s)E\big[\null_{k,2}\hat{\eta}_{s}^{\frac{q}{2}}|Y_{t_{1}}|_{V_{t_{1}}}^{\null_{p}\hat{\alpha}_{k}pq}\big]\theta_{k}(s)^{\hat{\beta}_{k}\frac{q}{2}}\,ds\bigg)^{\frac{1}{q}}
\end{align*}
and
\begin{align*}
h_{p,q}(t_{1},t) &:= p^{q} \sum_{k=1,\,\alpha_{k} > 0}^{l}c_{k,q}(t_{1},t)^{q}\int_{t_{1}}^{t}\eta_{k,1}^{+}(s)\big[\null_{k,2}\eta_{s}\big]_{\frac{pq}{2-\alpha_{k}}}^{q}\theta_{k}(s)^{\beta_{k}q}\,ds\\
&\quad + p^{q} \sum_{k=1}^{l}\hat{c}_{k,q}(t_{1},t)^{q}\overline{w}_{q}\int_{t_{1}}^{t}\hat{\eta}_{k,1}(s)\big[\null_{k,2}\hat{\eta}_{s}\big]_{\frac{pq}{2-\hat{\alpha}_{k}}}^{\frac{q}{2}}\theta_{k}(s)^{\hat{\beta}_{k}\frac{q}{2}}\,ds,
\end{align*}
which are finite if $E[|Y|_{V}^{pq}]$ is locally bounded. This fact follows from H\oe lder's inequality, applied to any two $\R_{+}$-valued random variables $\zeta$ and $X$ as follows:
\begin{equation}\label{eq:application of Hoelder's inequality}
E\big[\zeta^{q}X^{\null_{p}\alpha_{k}pq}\big] \leq \big[\zeta\big]_{\frac{pq}{2-\alpha_{k}}}^{q} E\big[X^{pq}\big]^{\null_{p}\alpha_{k}}\quad\text{and}\quad E\big[\zeta^{\frac{q}{2}}X^{\null_{p}\hat{\alpha}_{k}pq}\big] \leq \big[\zeta\big]_{\frac{pq}{2-\hat{\alpha}_{k}}}^{\frac{q}{2}} E\big[X^{pq}\big]^{\null_{p}\hat{\alpha}_{k}}
\end{equation}
for each $k\in\{1,\dots,l\}$. Finally, let us introduce the numbers $\underline{\alpha}:=\min_{k\in\{1,\dots,l\}} (\null_{p}\alpha_{k}\wedge\null_{p}\hat{\alpha}_{k})$ and $\overline{\alpha}:=\max_{k\in\{1,\dots,l\}}(\null_{p}\alpha_{k}\vee\null_{p}\hat{\alpha}_{k})$.

\begin{Proposition}\label{pr:general pq-th moment estimate}
Let~\eqref{con:assumption 3} be valid, $\sum_{k=1,\,\beta_{k} > 0}^{l}[\null_{k}\eta]_{pq/(2-\alpha_{k})}\theta_{k}^{\beta_{k}}$ be locally integrable and $\rho_{0}\in C(\R_{+})$ be given by $\rho_{0}(v):=v^{\underline{\alpha}}\mathbbm{1}_{(0,1]}(v) + v^{\overline{\alpha}}_{(1,\infty)}(v)$. If
\begin{equation}\label{eq:finite expectations}
E\big[|Y_{t_{0}}|_{V_{t_{0}}}^{pq}\big]\quad\text{and}\quad E\bigg[\bigg(\int_{t_{0}}^{t}|Y_{s}|_{V_{s}}^{p-2}\sum_{k=1,\,\alpha_{k} = 0}^{l}\null_{k}\eta_{s}^{+}\,ds\bigg)^{q}\bigg]
\end{equation}
are finite, then $\sup_{s\in [t_{0},t]} |Y_{s}|_{V_{s}}$ is $pq$-fold integrable and
\begin{align*}
E\bigg[\bigg(\sup_{s\in [t_{1},t]} |Y_{s}|_{V_{s}}^{p} - |Y_{t_{1}}|_{V_{t_{1}}}^{p}\bigg)^{q}\bigg] \leq \Psi_{\rho_{0}}\bigg((2l+1)^{q-1}g_{p,q}(t_{1},t)^{q},(2l+1)^{q-1}h_{p,q}(t_{1},t)\bigg)
\end{align*}
for all $t_{1},t\in [t_{0},\infty)$ with $t_{1}\leq t$. In particular, $E[|Y|_{V}^{pq}]$ is continuous.
\end{Proposition}

\begin{proof}
Under the integrability assertion, dominated convergence yields $\lim_{n\uparrow\infty} E[|Y_{t_{n}}|_{V_{t_{n}}}^{pq}]$ $= E[|Y_{t}|_{V_{t}}^{pq}]$ for each sequence $(t_{n})_{n\in\N}$ in $[t_{0},\infty)$ converging to some $t\in [t_{0},\infty)$. Hence, we merely need to show the first two claims.

As~\eqref{con:assumption 3} implies~\eqref{con:assumption 2} when $p$ is replaced by $pq$ in the latter assumption, we know from Theorem~\ref{th:stability moment estimate} that $E[|Y|_{V}^{pq}]$ is locally bounded. Thus, Proposition~\ref{pr:auxiliary p-th moment estimate in supremum seminorm}, the triangle inequality in the $L^{q}$-norm and Jensen's inequality yield that
\begin{align*}
E\bigg[\bigg(\sup_{s\in [t_{1},t]} &|Y_{s}^{\tau_{n}}|_{V_{s}^{\tau_{n}}}^{p} - |Y_{t_{1}}|_{V_{t_{1}}}^{p}\bigg)^{q}\bigg]^{\frac{1}{q}}\leq f_{p,q}(t_{1},t)\\
&\quad + p\sum_{k=1,\,\alpha_{k} > 0}^{l}c_{k,q}(t_{1},t)\bigg(\int_{t_{1}}^{t}\eta_{k,1}^{+}(s)E\big[\null_{k,2}\eta_{s}^{q}|Y_{s}^{\tau_{n}}|_{V_{s}^{\tau_{n}}}^{\null_{p}\alpha_{k}pq}\big]\theta_{k}(s)^{\beta_{k}q}\,ds\bigg)^{\frac{1}{q}}\\
&\quad + p\sum_{k=1}^{l}\hat{c}_{k,q}(t_{1},t)\bigg(\overline{w}_{q}\int_{t_{1}}^{t}\hat{\eta}_{k,1}(s)E\big[\null_{k,2}\hat{\eta}_{s}^{\frac{q}{2}}|Y_{s}^{\tau_{n}}|_{V_{s}^{\tau_{n}}}^{\null_{p}\hat{\alpha}_{k}pq}\big]\theta_{k}(s)^{\hat{\beta}_{k}\frac{q}{2}}\,ds\bigg)^{\frac{1}{q}}
\end{align*}
for any fixed $t_{1},t\in [t_{0},\infty)$ with $t_{1}\leq t$ and $n\in\N$, where $\tau_{n}:=\inf\{t\in [t_{1},\infty)\,|\,|Y_{t}|_{V_{t}}\geq n\}$. Hence, if $E[|Y_{t_{1}}|_{V_{t_{1}}}^{pq}] < \infty$, then Minkowski's inequality,~\eqref{eq:application of Hoelder's inequality} and Bihari's inequality give the claimed estimate for
\begin{equation*}
E\bigg[\bigg(\sup_{s\in [t_{1},t]} |Y_{s}^{\tau_{n}}|_{V_{s}^{\tau_{n}}}^{p} - |Y_{t_{1}}|_{V_{t_{1}}}^{p}\bigg)^{q}\bigg]^{\frac{1}{q}}.
\end{equation*}
Afterwards, Fatou's lemma implies the asserted bound. Moreover, if we take $t_{1}=t_{0}$, then the $pq$-fold integrability of $|Y_{t_{0}}|_{V_{t_{0}}}$ implies that of $\sup_{s\in [t_{0},t]} |Y_{s}|_{V_{s}}$. For this reason, the proposition is proven.
\end{proof}

\begin{Remark}\label{re:general pq-th moment estimate remark}
The second expectation in~\eqref{eq:finite expectations} is finite if for each $k\in\{1,\dots,l\}$ with $\alpha_{k}=0$ there are $\eta_{k,1}\in\mathscr{L}_{loc}^{1}(\R)$ and $\null_{k,2}\eta\in\mathscr{S}(\R_{+})$ such that $\null_{k}\eta = \eta_{k,1}\,\null_{k,2}\eta$ and $\eta_{k,1}^{+}[\null_{k,2}\eta]_{pq/2}^{q}$ is locally integrable, by the inequalities of Jensen and H\oe lder.
\end{Remark}

If~\eqref{con:assumption 3} is satisfied, then for $\gamma\in\mathscr{L}_{loc}^{1}(\R)$ we define an $[0,\infty]$-valued measurable function $h_{\gamma,p,q}$ on the set of all $(t_{1},t)\in [t_{0},\infty)^{2}$ with $t_{1} \leq t$ via
\begin{align*}
&h_{\gamma,p,q}(t_{1},t) := p E\bigg[\bigg(\int_{t_{1}}^{t}e^{-\int_{t_{1}}^{s}\gamma(s_{0})\,ds_{0}}|Y_{s}|_{V_{s}}^{p-2}\sum_{k=1,\,\alpha_{k}=0}^{l}\null_{k}\eta_{s}^{+}\theta_{k}^{\beta_{k}}(s)\,ds\bigg)^{q}\bigg]^{\frac{1}{q}}\\
& + p\sum_{k=1,\,
\alpha_{k}\in (0,2)}^{l}c_{k,q}(t_{1},t)\bigg(\int_{t_{1}}^{t}\eta_{k,1}^{+}(s)E\big[\null_{k,2}\eta_{s}\big]_{\frac{pq}{2-\alpha_{k}}}^{q}e^{-q\int_{t_{1}}^{s}\gamma(s_{0})\,ds_{0}}E\big[|Y_{s}|_{V_{s}}^{pq}\big]^{\null_{p}\alpha_{k}}\theta(s)^{\beta_{k}q}\,ds\bigg)^{\frac{1}{q}}\\
& + p\sum_{k=1}^{l}\hat{c}_{k,q}(t_{1},t)\bigg(\overline{w}_{q}\int_{t_{0}}^{t}\hat{\eta}_{k,1}(s)\big[\null_{k,2}\hat{\eta}_{s}\big]_{\frac{pq}{2-\hat{\alpha}_{k}}}^{\frac{q}{2}}e^{-q\int_{t_{1}}^{s}\gamma(s_{0})\,ds_{0}}E\big[|Y_{s}|_{V_{s}}^{pq}\big]^{\null_{p}\hat{\alpha}_{k}}\theta_{k}(s)^{\hat{\beta}_{k}\frac{q}{2}}\,ds\bigg)^{\frac{1}{q}},
\end{align*}
then an auxiliary $pq$-th moment stability estimate in a supremum seminorm follows.

\begin{Lemma}\label{le:general pq-th moment stability estimate}
Let~\eqref{con:assumption 3} be valid and $\sum_{k=1,\,\beta_{k} > 0}^{l}[\null_{k}\eta]_{pq/(2-\alpha_{k})}\theta_{k}^{\beta_{k}}$ be locally integrable. If the expectations in~\eqref{eq:finite expectations} are finite, then
\begin{align*}
&E\bigg[\bigg(\sup_{s\in [t_{1},t]} e^{-\int_{t_{1}}^{s}\gamma(s_{0})\,ds_{0}}|Y_{s}|_{V_{s}}^{p} - |Y_{t_{1}}|_{V_{t_{1}}}^{p}\bigg)^{q}\bigg]^{\frac{1}{q}} \leq  h_{\gamma,p,q}(t_{1},t)\\
&\quad + \bigg(\int_{t_{1}}^{t}(p\big[\null_{l}\eta_{s}\big]_{\infty} - \gamma)^{+}(s)\,ds\bigg)^{1 - \frac{1}{q}}\bigg(\int_{t_{1}}^{t}(p\big[\null_{l}\eta_{s}\big]_{\infty} - \gamma)^{+}(s)e^{-q\int_{t_{1}}^{s}\gamma(s_{0})\,ds_{0}}E\big[|Y_{s}|_{V_{s}}^{pq}\big]\,ds\bigg)^{\frac{1}{q}}
\end{align*}
for any $\gamma\in\mathscr{L}_{loc}^{1}(\R)$ and every $t_{1},t\in [t_{0},\infty)$ with $t_{1} \leq t$.
\end{Lemma}

\begin{proof}
From Theorem~\ref{th:stability moment estimate} we deduce that  $h_{\gamma,p,q}$ is finite, Proposition~\ref{pr:general pq-th moment estimate} gives the $pq$-fold integrability of $\sup_{s\in [t_{0},t]} |Y_{s}|_{V_{s}}$ and we readily observe that
\begin{equation*}
-\gamma |Y|_{V}^{2} + p\sum_{k=1}^{l}\null_{k}\eta|Y|_{V}^{\alpha_{k}}\theta_{k}^{\beta_{k}} = p\bigg(\sum_{k=1}^{l-1}\null_{k}\eta |Y|_{V}^{\alpha_{k}}\theta_{k}^{\beta_{k}}\bigg) + (p\,\null_{l}\eta - \gamma)|Y|_{V}^{2}.
\end{equation*}
Hence, we may infer the assertion from Proposition~\ref{pr:auxiliary p-th moment estimate in supremum seminorm}, by using the triangle inequality in the $L^{q}$-norm, Jensen's inequality and~\eqref{eq:application of Hoelder's inequality}.
\end{proof}

At last, we introduce a condition that forces the function $\delta_{pq,c}$ in~\eqref{eq:stability drift coefficient} when $p$ is replaced by $pq$ and $\zeta_{k}=1$ for all $k\in\{1,\dots,l\}$ to vanish a.e.~on $[t_{1},\infty)$ for some $t_{1}\in [t_{0},\infty)$.
\begin{enumerate}[label=(A.\arabic*), ref=A.\arabic*, leftmargin=\widthof{(A.4)} + \labelsep]
\setcounter{enumi}{3}
\item\label{con:assumption 4} Assumption~\eqref{con:assumption 3} is satisfied and for any $k\in\{1,\dots,l\}$ with $\alpha_{k}=0$ there are $\eta_{k,1}\in\mathscr{L}_{loc}^{1}(\R)$ and $\null_{k,2}\eta\in\mathscr{S}(\R_{+})$ so that $\eta_{k,1}^{+}[\null_{k,2}\eta]_{pq/2}^{q}$ is locally integrable and
\begin{equation*}
\null_{k}\eta = \eta_{k,1}\,\null_{k,2}\eta.
\end{equation*}
Moreover, there are $t_{1}\in [t_{0},\infty)$, $\hat{\delta} > 0$ and $\overline{c}_{0}\in\R_{+}$ such that $\null_{k}\eta$ (resp.~$\null_{k}\hat{\eta}$) vanishes a.e.~on $[t_{1},\infty)$ for any $k\in\{1,\dots,l\}$ with $\alpha_{k} + \beta_{k} < 2$ (resp.~$\hat{\alpha}_{k} + \hat{\beta}_{k} < 2$) and
\begin{align*}
\bigg(\int_{t}^{t+\hat{\delta}}\eta_{j,1}^{+}(s)\max\big\{1,\big[\null_{j,2}\eta_{s}\big]_{\frac{pq}{2-\alpha_{j}}}^{q}\big\}\,ds\bigg)\vee\bigg(\int_{t}^{t+\hat{\delta}}\hat{\eta}_{k,1}(s)\max\big\{1,\big[\null_{k,2}\hat{\eta}_{s}\big]_{\frac{pq}{2-\hat{\alpha}_{k}}}^{\frac{q}{2}}\big\}\,ds\bigg)
\end{align*}
is bounded by $\overline{c}_{0}$ for each $t\in [t_{1},\infty)$ and any $j,k\in\{1,\dots,l\}$ with $j\leq l-1$.
\end{enumerate}

Then the pathwise asymptotic behaviour of $Y$ in the next section can be handled with the subsequent $pq$-th moment estimate in a supremum seminorm.

\begin{Proposition}\label{pr:pq-th moment stability estimate in supremum norm}
Let~\eqref{con:assumption 4} hold, $E[|Y_{t_{0}}|_{V_{t_{0}}}^{pq}] < \infty$ and $\sum_{k=1,\,\beta_{k} > 0}^{l}[\null_{k}\eta]_{pq/(2-\alpha_{k})}\theta_{k}^{\beta_{k}}$ be locally integrable and assume there are $\gamma\in\mathscr{L}_{loc}^{1}(\R)$ and $\overline{c}_{\gamma,-1},\overline{c}_{\gamma,0},\overline{c}_{\gamma,q},\hat{c}_{\gamma,0}\in\R_{+}$ so that
\begin{equation*}
\int_{t_{2}}^{t}(p[\null_{l}\eta_{s}]_{\infty} - \gamma)^{+}(s)\,ds \leq \overline{c}_{\gamma,-1},\quad \int_{t_{2}}^{t}(\gamma_{pq,c} - q_{0}\gamma)(s)\,ds \leq \overline{c}_{\gamma,q_{0}},\quad \int_{t_{2}}^{t}\gamma(s)\,ds \leq\hat{c}_{\gamma,0}
\end{equation*}
for any $t_{2},t\in [t_{1},\infty)$ with $t_{2}\leq t < \hat{\delta}$ and each $q_{0}\in\{0,q\}$. Then there is $\overline{c} > 0$ such that
\begin{equation}\label{eq:auxiliary pq-th moment stability estimate}
E\bigg[\sup_{s\in [t_{2},t]} |Y_{s}|_{V_{s}}^{pq}\bigg]^{\frac{1}{q}} \leq \overline{c}\bigg(E\big[|Y_{t_{0}}|_{V_{t_{0}}}^{pq}\big] + \int_{t_{0}}^{t_{1}}\delta_{pq,c}(s)\,ds\bigg)^{\frac{1}{q}}e^{\frac{1}{q}\int_{t_{1}}^{t_{2}}\gamma_{pq,c}(s)\,ds}
\end{equation}
for every $t_{2},t\in [t_{1},\infty)$ with $t_{2} \leq t < \hat{\delta}$.
\end{Proposition}

\begin{proof}
As $\null_{k}\eta = 0$ (resp.~$\null_{k}\hat{\eta} = 0$) a.e.~on $[t_{1},\infty)$ for any $k\in\{1,\dots,l\}$ with $\alpha_{k} + \beta_{k} < 2$ (resp.~$\hat{\alpha}_{k} + \hat{\beta}_{k} < 2$), it follows from Remark~\ref{re:general pq-th moment estimate remark} together with Lemma~\ref{le:general pq-th moment stability estimate} and Jensen's inequality that
\begin{equation}\label{eq:auxiliary pq-th moment stability estimate 1}
\begin{split}
&E\bigg[\sup_{s\in [t_{2},t]} e^{-q\int_{t_{2}}^{s}\gamma(s_{0})\,ds_{0}}|Y_{s}|_{V_{s}}^{pq}\bigg]^{\frac{1}{q}} \leq E\big[|Y_{t_{2}}|_{V_{t_{2}}}^{pq}\big]^{\frac{1}{q}}\\
&\quad + p \overline{c}_{0}^{1-\frac{1}{q}}\sum_{k=1}^{l-1} c_{k}^{2-\alpha_{k}}\bigg(\int_{t_{2}}^{t}\eta_{k,1}^{+}(s)\big[\null_{k,2}\eta_{s}\big]_{\frac{pq}{2-\alpha_{k}}}^{q}e^{-q\int_{t_{2}}^{s}\gamma(s_{0})\,ds_{0}} E\big[|Y_{s}|_{V_{s}}^{pq}\big]\,ds\bigg)^{\frac{1}{q}} \\
&\quad + \overline{c}_{\gamma,-1}^{1-\frac{1}{q}}\bigg(\int_{t_{2}}^{t}(p\big[\null_{l}\eta_{s}\big]_{\infty} - \gamma)^{+}(s)e^{-q\int_{t_{2}}^{s}\gamma(s_{0})\,ds_{0}}E\big[|Y_{s}|_{V_{s}}^{pq}\big]\,ds\bigg)^{\frac{1}{q}}\\
&\quad + p\overline{c}_{0}^{\frac{1}{2}-\frac{1}{q}}\sum_{k=1}^{l}c_{k}^{1-\frac{\hat{\alpha}_{k}}{2}}\bigg(\overline{w}_{q}\int_{t_{2}}^{t}\hat{\eta}_{k,1}(s)\big[\null_{k,2}\hat{\eta}_{s}\big]_{\frac{pq}{2-\hat{\alpha}_{k}}}^{\frac{q}{2}}e^{-q\int_{t_{2}}^{s}\gamma(s_{0})\,ds_{0}}E\big[|Y_{s}|_{V_{s}}^{pq}\big]\,ds\bigg)^{\frac{1}{q}}.
\end{split}
\end{equation}
In this context, we used that $E[\null_{k}\eta^{q}|Y|_{V}^{\null_{p}\alpha_{0}pq}] \leq [\null_{k}\eta]_{pq/2}^{q}E[|Y|_{V}^{pq}]^{\null_{p}\alpha_{0}}$ for any $k\in\{1,\dots,l\}$ with $\alpha_{k} = 0$ and $\alpha_{p,0}:=(p-2)/p$, by H\oe lder's inequality, and
\begin{equation*}
\frac{p - 2 + \alpha_{0}}{p} + \frac{\beta_{0}}{p} = \frac{2p - 2 + \alpha_{0}}{2p} + \frac{\beta_{0}}{2p} = 1
\end{equation*}
for each $\alpha_{0},\beta_{0}\in [0,1]$ with $\alpha_{0} + \beta_{0} = 1$. Moreover, as $\delta_{pq,c} = 0$ a.e.~on $[t_{1},\infty)$, the moment stability estimate of Theorem~\ref{th:stability moment estimate} gives us that
\begin{equation*}
e^{-q_{0}\int_{t_{2}}^{s}\gamma(s_{0})\,ds_{0} - \overline{c}_{\gamma,q_{0}}} E\big[|Y_{s}|_{V_{s}}^{pq}\big] \leq e^{\int_{t_{0}}^{t_{2}}\gamma_{pq,c}(s_{0})\,ds_{0}}E\big[|Y_{t_{0}}|_{V_{t_{0}}}^{pq}\big] + \int_{t_{0}}^{t_{1}}e^{\int_{s_{0}}^{t_{2}}\gamma_{pq,c}(s_{1})\,ds_{1}}\delta_{pq,c}(s_{0})\,ds_{0}
\end{equation*}
for all $s\in [t_{2},t]$ and $q_{0}\in\{0,q\}$. Thus, the sum of the four right-hand terms in~\eqref{eq:auxiliary pq-th moment stability estimate 1} is bounded by the right-hand expression in~\eqref{eq:auxiliary pq-th moment stability estimate} when $\overline{c}$ is replaced by the constant
\begin{equation*}
\overline{c}_{1} := e^{\frac{1}{q}\int_{t_{0}}^{t_{1}}\gamma_{pq,c}^{+}(s)\,ds}\bigg(e^{\frac{1}{q}\overline{c}_{\gamma,0}} + e^{\frac{1}{q}\overline{c}_{\gamma,q}}\bigg(p\overline{c}_{0}\bigg(\sum_{k=1}^{l-1}c_{k}^{2-\alpha_{k}}\bigg) + \overline{c}_{\gamma,-1} + p\overline{c}_{0}^{\frac{1}{2}}\overline{w}_{q}^{\frac{1}{q}}\sum_{k=1}^{l}c_{k}^{1-\frac{\hat{\alpha}_{k}}{2}}\bigg)\bigg).
\end{equation*}
Because $\exp(-\int_{t_{2}}^{s}\gamma(s_{0})\,ds_{0})\geq \exp(-\hat{c}_{\gamma,0})$ for any $s\in [t_{2},t]$, the claimed bound holds for $\overline{c}:=\exp(\hat{c}_{\gamma,0})\overline{c}_{1}$.
\end{proof}

\subsection{Pathwise asymptotic behaviour}\label{se:4.4}

Finally, we derive a limiting bound for $Y$ under the random seminorm $|\cdot|_{V}$ from the moment estimate of Proposition~\ref{pr:pq-th moment stability estimate in supremum norm}. For this purpose, we recall an application of the Borel-Cantelli Lemma in~\cite{KalMeyPro21}[Lemma 4.11 and Remark 4.12].

Namely, let $A\in\mathscr{F}$ and $X$ be an $\R_{+}$-valued right-continuous process for which there are a strictly increasing sequence $(t_{n})_{n\in\N}$ in $[t_{0},\infty)$ with $\lim_{n\uparrow\infty} t_{n}=\infty$, a sequence $(c_{n})_{n\in\N}$ in $(0,\infty)$, $\hat{c} > 0$ and $\hat{\varepsilon}\in (0,1)$ such that
\begin{equation*}
E\bigg[\sup_{s\in (t_{n},t_{n+1}]}X_{s}\mathbbm{1}_{A}\bigg] \leq \hat{c}c_{n}\quad\text{for all $n\in\N$}\quad\text{and}\quad \sum_{n=1}^{\infty} c_{n}^{\varepsilon} < \infty
\end{equation*}
for each $\varepsilon\in (0,\hat{\varepsilon})$. Then any $(0,\infty)$-valued lower semicontinuous function $\varphi$ on $(t_{1},\infty)$ satisfies
\begin{equation}\label{eq:pathwise result}
\limsup_{t\uparrow\infty} \frac{\log(X_{t})}{\varphi(t)} \leq \limsup_{n\uparrow\infty} \frac{\log(c_{n})}{\inf_{s\in (t_{n},t_{n+1}]}\varphi(s)}\quad\text{a.s.~on $A$.}
\end{equation}

\begin{Theorem}\label{th:pathwise stability}
Let~\eqref{con:assumption 4} be satisfied and $\sum_{k=1,\,\beta_{k} > 0}^{l}[\null_{k}\eta]_{pq/(2-\alpha_{k})}\theta_{k}^{\beta_{k}}$ be locally integrable. Assume that $\gamma_{pq,c} \leq 0$ a.e.~on $[t_{1},\infty)$ and there is an increasing sequence $(t_{n})_{n\in\N\setminus\{1\}}$ in $[t_{1},\infty)$ such that
\begin{equation*}
\sup_{n\in\N} (t_{n+1} - t_{n}) < \hat{\delta},\quad \lim_{n\uparrow\infty} t_{n} = \infty
\end{equation*}
and $\sum_{n=1}^{\infty}\exp(\varepsilon(pq)^{-1}\int_{t_{1}}^{t_{n}}\gamma_{pq,c}(s)\,ds) < \infty$ for all $\varepsilon\in (0,\hat{\varepsilon})$ and some $\hat{\varepsilon}\in (0,1)$. If $E[|Y_{t_{0}}|_{V_{t_{0}}}^{pq}]$ is finite or $\beta=\hat{\beta}=0$, then
\begin{equation*}
\limsup_{t\uparrow\infty} \frac{1}{\varphi(t)}\log\big(|Y_{t}|_{V_{t}}\big)  \leq \frac{1}{pq}\limsup_{n\uparrow\infty}\frac{1}{\varphi(t_{n})}\int_{t_{1}}^{t_{n}}\gamma_{pq,c}(s)\,ds\quad\text{a.s.}
\end{equation*}
for each increasing function $\varphi:[t_{1},\infty)\rightarrow\R_{+}$ that is positive on $(t_{1},\infty)$.
\end{Theorem}

\begin{proof}
Since $\gamma_{pq,c} = pq([\null_{l}\eta]_{\infty} +  \sum_{k=1}^{l-1}c_{k}^{2-\alpha_{k}}[\null_{k}\eta]_{pq/(2-\alpha_{k})})$ a.e.~on $[t_{1},\infty)$, it follows that $[\null_{l}\eta]_{\infty} \leq 0$ a.e.~on the same interval. Consequently, if $E[|Y_{t_{0}}|_{V_{t_{0}}}^{pq}] < \infty$, then 
\begin{equation*}
E\bigg[\sup_{s\in [t_{n},t_{n+1}]} |Y_{s}|_{V_{s}}^{pq}\bigg] \leq \hat{c}e^{\int_{t_{1}}^{t_{n}}\gamma_{pq,c}(s)\,ds}
\end{equation*}
for each $n\in\N$ and some $\hat{c} > 0$, by taking $\gamma = p[\null_{l}\eta]_{\infty}$ in Proposition~\ref{pr:pq-th moment stability estimate in supremum norm}. Indeed, from Jensen's inequality and~\eqref{con:assumption 4} we immediately obtain that
\begin{equation*}
\int_{t_{2}}^{t}\gamma_{pq,c}(s) - pq\big[\null_{l}\eta_{s}\big]_{\infty}\,ds \leq pq\sum_{k=1}^{l-1}c_{k}^{2-\alpha_{k}}\overline{c}_{0}
\end{equation*}
for any $t_{2},t\in [t_{1},\infty)$ with $t_{2}\leq t < \hat{\delta}$. Next, we suppose that $\beta = \hat{\beta} = 0$ and set $A_{k}:=\{|Y_{t_{0}}|_{V_{t_{0}}}\leq k\}$ for given $k\in\N$. Then Proposition~\ref{pr:pq-th moment stability estimate in supremum norm} yields $\hat{c}_{k} > 0$ such that
\begin{equation*}
E\bigg[\sup_{s\in [t_{n},t_{n+1}]} |Y_{s}\mathbbm{1}_{A_{k}}|_{V_{s}}^{pq}\bigg] \leq \hat{c}_{k}e^{\int_{t_{1}}^{t_{n}}\gamma_{pq,c}(s)\,ds}
\end{equation*}
for all $n\in\N$, as the random It{\^o} process $Y\mathbbm{1}_{A_{k}}$ with drift $\hat{\B}\mathbbm{1}_{A_{k}}$ and diffusion $\hat{\Sigma}\mathbbm{1}_{A_{k}}$ satisfies~\eqref{con:assumption 3} and~\eqref{con:assumption 4}. Hence, in both cases the claimed pathwise inequality follows from the result recalled at~\eqref{eq:pathwise result} and the fact that $\bigcup_{k\in\N} A_{k} = \Omega$.
\end{proof}

\section{Proofs of the main results}\label{se:5}

\subsection{Proofs of the moment estimates, uniqueness and moment stability}\label{se:5.1}

\begin{proof}[Proof of Proposition~\ref{pr:specific abstract second moment estimate}]
Let $\hat{\B}\in\mathscr{S}(\R^{m})$ and $\hat{\Sigma}\in\mathscr{S}(\R^{m\times d})$ be defined via
\begin{equation}\label{eq:specific drift and diffusion coefficients}
\hat{\B}_{s} := \B_{s}(X_{s},P_{X_{s}}) - \tilde{\B}_{s}(\tilde{X}_{s},P_{\tilde{X}_{s}})\quad\text{and}\quad \hat{\Sigma}_{s} :=\Sigma_{s}(X_{s},P_{X_{s}}) - \tilde{\Sigma}_{s}(\tilde{X}_{s},P_{\tilde{X}_{s}}).
\end{equation}
Then $Y$ is a random It{\^o} process with drift $\hat{\B}$ and diffusion $\hat{\Sigma}$ satisfying
\begin{equation*}
2Y'\hat{\B} + |\hat{\Sigma}|^{2} \leq \varepsilon + \eta\rho(|Y|^{2}) + \lambda\varrho\circ\theta^{2}
\end{equation*}
a.e.~on $[t_{0},\infty)$ a.s.~with the measurable function $\theta:=\vartheta(P_{X},P_{\tilde{X}})$. For this reason, the proposition is a special case of Proposition~\ref{pr:abstract second moment estimate} for the constant choice $V=\mathbbm{I}_{m}$.
\end{proof}

\begin{proof}[Proof of Corollary~\ref{co:pathwise uniqueness}]
For both claims in (i) and (ii) let $X$ and $\tilde{X}$ be two solutions to~\eqref{eq:McKean-Vlasov} with $X_{t_{0}} = \tilde{X}_{t_{0}}$ a.s. Suppose first that~\eqref{con:2} holds and the expression
\begin{equation*}
E\big[\lambda\big]\varrho\big(\vartheta(P_{X},P_{\tilde{X}})^{2}\big) + \eta E\big[\rho(|X-\tilde{X}|^{2})\big]\mathbbm{1}_{\R_{+}}\big(\Phi_{\rho}(\infty)\big)
\end{equation*}
is locally integrable. Then Proposition~\ref{pr:specific abstract second moment estimate} gives $E[|X_{t} -\tilde{X}_{t}|^{2}] = 0$ for any $t\in [t_{0},\infty)$, as $\varrho_{0}\in C(\R_{+})$ given by $\varrho_{0}(v) := \rho(v)\vee\varrho(c_{2,\mathscr{P}}^{2}v)$ satisfies $(0,w)\in D_{\varrho_{0}}$ and $\Psi_{\varrho_{0}}(0,w) = 0$ for all $w\in\R_{+}$. Hence, path continuity implies that $X$ and $\tilde{X}$ are indistinguishable.

Now let~\eqref{con:3} hold and set $\tau_{n}:=\inf\{t\in [t_{0},\infty)\,|\,|X_{t}|\geq n\text{ or } |\tilde{X}_{t}|\geq n\}$ for fixed $n\in\N$. Then $X^{\tau_{n}}$ and $\tilde{X}^{\tau_{n}}$ are solutions to~\eqref{eq:McKean-Vlasov} when $\B$ and $\Sigma$ are replaced by the admissible maps $\null_{n}\B$ and $\null_{n}\Sigma$ with values in $\R^{m}$ and $\R^{m\times d}$, respectively, given by
\begin{equation*}
\null_{n}\B_{s}(x) := \hat{\B}_{s}(\varphi_{n}(x))\mathbbm{1}_{\{\tau_{n} > s\}}\quad\text{and}\quad\null_{n}\Sigma_{s}(x) :=\hat{\Sigma}_{s}(\varphi_{n}(x))\mathbbm{1}_{\{\tau_{n} > s\}},
\end{equation*}
where $\varphi_{n}$ is the radial retraction of the ball $\{x\in\R^{m}\,|\,|x|\leq n\}$. This $\R^{m}$-valued map on $\R^{m}$ is defined by $\varphi(x) = x$, if $|x|\leq n$, and $\varphi_{n}(x) = (n/|x|)x$, if $|x| > n$. 

Because~\eqref{con:1} holds when $(\B,\Sigma) = (\tilde{\B},\tilde{\Sigma})=(\null_{n}\B,\null_{n}\Sigma)$, $\varepsilon =0$ and $(\eta,\lambda,\rho)$ $= (\eta_{n},0,\rho_{n})$, Proposition~\ref{pr:specific abstract second moment estimate} yields $X^{\tau_{n}}=\tilde{X}^{\tau_{n}}$ a.s., which implies that $X = \tilde{X}$ a.s.
\end{proof}

\begin{proof}[Proof of Proposition~\ref{pr:specific stability moment estimate}]
We define $\hat{\B}\in\mathscr{S}(\R^{m})$ and $\hat{\Sigma}\in\mathscr{S}(\R^{m\times d})$ by~\eqref{eq:specific drift and diffusion coefficients} and observe that
\begin{equation}\label{eq:stability estimates}
Y'\hat{\B} \leq \sum_{k=1}^{l}\zeta_{k}\,\null_{k}\eta|Y|^{1 + \alpha_{k}}\theta^{\beta_{k}}\quad\text{and}\quad |\hat{\Sigma}|^{2} \leq \sum_{j,k=1}^{l}\hat{\zeta}_{j}\hat{\zeta}_{k}\,\null_{j}\hat{\eta}\,\null_{k}\hat{\eta}|Y|^{\alpha_{j} + \alpha_{k}}\theta^{\beta_{j} + \beta_{k}}
\end{equation}
a.e.~on $[t_{0},\infty)$ a.s.~with $\theta:=\vartheta(P_{X},P_{\tilde{X}})$. So, all claims follow from Theorem~\ref{th:stability moment estimate}, as our considerations succeeding the definitions~\eqref{eq:specific stability coefficient} and~\eqref{eq:drift stability coefficient} explain.

Thereby, Young's inequality shows that the local integrability of $\Theta(\cdot,P_{X},P_{\tilde{X}})$ entails that of $\hat{\zeta}_{j}\hat{\zeta}_{k}\big[\null_{j}\hat{\eta}\,\null_{k}\hat{\eta}\big]_{p/(2-\alpha_{j} - \alpha_{k})}\theta^{\beta_{j} + \beta_{k}}$ for each $j,k\in\{1,\dots,l\}$, since
\begin{equation*}
\big[\null_{j}\hat{\eta}\,\null_{k}\hat{\eta}\big]_{\frac{p}{2-\alpha_{j} - \alpha_{k}}}\leq \big[\null_{j}\hat{\eta}\big]_{\frac{p}{1-\alpha_{j}}}\big[\null_{k}\hat{\eta}\big]_{\frac{p}{1-\alpha_{k}}}\quad\text{and}\quad p\hat{\zeta}_{k}^{2}\leq p-1+\alpha_{k}+\beta_{k} + (1-\alpha_{k}-\beta_{k})\hat{\zeta}_{k}^{\frac{2p}{1-\alpha_{k}-\beta_{k}}}.
\end{equation*}
\end{proof}

\begin{proof}[Proof of Corollary~\ref{co:moment stability}]
Definition~\ref{de:stability} clarifies that both stability claims are immediate consequences of Proposition~\ref{pr:specific stability moment estimate}.
\end{proof}

\begin{proof}[Proof of Corollary~\ref{co:exponential moment stability}]
From Remark~\ref{re:moment convergence} we infer the first two assertions, by replacing $\gamma_{p,c}$ by $\gamma_{p,\mathscr{P}}$ there. Regarding the exponential $p$-th moment stability, we note that
\begin{equation*}
E\big[|X_{t} - \tilde{X}_{t}|^{p}\big] \leq e^{\int_{t_{0}}^{t}\gamma_{p,\mathscr{P}}(s)\,ds}E\big[|X_{t_{0}} - \tilde{X}_{t_{0}}|^{p}\big]
\end{equation*}
for each $t\in [t_{0},\infty)$, according to Proposition~\ref{pr:specific stability moment estimate}. Thus, let us directly exclude the case that $X_{t_{0}} = \tilde{X}_{t_{0}}$ a.s. Then~\eqref{con:6} implies that
\begin{equation*}
\limsup_{t\uparrow\infty}\frac{1}{t^{\alpha_{l}}}\log\big(E\big[|X_{t} -\tilde{X}_{t}|^{p}\big]\big) \leq \limsup_{t\uparrow\infty}\sum_{k=1}^{l}\hat{\lambda}_{k}\frac{(t-s_{k})^{\alpha_{k}} - (t_{1} - s_{k})^{\alpha_{k}}}{t^{\alpha_{l}}} = \hat{\lambda}_{l}
\end{equation*}
and we obtain the first assertion from Remark~\ref{re:exponential moment stability}. To prove the second claim, we may suppose that $l=1$ and readily check that
\begin{equation}\label{eq:exponential moment stability auxiliary estimate}
E\big[|X_{t} - \tilde{X}_{t}|^{p}\big] \leq \hat{c}_{0}e^{\hat{\lambda}_{1}(t-t_{0})^{\alpha_{1}}}E\big[|X_{t_{0}} - \tilde{X}_{t_{0}}|^{p}\big]
\end{equation}
for any $t\in [t_{0},t_{1}]$ with $\hat{c}_{0}:=\max_{t\in [t_{0},t_{1}]}\exp(\int_{t_{0}}^{t}\gamma_{p,\mathscr{P}}(s)\,ds - \hat{\lambda}_{1}(t-t_{0})^{\alpha_{1}})$. To ensure that~\eqref{eq:exponential moment stability auxiliary estimate} is satisfied for all $t\in [t_{0},\infty)$ we take the constant
\begin{equation*}
\hat{c} := \hat{c}_{0}\vee\exp\bigg(\int_{t_{0}}^{t_{1}}\gamma_{p,\mathscr{P}}(s)\,ds - \hat{\lambda}_{1}(t_{1}-s_{1})^{\alpha_{1}}\bigg)
\end{equation*}
instead of $\hat{c}_{0}$, because $\int_{t_{1}}^{t}\gamma_{p,\mathscr{P}}(s)\,ds \leq \hat{\lambda}_{1}((t - s_{1})^{\alpha_{1}} - (t_{1} - s_{1})^{\alpha_{1}})$ for each $t\in [t_{1},\infty)$.
\end{proof}

\subsection{Proofs for pathwise stability and moment growth bounds}\label{se:5.2}

\begin{proof}[Proof of Proposition~\ref{pr:pathwise stability}]
Because $Y$ is random It{\^o} process with drift $\hat{\B}\in\mathscr{S}(\R^{m})$ and diffusion $\hat{\Sigma}\in\mathscr{S}(\R^{m\times d})$ defined via~\eqref{eq:specific drift and diffusion coefficients} for the choice $(\tilde{\B},\tilde{\Sigma}) = (\B,\Sigma)$, it follows from the estimates~\eqref{eq:stability estimates} that the assertion is a special case of Theorem~\ref{th:pathwise stability}.
\end{proof}

\begin{proof}[Proof of Corollary~\ref{co:pathwise stability}]
To show both claims simultaneously, we may argue as in the proof of Corollary~3.19 in~\cite{KalMeyPro21}. Namely, we take $\hat{t}_{1}\in [t_{1},\infty)$ and $\tilde{\delta} > 0$ such that $\gamma_{pq,\mathscr{P}}\leq 0$ a.e.~on $[\hat{t}_{1},\infty)$ and set $t_{n} := \hat{t}_{1} + \tilde{\delta}(n-1)$ for all $n\in\N$. Then
\begin{equation*}
\int_{0}^{\infty}\exp\bigg(\frac{\varepsilon}{pq}\int_{\hat{t}_{1}}^{\hat{t}_{1} + \tilde{\delta} t}\gamma_{pq,\mathscr{P}}(s)\,ds\bigg)\,dt \leq \int_{0}^{\infty}\exp\bigg(\frac{\varepsilon}{pq}\sum_{k=1}^{l}\hat{\lambda}_{k}\int_{\hat{t}_{1}}^{\hat{t}_{1} + \tilde{\delta}t}\alpha_{k}(s-s_{k})^{\alpha_{k-1}}\,ds\bigg)\,dt
\end{equation*}
for any given $\varepsilon > 0$. Since $\hat{\lambda}_{l} < 0$, the integral on the right-hand side is finite, see Lemma~5.1 in~\cite{KalMeyPro21}, for instance. Thus, the integral test for the convergence of series implies that $\sum_{n=1}^{\infty}\exp(\varepsilon(pq)^{-1}\int_{\hat{t}_{1}}^{t_{n}}\gamma_{pq,\mathscr{P}}(s)\,ds) < \infty$.

Consequently,~\eqref{con:9} follows as soon as $\tilde{\delta} < \hat{\delta}$ and Proposition~\ref{pr:pathwise stability} entails that the difference $Y$ of any two solutions $X$ and $\tilde{X}$ to~\eqref{eq:McKean-Vlasov} for which $\Theta(\cdot,P_{X},P_{\tilde{X}})$ is locally integrable satisfies
\begin{equation*}
\limsup_{t\uparrow\infty} \frac{\log(|Y_{t}|)}{t^{\alpha_{l}}} \leq \frac{1}{pq}\limsup_{n\uparrow\infty}\sum_{k=1}^{l}\hat{\lambda}_{k} \frac{(t_{n} - s_{k})^{\alpha_{k}} - (\hat{t}_{1} - s_{k})^{\alpha_{k}}}{t_{n}^{\alpha_{l}}} = \frac{\hat{\lambda}_{l}}{pq}\quad\text{a.s.,}
\end{equation*}
provided $E[|Y_{t_{0}}|^{pq}] < \infty$ or both $\B$ and $\Sigma$ are independent of $\mu\in\mathscr{P}$.
\end{proof}

\begin{proof}[Proof of Lemma~\ref{le:2nd moment growth estimate}]
By hypothesis, $X$ is a random It{\^o} process with drift $\hat{\B}\in\mathscr{S}(\R^{m})$ and diffusion $\hat{\Sigma}\in\mathscr{S}(\R^{m\times d})$ given by $\hat{\B}_{s}:=\B(X_{s},P_{X_{s}})$ and $\hat{\Sigma}_{s}:=\Sigma_{s}(X_{s},P_{X_{s}})$ so that
\begin{equation*}
2X'\hat{\B} + |\hat{\Sigma}|^{2} \leq \kappa + \upsilon\phi(|X|^{2}) + \chi\varphi(\theta^{2})
\end{equation*}
a.e.~on $[t_{0},\infty)$ a.s.~with the measurable function $\theta:=\vartheta(P_{X},\delta_{0})$. For this reason, the lemma is a special case of Proposition~\ref{pr:abstract second moment estimate}.
\end{proof}

\begin{proof}[Proof of Lemma~\ref{le:p-th moment growth estimate}]
We let $\hat{\B}\in\mathscr{S}(\R^{m})$ and $\hat{\Sigma}\in\mathscr{S}(\R^{m\times d})$ be given as in the proof of Lemma~\ref{le:2nd moment growth estimate} and readily see that
\begin{equation*}
X'\hat{\B}\leq \sum_{k=1}^{l}\kappa_{k}\,\null_{k}\upsilon |X|^{1+\alpha_{k}}\vartheta(P_{X},\delta_{0})^{\beta_{k}}\quad\text{and}\quad|\hat{\Sigma}|^{2}\leq \sum_{j,k=1}^{l}\hat{\kappa}_{j}\hat{\kappa}_{k}\,\null_{j}\hat{\upsilon}\,\null_{k}\hat{\upsilon}|X|^{\alpha_{k}}\vartheta(P_{X},\delta_{0})^{\beta_{k}}
\end{equation*}
a.e.~on $[t_{0},\infty)$ a.s. Hence, Theorem~\ref{th:stability moment estimate} yields all assertions.
\end{proof}

\subsection{Proofs for unique strong solutions}\label{se:5.3}

\begin{proof}[Proof of Proposition~\ref{pr:results on mu-SDE}]
(i) Because~\eqref{con:3} is satisfied when $(\B,\Sigma) = (b_{\mu},\sigma_{\mu})$, pathwise uniqueness for~\eqref{eq:mu-SDE} follows from Corollary~\ref{co:pathwise uniqueness}.

(ii) In essence, we may proceed as in the proof of Proposition 3.26 in~\cite{KalMeyPro21}. First, let $\xi$ be essentially bounded. Then Theorem 2.3 in~\cite{IkeWat89}[Chapter IV] yields a local weak solution $\tilde{X}$ to~\eqref{eq:mu-SDE}.

That is, $\tilde{X}$ is an $\R^{m}\cup\{\infty\}$-valued adapted continuous process on a filtered probability space $(\tilde{\Omega},\tilde{\mathscr{F}},(\tilde{\mathscr{F}}_{t})_{t\in\R_{+}},\tilde{P})$ on which there is an $(\tilde{\mathscr{F}}_{t})_{t\in\R_{+}}$-Brownian motion $\tilde{W}$ such that the usual and the following conditions hold:
\begin{enumerate}[(1)]
\item If $(s,\omega)\in [t_{0},\infty)\times\tilde{\Omega}$ satisfies $\tilde{X}_{s}(\omega) = \infty$, then $\tilde{X}_{t}(\omega)=\infty$ for any $t\in [s,\infty)$.
\item $\mathscr{L}(\tilde{X}_{t_{0}}) = \mathscr{L}(\xi)$ and for the supremum $\tau$ of the sequence $(\tau_{n})_{n\in\N}$ of stopping times defined by $\tau_{n}:=\inf\{t\in [t_{0},\infty)\,|\,|\tilde{X}_{t}|\geq n\}$ we have $\tau > t_{0}$ a.s. 
\item $\tilde{X}^{\tau_{n}}$ solves~\eqref{eq:McKean-Vlasov} on $(\tilde{\Omega},\tilde{\mathscr{F}},(\tilde{\mathscr{F}}_{t})_{t\in\R_{+}},\tilde{P})$ relative to $\tilde{W}$ when $\B$ and $\Sigma$ are replaced by the admissible maps $b_{\mu}\mathbbm{1}_{\{\tau_{n} > \cdot\}}$ and $\sigma_{\mu}\mathbbm{1}_{\{\tau_{n} > \cdot\}}$, respectively, for every $n\in\N$.
\end{enumerate}

Clearly,~\eqref{con:det.1} ensures that~\eqref{con:11} holds for $(\B,\Sigma)=(b_{\mu},\sigma_{\mu})$, as $\kappa_{\mu}:=\kappa + \chi\varphi(\vartheta(\mu,\delta_{0})^{2})$ is locally integrable. Consequently, Lemma~\ref{le:2nd moment growth estimate} and Fatou's lemma assert that
\begin{equation*}
\tilde{E}\big[|\tilde{X}_{t}^{\tau}|^{2}\big] \leq\liminf_{n\uparrow\infty}\tilde{E}\big[|\tilde{X}_{t}^{\tau_{n}}|^{2}\big]\leq \Psi_{\phi}\bigg(E\big[|\xi|^{2}\big] + \int_{t_{0}}^{t}\kappa_{\mu}(s)\,ds, \int_{t_{0}}^{t}\upsilon(s)\,ds\bigg)
\end{equation*}
for each $t\in [t_{0},\infty)$, which implies that $\tau=\infty$ and $\tilde{X}\in\R^{m}$ $\tilde{P}$-a.s. Thus, $\tilde{X}\mathbbm{1}_{\{\tau = \infty\}}$ serves as weak solution to~\eqref{eq:deterministic McKean-Vlasov} in the standard sense.

In particular, this derivation applies to the case when $\xi$ is deterministic. Therefore, Remark~2.1 in~\cite{IkeWat89}[Chapter IV] entails that there is a weak solution $X$ to~\eqref{eq:deterministic McKean-Vlasov} with $X_{t_{0}}=\xi$ a.s., regardless of whether $\xi$ is essentially bounded.

(iii) By what we have just shown, pathwise uniqueness for~\eqref{eq:mu-SDE} holds and there exists a weak solution for any $\R^{m}$-valued $\mathscr{F}_{t_{0}}$-measurable random vector used as initial condition. Hence, Theorem 1.1 in~\cite{IkeWat89}[Chapter IV] entails the assertion.
\end{proof}

\begin{proof}[Proof of Theorem~\ref{th:strong existence}]
(i) and (ii) Pathwise uniqueness with respect to $\Theta$ follows from Proposition~\ref{pr:specific stability moment estimate}, as the underlying filtered probability space and Brownian motion were arbitrarily chosen. In particular, there exists at most a unique solution $X$ to~\eqref{eq:deterministic McKean-Vlasov} such that $X_{t_{0}} = \xi$ a.s.~and $E[|X|^{p}]$ is locally bounded.

Next, for any $\mu\in B_{b,loc}(\mathscr{P})$ Proposition~\ref{pr:results on mu-SDE} gives a unique strong solution $X^{\xi,\mu}$ to~\eqref{eq:mu-SDE} with $X_{t_{0}}^{\xi,\mu} = \xi$ a.s., and $E[|X^{\xi,\mu}|^{p}]$ is locally bounded, by Lemmas~\ref{le:2nd moment growth estimate} and~\ref{le:p-th moment growth estimate}. We note that this process is also a strong solution to~\eqref{eq:deterministic McKean-Vlasov} if $\mu$ is a fixed-point of the operator
\begin{equation*}
\Psi:B_{b,loc}(\mathscr{P})\rightarrow B_{b,loc}(\mathscr{P}_{p}(\R^{m})),\quad \Psi(\nu)(t) :=\mathscr{L}(X_{t}^{\xi,\nu}).
\end{equation*}
For given $\mu,\tilde{\mu}\in B_{b,loc}(\mathscr{P})$ we immediately check that~\eqref{con:4} is valid when $(\B,\Sigma)$ and $(\tilde{\B},\tilde{\Sigma})$ are replaced by $(b_{\mu},\sigma_{\mu})$ and $(b_{\tilde{\mu}},\sigma_{\tilde{\mu}})$, respectively. Hence, Proposition~\ref{pr:specific stability moment estimate} implies that
\begin{equation}\label{eq:auxiliary error estimate}
\vartheta_{p}(\Psi(\mu),\Psi(\tilde{\mu}))(t)^{p} \leq E\big[|X_{t}^{\xi,\mu} - X_{t}^{\xi,\tilde{\mu}}|^{p}\big]\leq \int_{t_{0}}^{t}e^{\int_{s}^{t}\gamma_{p}(\tilde{s})\,d\tilde{s}}\delta(s)\vartheta(\mu,\tilde{\mu})(s)^{p}\,ds
\end{equation}
for every $t\in [t_{0},\infty)$. In particular, this shows that there is at most a unique fixed-point of $\Psi$, due to Gronwall's inequality. 

Further, because $B_{b,loc}(\mathscr{P}_{p}(\R^{m}))$ is completely metrisable, the fixed-point theorem for time evolution operators in~\cite{Kal22} yields the existence of a fixed-point and the error estimate~\eqref{eq:error estimate}. Namely, it follows inductively that
\begin{equation}\label{eq:auxiliary error estimate 2}
\sup_{s\in [t_{0},t]}\vartheta_{p}(\mu_{m},\nu_{n})(s) \leq \Delta(t)\sum_{i=n}^{m-1}\bigg(\frac{c_{p,\mathscr{P}}^{i}}{i!}\bigg)^{\frac{1}{p}}\bigg(\int_{t_{0}}^{t}e^{\int_{s}^{t}\gamma_{p}^{+}(\tilde{s})\,d\tilde{s}}\delta(s)\,ds\bigg)^{\frac{i}{p}}
\end{equation}
for any $m,n\in\N$ with $m > n$ and each $t\in [t_{0},\infty)$. Hence, $(\mu_{n})_{n\in\N}$ is a Cauchy sequence in $B_{b,loc}(\mathscr{P}_{p}(\R^{m}))$ and from~\eqref{eq:auxiliary error estimate} we infer that its limit $\mu$ must be a fixed-point of $\Psi$. Now we may take the limit $m\uparrow\infty$ in~\eqref{eq:auxiliary error estimate 2} to get the desired bound~\eqref{eq:error estimate}.

(iii) The set $M_{p}$ is closed and convex, since the estimate in~\eqref{eq:growth estimate} does not depend on $\mu\in B_{b,loc}(\mathscr{P}_{p}(\R^{m}))$. An application of Lemma~\ref{le:p-th moment growth estimate} entails that
\begin{equation}\label{eq:p-th moment measure estimate}
\vartheta_{p}(\Psi(\mu)(t),\delta_{0})^{p} \leq e^{\int_{t_{0}}^{t}f_{p,\mathscr{P},0}(s)\,ds}E\big[|\xi|^{p}\big] + \int_{t_{0}}^{t}e^{\int_{s}^{t}f_{p,\mathscr{P},0}(\tilde{s})\,d\tilde{s}}g_{p,\mu}(s)\,ds
\end{equation}
for every $\mu\in B_{b,loc}(\mathscr{P})$ and all $t\in [t_{0},\infty)$ with the two measurable and locally integrable functions
\begin{equation*}
f_{p,\mathscr{P},0}:= p \sum_{\substack{k=1,\\ \alpha_{k} = 1}}^{l}\upsilon_{k} + \sum_{\substack{k=1,\\ \alpha_{k} < 1}}^{l} (p-1 + \alpha_{k})c_{p,\mathscr{P}}^{\beta_{k}}\upsilon_{k}^{+} + c_{p}\sum_{j,k=1}^{l}(p-2+\alpha_{j}+\alpha_{k})c_{p,\mathscr{P}}^{\beta_{j} + \beta_{k}}\hat{\upsilon}_{j}\hat{\upsilon}_{k}
\end{equation*}
and
\begin{align*}
g_{p,\mu} &:= \sum_{\substack{k=1,\\ \alpha_{k} < 1}}^{l} (1 - \alpha_{k})\vartheta_{p}(\mu,\delta_{0})^{\frac{\beta_{k}}{1-\alpha_{k}}p}c_{p,\mathscr{P}}^{\beta_{k}}\upsilon_{k}^{+}\\
&\quad + c_{p}\sum_{\substack{j,k=1,\\ \alpha_{j} < 1 \text{ or } \alpha_{k} < 1}}^{l}(2 - \alpha_{j} - \alpha_{k})\vartheta_{p}(\mu,\delta_{0})^{\frac{\beta_{j}+\beta_{k}}{2-\alpha_{j}-\alpha_{k}}p}c_{p,\mathscr{P}}^{\beta_{j} + \beta_{k}}\hat{\upsilon}_{j}\hat{\upsilon}_{k}.
\end{align*}
Thereby, we used the fact that $(b_{\mu},\sigma_{\mu})$ satisfies~\eqref{con:12} for the choice $\beta = 0$. Next, Young's inequality gives us that
\begin{equation*}
(1-\alpha_{k})\vartheta_{p}(\mu,\delta_{0})^{\frac{\beta_{k}}{1-\alpha_{k}}p} \leq 1 - \alpha_{k} - \beta_{k} + \beta_{k}\vartheta_{p}(\mu,\delta_{0})^{p}
\end{equation*}
for every $k\in\{1,\dots,l\}$ with $\alpha_{k} < 1$ and
\begin{equation*}
(2 - \alpha_{j} - \alpha_{k})\vartheta_{p}(\mu,\delta_{0})^{\frac{\beta_{j} + \beta_{k}}{2-\alpha_{j}-\alpha_{k}}p} \leq 2 - \alpha_{j} - \alpha_{k} - \beta_{j} - \beta_{k} + (\beta_{j}+\beta_{k})\vartheta_{p}(\mu,\delta_{0})^{p}
\end{equation*}
for any $j,k\in\{1,\dots,l\}$ with $\alpha_{j} < 1$ or $\alpha_{k} < 1$. From these estimates we infer that $f_{p,\mathscr{P},1} := f_{p,\mathscr{P}} - f_{p,\mathscr{P},0}$ satisfies
\begin{equation*}
g_{p,\mu} \leq g_{p,\mathscr{P}} + f_{p,\mathscr{P},1}\vartheta_{p}(\mu,\delta_{0})^{p}.
\end{equation*}
So, the inequality~\eqref{eq:p-th moment measure estimate}, the Fundamental Theorem of Calculus for Lebesgue-Stieltjes integrals and Fubini's theorem show that $\Psi$ maps $M_{p}$ into itself, which implies the claim.
\end{proof}

%-------------------------------------------------------------------
% Bibliography
%-------------------------------------------------------------------
\let\OLDthebibliography\thebibliography
\renewcommand\thebibliography[1]{
  \OLDthebibliography{#1}
  \setlength{\parskip}{0pt}
  \setlength{\itemsep}{2pt}}

%-------------------------------------------------------------------


\begin{thebibliography}{10}

\bibitem{BauMey19-1}
M.~Bauer and T.~Meyer-Brandis.
\newblock Existence and regularity of solutions to multi-dimensional mean-field
  stochastic differential equations with irregular drift.
\newblock {\em arXiv preprint arXiv:1912.05932}, 2019.

\bibitem{BauMey19-2}
M.~Bauer and T.~Meyer-Brandis.
\newblock Mc{K}ean-{V}lasov equations on infinite-dimensional {H}ilbert spaces
  with irregular drift and additive fractional noise.
\newblock {\em arXiv preprint arXiv:1912.07427}, 2019.

\bibitem{BauMeyPro18}
M.~Bauer, T.~Meyer-Brandis, and F.~Proske.
\newblock Strong solutions of mean-field stochastic differential equations with
  irregular drift.
\newblock {\em Electron. J. Probab.}, 23:Paper No. 132, 35, 2018.

\bibitem{BucDjeLiPen09}
R.~Buckdahn, B.~Djehiche, J.~Li, and S.~Peng.
\newblock Mean-field backward stochastic differential equations: a limit
  approach.
\newblock {\em Ann. Probab.}, 37(4):1524--1565, 2009.

\bibitem{BucLiPen09}
R.~Buckdahn, J.~Li, and S.~Peng.
\newblock Mean-field backward stochastic differential equations and related
  partial differential equations.
\newblock {\em Stochastic Process. Appl.}, 119(10):3133--3154, 2009.

\bibitem{BucLiPenRai17}
R.~Buckdahn, J.~Li, S.~Peng, and C.~Rainer.
\newblock Mean-field stochastic differential equations and associated {PDE}s.
\newblock {\em Ann. Probab.}, 45(2):824--878, 2017.

\bibitem{CarDel13}
R.~Carmona and F.~Delarue.
\newblock Probabilistic analysis of mean-field games.
\newblock {\em SIAM J. Control Optim.}, 51(4):2705--2734, 2013.

\bibitem{CarDel14}
R.~Carmona and F.~Delarue.
\newblock The master equation for large population equilibriums.
\newblock In {\em Stochastic analysis and applications 2014}, volume 100 of
  {\em Springer Proc. Math. Stat.}, pages 77--128. Springer, Cham, 2014.

\bibitem{CarDel18I}
R.~Carmona and F.~Delarue.
\newblock {\em Probabilistic theory of mean field games with applications.
  {I}}, volume~83 of {\em Probability Theory and Stochastic Modelling}.
\newblock Springer, Cham, 2018.
\newblock Mean field FBSDEs, control, and games.

\bibitem{CarDel18II}
R.~Carmona and F.~Delarue.
\newblock {\em Probabilistic theory of mean field games with applications.
  {II}}, volume~84 of {\em Probability Theory and Stochastic Modelling}.
\newblock Springer, Cham, 2018.
\newblock Mean field games with common noise and master equations.

\bibitem{CarDelLac16}
R.~Carmona, F.~Delarue, and A.~Lachapelle.
\newblock Control of {M}c{K}ean-{V}lasov dynamics versus mean field games.
\newblock {\em Math. Financ. Econ.}, 7(2):131--166, 2013.

\bibitem{CarFouMouSun18}
R.~Carmona, J.-P. Fouque, S.~M. Mousavi, and L.-H. Sun.
\newblock Systemic risk and stochastic games with delay.
\newblock {\em J. Optim. Theory Appl.}, 179(2):366--399, 2018.

\bibitem{CarFouSun15}
R.~Carmona, J.-P. Fouque, and L.-H. Sun.
\newblock Mean field games and systemic risk.
\newblock {\em Commun. Math. Sci.}, 13(4):911--933, 2015.

\bibitem{Chi94}
T.~S. Chiang.
\newblock Mc{K}ean-{V}lasov equations with discontinuous coefficients.
\newblock {\em Soochow J. Math.}, 20(4):507--526, 1994.

\bibitem{ConIye08}
P.~Constantin and G.~Iyer.
\newblock A stochastic {L}agrangian representation of the three-dimensional
  incompressible {N}avier-{S}tokes equations.
\newblock {\em Comm. Pure Appl. Math.}, 61(3):330--345, 2008.

\bibitem{FouIch13}
J.-P. Fouque and T.~Ichiba.
\newblock Stability in a model of interbank lending.
\newblock {\em SIAM J. Financial Math.}, 4(1):784--803, 2013.

\bibitem{GarPapYan13}
J.~Garnier, G.~Papanicolaou, and T.-W. Yang.
\newblock Large deviations for a mean field model of systemic risk.
\newblock {\em SIAM J. Financial Math.}, 4(1):151--184, 2013.

\bibitem{IkeWat89}
N.~Ikeda and S.~Watanabe.
\newblock {\em Stochastic differential equations and diffusion processes},
  volume~24 of {\em North-Holland Mathematical Library}.
\newblock North-Holland Publishing Co., Amsterdam; Kodansha, Ltd., Tokyo,
  second edition, 1989.

\bibitem{Jou97}
B.~Jourdain.
\newblock Diffusions with a nonlinear irregular drift coefficient and
  probabilistic interpretation of generalized {B}urgers' equations.
\newblock {\em ESAIM Probab. Statist.}, 1:339--355, 1997.

\bibitem{JouMelWoy08}
B.~Jourdain, S.~M\'{e}l\'{e}ard, and W.~A. Woyczynski.
\newblock Nonlinear {SDE}s driven by {L}\'{e}vy processes and related {PDE}s.
\newblock {\em ALEA Lat. Am. J. Probab. Math. Stat.}, 4:1--29, 2008.

\bibitem{Kac56}
M.~Kac.
\newblock Foundations of kinetic theory.
\newblock In {\em Proceedings of the {T}hird {B}erkeley {S}ymposium on
  {M}athematical {S}tatistics and {P}robability, 1954--1955, vol. {III}}, pages
  171--197. University of California Press, Berkeley and Los Angeles, Calif.,
  1956.

\bibitem{Kal22}
A.~Kalinin.
\newblock Fixed-point operators and {V}olterra integral equations with
  path-dependency.
\newblock {\em preprint}, 2022.

\bibitem{KalMeyPro21}
A.~Kalinin, T.~Meyer-Brandis, and F.~Proske.
\newblock Stability, uniqueness and existence of solutions to
  {M}c{K}ean-{V}lasov {SDE}s: a multidimensional {Y}amada-{W}atanabe approach.
\newblock {\em arXiv preprint arXiv:2107.07838}, 2021.

\bibitem{KleKluRei15}
O.~Kley, C.~Kl{\"u}ppelberg, and L.~Reichel.
\newblock Systemic risk through contagion in a core-periphery structured
  banking network.
\newblock {\em Banach Center Publications}, 1(104):133--149, 2015.

\bibitem{LasLio07}
J.-M. Lasry and P.-L. Lions.
\newblock Mean field games.
\newblock {\em Jpn. J. Math.}, 2(1):229--260, 2007.

\bibitem{Ler34}
J.~Leray.
\newblock Sur le mouvement d'un liquide visqueux emplissant l'espace.
\newblock {\em Acta Math.}, 63(1):193--248, 1934.

\bibitem{LiMin16}
J.~Li and H.~Min.
\newblock Weak solutions of mean-field stochastic differential equations and
  application to zero-sum stochastic differential games.
\newblock {\em SIAM Journal on Control and Optimization}, 54(3):1826--1858,
  2016.

\bibitem{Mao08}
X.~Mao.
\newblock {\em Stochastic differential equations and applications}.
\newblock Horwood Publishing Limited, Chichester, second edition, 2008.

\bibitem{McK66}
H.~P. McKean, Jr.
\newblock A class of {M}arkov processes associated with nonlinear parabolic
  equations.
\newblock {\em Proc. Nat. Acad. Sci. U.S.A.}, 56:1907--1911, 1966.

\bibitem{MisVer20}
Y.~Mishura and A.~Veretennikov.
\newblock Existence and uniqueness theorems for solutions of
  {M}c{K}ean--{V}lasov stochastic equations.
\newblock {\em Theory of Probability and Mathematical Statistics}, 103:59--101,
  2020.

\bibitem{RoeZha21}
M.~R{\"o}ckner and G.~Zhao.
\newblock S{DE}s with critical time dependent drifts: strong solutions.
\newblock {\em arXiv preprint arXiv:2103.05803}, 2021.

\bibitem{Vla68}
A.~A. Vlasov.
\newblock The vibrational properties of an electron gas.
\newblock {\em Soviet Physics Uspekhi}, 10(6):721, 1968.

\end{thebibliography}
\end{document}